\documentclass[final]{siamltex}
\usepackage{mathrsfs}
\usepackage{amsmath}
  \usepackage{paralist}
  \usepackage{longtable}
   \usepackage{multirow}
    \usepackage{rotating}
  \usepackage{graphics} 
  \usepackage{epsfig} 
\usepackage{graphicx}  \usepackage{epstopdf}
 \usepackage[colorlinks=true]{hyperref}
\makeatother

\normalsize
\usepackage{amsfonts}
\usepackage{arydshln}
\numberwithin{equation}{section}

\newcommand{\D}{\displaystyle}
\newcommand{\lbl}[1]{\label{#1}}
\newcommand\ee{\end{equation}}
\newcommand\bes{\begin{eqnarray}}
\newcommand\ees{\end{eqnarray}}
\newcommand\bess{\begin{eqnarray*}}
\newcommand\eess{\end{eqnarray*}}

\newcommand\nm{\nonumber}

\title{ Pointwise estimates for bipolar compressible Navier-Stokes-Poisson system in dimension three}

\author{ZHIGANG WU\footnotemark[2]
\ \ WEIKE WANG\footnotemark[1]\thanks{Corresponding author.}
        }

\begin{document}

\maketitle
\renewcommand{\thefootnote}{\fnsymbol{footnote}}
 \footnotetext[2]{Department of Applied Mathematics, Donghua University,
Shanghai, P.R. China,  }
\footnotetext[1]{Department of Mathematics, Shanghai Jiao Tong University,
Shanghai, P.R. China, ({\tt wkwang@sjtu.edu.cn(W.K. WANG)}). }
\renewcommand{\thefootnote}{\arabic{footnote}}

\date{}
\maketitle

{{\bf  Abstract:}} The Cauchy problem of the bipolar
Navier-Stokes-Poisson system (\ref{1.1}) in dimension three is
considered. We obtain the pointwise estimates of the time-asymptotic shape of the solution, which exhibit generalized Huygens' principle as the  Navier-Stokes system. This phenomenon is the the most important difference from the unipolar Navier-Stokes-Poisson system.  Due to non-conservative structure of the system (\ref{1.1}) and interplay of two carriers which counteracts the influence of electric field (a nonlocal term), some new observations are essential for the proof. We make full use of the conservative structure of the system for the total density and total momentum, and the mechanism of the linearized unipolar Navier-Stokes-Poisson system together with the special form of the nonlinear terms in the system for the difference of densities and the difference of momentums. Lastly, as a byproduct, we extend the usual $L^2(\mathbb{R}^3)$-decay rate to $L^p(\mathbb{R}^3)$-decay rate with $p>1$ and also improve former decay rates in part.



\section{Introduction}
\ \

\quad The bipolar
Navier-Stokes-Poisson system has been used to simulate the transport of charged particles under the influence of
electrostatic force governed by the self-consistent Poisson equation, cf. \cite{Degond}. In this paper, we are mainly concerned with the Cauchy problem of the bipolar  Navier-Stokes-Poisson system (BNSP) in $n$ dimensions:
\bes
\left\{\begin{array}{l}
\partial_t\rho_1+{\rm div} J_1=0, \\
\partial_tJ_1+{\rm div}(\frac{J_1\otimes J_1}{\rho_1})+\nabla P_1(\rho_1)
=\mu_1 \Delta(\frac{J_1}{\rho_1}) +\mu_2\nabla{\rm div}(\frac{J_1}{\rho_1})+\rho_1\nabla\phi, \\
\partial_t\rho_2+{\rm div} J_2=0, \\
\partial_tJ_2+{\rm div}(\frac{J_2\otimes J_2}{\rho_2})+\nabla P_2(\rho_2)
=\mu_1 \Delta(\frac{J_2}{\rho_2}) +\mu_2\nabla{\rm div}(\frac{J_2}{\rho_2})-\rho_2\nabla\phi, \\
\Delta\phi=\rho_1-\rho_2, \lim\limits_{|x|\rightarrow\infty}\phi(x,t)\rightarrow 0,
\end{array}
        \right.
        \lbl{1.1}
\ees
with initial data
\bes
(\rho_1,\rho_2,J_1,J_2,\nabla\phi)(x,0)=(\rho_{1,0},\rho_{2,0},J_{1,0},J_{2,0},\nabla\phi_0)(x), x\in\mathbb{R}^n.
\lbl{1.2}
\ees
Here $\rho_i(x,t),J_i(x,t)=\rho_iu_i,\phi(x,t)$ and $P_i\!=\! P_i(\rho_i)$
represent the fluid density, momentum, self-consistent electric potential and
pressure. The viscosity coefficients satisfy the usual physical conditions $\mu_1>0,\
\mu_1+\frac{2}{n}\mu_2>0$. $\bar{\rho}>0$ denotes the prescribed density of
positive charged background ions, and in this paper is taken as a
positive constant. In the present paper, we mainly consider the three dimensions case.

Now we mainly review some previous works on the Cauchy problem for some related models. For the compressible Navier-Stokes system, a lot of works have been done on the existence, stability and $L^p$-decay rates with $p\geq2$ for either isentropic or non-isentropic
(heat-conductive) cases, cf. \cite{D1,D2,HJ, HM, K, KM,KN,KS,KW,L,LX,MN,MN1,OK} in various settings by using (weighted) energy method together with spectrum analysis.  On the other hand, Liu and Zeng \cite{LZ2} first studied the pointwise estimates of solution to general hyperbolic-parabolic systems in one space dimension by using the method of Green's function. Hoff and Zumbrun \cite{HZ0} investigated the $L^p(\mathbb{R}^n)\ (p\geq1)$ estimates for the isentropic
Navier-Stokes system in multi-dimensions.
Then they in \cite{HZ} deduced a detailed, pointwise
description of the Green's function for the related linearized
Navier-Stokes system with artificial viscosity. Later, Liu and Wang
\cite{LW} considered the pointwise estimates of the solution in odd
dimensions and showed the generalized Huygens' principle by using real analysis method, which was reconsidered in \cite{LS} by using the complex analysis method for the Green's function recently. In particular, they in \cite{LS,LW} showed the solution behaves as
\bes
(1+t)^{-\frac{n}{2}}\bigg(1+\frac{|x|^2}{1+t}\bigg)^{-\frac{n}{2}}+(1+t)^{-\frac{3n-1}{4}}\bigg(1+\frac{(|x|-ct)^2}{1+t}\bigg)^{-\frac{n}{2}},\lbl{1.2(0)}
\ees
where the first profile is called as the diffusion wave(D-wave) and the second profile is called as the generalized Huygens' wave(H-wave). Wang $et\ al.$ \cite{Wang2,Wu3}  also discussed the pointwise estimates of the solution for the damped Euler system when initial data is a small
perturbation of the constant state and showed the solution behaves as the diffusion wave since the long wave of the Green's function does not contain the wave operator. All of the results above showed the different asymptotic profiles
in the pointwise sense for different systems. Indeed, the pointwise estimates of the solution play an important role in the description of the partial differential equations, since it can give explicit expressions of the time-asymptotic behavior of the solution. Moreover, one can get the global existence and optimal $L^p$-estimates of the solution directly from the pointwise estimates of
the solution. Here, we give some comments on the diffusion wave(D-wave) and the generalized Huygens' wave(H-wave) in (\ref{1.2(0)}) for the convenience of readers. As we know, for the $L^p(\mathbb{R}^n)$-estimate of these two waves, when $p<2$ the H-wave is the dominated part and when $p>2$ the D-wave is the dominated part, and $p=2$ is the critical case and the $L^2$-decay rate of these two waves is the same as the heat kernel. In other words, the usual $L^2$-estimates for some hyperbolic-parabolic coupled systems indeed conceal the hyperbolic characteristic of the solution. Besides, there are also other results in \cite{LY1,LY2,LY3,LY4,LY5,LZ1,Y1,Z1} on the related models for the $L^1$-estimate, wave propagation pattern around shock wave by using the method of Green's function.

Nowadays, for unipolar Navier-Stokes-Poisson system(NSP), there are also some results on the decay rate of the Cauchy problem when the initial data $(\rho_0,u_0)$ is a small perturbation of the constant state $(\bar{\rho},0)$. Li $et\ al.$ \cite{Li1} obtained the global existence and large time behavior of classical solution by spectrum method and energy method. They mainly  found that the density of NSP system
 converges to its equilibrium state at the same $L^2$-decay rate $(1+t)^{-\frac{3}{4}}$ as the  Navier-Stokes system, but the momentum of NSP system decays at $L^2$-decay rate $(1+t)^{-\frac{1}{4}}$, and then they extended the similar result to the non-isentropic case \cite{Zhang}. On the other hand, in \cite{Wang,Wu1}, we investigated the pointwise estimates of the solution and showed the pointwise profile of the solution contains the D-wave but does not contain the H-wave, which is different from (\ref{1.2(0)}) for the  Navier-Stokes system in \cite{LS,LW}.  In fact,  the NSP system is a hyperbolic-parabolic system with a non-local term arising from the electric field $\nabla\phi$. The symbol of this nonlocal term is singular in the long wave of the Green's function. It destroys the acoustic wave propagation, at the same time, this non-local term also destroys the time-decay rate and the regularity (integrability) on the space variable $x$ for the momentum (or the velocity $u$). Recently, Wang \cite{WangYJ} improved the $L^2$-decay rate and Wu $et\ al.$ \cite{Wu2} improved the pointwise estimate under an additional assumption that the initial data of the electric potential term $\nabla\phi$ is in $L^1(\mathbb{R}^3)$, respectively. Li $et\ al.$ \cite{Li4} considered the similar problem in the negative Besov space to improve the $L^2$-decay rate.

Compared with NSP system, there are few results for the Cauchy problem of BNSP system (\ref{1.1})-(\ref{1.2}) due to the interplay of two carriers which counteracts the influence of electric field. Li $et\ al.$ \cite{Li3} first gave the global existence and the $L^2$-decay rate of the classical solution as in \cite{Li1} when the initial data is small in $H^4\cap L^1(\mathbb{R}^3)$ by changing the system (\ref{1.1}) to (\ref{2.4}). Later, Zou \cite{Zou} improved the $L^2$-decay rate under the assumption on the initial data in the negative Besov space. Recently, Wang $et\ al.$ \cite{WX} directly considered the system (\ref{1.1}) and also obtained the $L^2$-decay rate of the solution as the Navier-Stokes system by using long wave and short wave decomposition method together with energy method when the initial data is small in $H^4\cap L^1(\mathbb{R}^3)$ together with an additional assumption of the antiderivative of two densities. Both the assumptions on the initial data in the negative Besov space and the antiderivative of two densities in \cite{WX, Zou} are used to lift the $L^2$-decay rate of two momentums to $(1+t)^{-\frac{3}{4}}$ as that of the Navier-Stokes system.

The motivation of this paper is as follows. Firstly, from the analysis on the Green's function of BNSP system (\ref{1.1}), we find that there exists the wave operator in the long wave of the Green's function, therefore we hope that the pointwise estimates of the solution to the Cauchy problem (\ref{1.1})-(\ref{1.2}) exhibit the generalized Huygens' principle as in \cite{LS,LW}, which is an absolutely different phenomenon from the unipolar case in \cite{Wang, Wu2}. However, comparing with the Navier-Stokes system in \cite{LS,LW} and NSP system in \cite{Wang,Wu1}, the non-conservative structure of the system (\ref{1.1}) and the presence of the H-wave and the nonlocal electric field $\nabla\phi$ are all coming together, which coupled with the nonlinear interplay of two carriers through the electric field ultimately bring us more new difficulties for nonlinear estimates. The key observations include that the optimum matching relation between the H-wave and conservative structure for the total density and total momentum, besides the delicate matching relation between the absence of the H-wave in linearized NSP system and the special form of the corresponding nonlinear terms. Lastly, with the pointwise estimates in hand, we can immediately obtain the $L^p(p>1)$-estimate of the solution, which is a generalization for the usual $L^p(p\geq2)$-decay rate of the solution given in \cite{Li1,WX,Zou}.

We make a brief interpretation for the main idea of the proof. Due to the slower $L^p$-decay rate of the H-wave when $p<2$, the authors in \cite{LS,LW} rely heavily on the conservative structure of the Navier-Stokes system on the density and momentum, which allows to ``borrow" the first order derivative from the nonlinear terms to deduce a H-wave above when estimating the convolution between the Green's function and nonlinear terms. However, the system (\ref{1.1}) has not the conservative structure because of the presence of the terms $\rho_1\nabla\phi$ and $\rho_2\nabla\phi$ in the two momentum equations. Fortunately, by rewriting (\ref{1.1}) into a new system (\ref{2.4}) with the new variables $(\rho_1+\rho_2,J_1+J_2,\rho_1-\rho_2,J_1-J_2)$, and from the Poisson equation $\Delta\phi=\rho_1-\rho_2=:n_2$, we find that the nonlinear term $n_2\nabla\phi=\Delta\phi\nabla\phi={\rm div}(\nabla\phi\otimes\nabla\phi-\frac{1}{2}|\nabla\phi|^2I_{3\times3})$. Then the system $(\ref{2.4})_{1,2}$ is conservative, which is called as a Navier-Stokes system on the total density $n_1=\rho_1+\rho_2$ and the total momentum $w_1=J_1+J_2$. Hence, it is possible to help us to obtain the H-wave as in \cite{LS,LW}. Nevertheless, to utilize this conservative structure in (\ref{2.4})$_{1,2}$, we need the pointwise estimate of the electronic potential $\nabla\phi$, since there exists the term $n_2\nabla\phi$ in the equation $(\ref{2.4})_{2}$.

By the relation $\nabla\phi=\frac{\nabla}{-\Delta}n_2$, we know that the pointwise estimate of $\nabla\phi$ can only be derived from the pointwise estimate of $n_2$ in the system $(\ref{2.4})_{3,4,5}$ with the variables $n_2=\rho_1-\rho_2$ and $w_2=J_1-J_2$ . Thus, in the first step we should consider the NSP system $(\ref{2.4})_{3,4,5}$, which has still not the conservative structure because of the term $(\rho_1+\rho_2)\nabla\phi$ in $(\ref{2.4})_4$. As showed in \cite{Wang, Wu2}, because of the presence of the electric field $\nabla\phi$, the long wave of the Green's function of NSP only contains the D-wave and does not contain the H-wave, then one can deduce the pointwise estimates as the D-wave without the help from the conservative structure of the system. Based on the mechanism of linearized NSP system, we hope the solution $(n_2,w_2)$ of the system $(\ref{2.4})_{3,4,5}$ also behaves like the D-wave,  though there exists both the D-wave and the H-wave arising from the highly coupled unknown variables of the nonlinear terms $F_2(n_1,w_1,n_2,w_2)$ in (\ref{2.4})$_4$. Otherwise, once the pointwise estimate of $n_2=\rho_1-\rho_2$ contains a H-wave, it will generate a new H-wave with ``worse" decay rate for the pointwise estimate of $\nabla\phi$ from the relation between $n_2$ and $\nabla\phi$. Thus, it is crucial for us to prevent the H-wave from happening in the pointwise estimate of $n_2$. On the other hand, when using the system $(\ref{2.4})_{3,4,5}$ to get the pointwise estimate for $n_2$ and $\nabla\phi$, one will encounter the nonlinear term $n_1\nabla\phi$ in the momentum equation (\ref{2.4})$_4$. Then it need the pointwise estimate of $\nabla\phi$ in turn when deducing the pointwise estimate of $n_2$. What's worse, there is a factor $\varepsilon$ in Proposition \ref{l 53} for the pointwise estimate of $\nabla\phi$, that is,  the nonlocal operator $\frac{\nabla}{-\Delta}$ acting on a D-wave for $n_2$ will lead to  a D-wave with ``bad" decay rate both on the time and the space for $\nabla\phi$. Due to this mutual mechanism between $n_2$ and $\nabla\phi$,  and the presence of the factor $\varepsilon$, one will go into an endless loop when making the nonlinear estimates and ultimately cannot close the ansatz (\ref{5.20}). Hence, it is likewise crucial for us to avoid the presence of the term $\nabla\phi$ in the nonlinear terms of the system on the variable $n_2$ as far as possible.

To overcome the difficulties from the system $(\ref{2.4})_{3,4,5}$, we mainly rely on the good form of nonlinear term $F_2(n_1,w_1,n_2,w_2)$. In fact, after a careful computation we find that each nonlinear term in $F_2(n_1,w_1,n_2,w_2)$ contains a variable $n_2$ or $w_2$, that is, there is not the interaction between the H-wave and the H-wave in $F_2(n_1,w_1,n_2,w_2)$. Then, after a careful computation, we can prove the pointwise asymptotic shape of $(n_2,w_2)$ only contains the D-wave. On the other hand, to avoid dealing with the nonlinear term $n_1\nabla\phi$ in the momentum equation (\ref{2.4})$_4$, we consider a new system on the variables $\rho_1-\rho_2=n_2$ and $u_1-u_2=:v_2$ in (\ref{5.18}), where there is actually not the variable $\nabla\phi$ in the nonlinear term $F_4(n_1,v_1,n_2,v_2)$. That is, when deducing the pointwise estimate of $n_2$, we need not the pointwise estimate of $\nabla\phi$, which is the key difference from using the variable $(n_2,w_2)$. On the other side, we have to overcome the difficulty from the nonlocal operators in the Green's function of the system $(\ref{2.4})_{3,4,5}$. Noticing the following proposition in Fourier space of long wave part for the Green's function for linearized NSP system $(2.4)_{3,4,5}$:
\bess
\hat{\mathbb{G}}(\xi,t)=\left(
\begin{array}{cc}
\hat{\mathbb{G}}_{11} & \hat{\mathbb{G}}_{12}  \\
\hat{\mathbb{G}}_{21} & \hat{\mathbb{G}}_{22} \\
\end{array}
\right)\sim\left(
\begin{array}{cc}
1 & \xi^\tau  \\
\frac{\xi}{|\xi|^2} & \frac{\xi\xi^\tau}{|\xi|^2} \\
\end{array}
\right)e^{-C_1|\xi|^2t},
\eess
i.e., there are nonlocal operators with symbol
$\frac{\xi}{|\xi|^2}$ and $\frac{\xi\xi^\tau}{|\xi|^2}$ in Fourier space. Under the assumption on  $\nabla\phi_0$ in (\ref{1.5(0)}), we can regard the nonlocal operator with the symbol $\frac{\xi}{|\xi|^2}$ in the Green's function of
$(\ref{2.4})_{3,4,5}$ as the nonlocal operator with the symbol $\frac{\xi\xi^\tau}{|\xi|^2}$ because of the Poisson equation (\ref{2.4})$_5$. In addition, for the nonlinear estimates, the operator $\frac{\nabla}{-\Delta}$ in the long wave of the Green's function just corresponds to the nonlinear term $F_3(n_1,v_1,n_2,v_2)=-{\rm div}(\frac{n_1v_2+n_2v_1}{2})$  in (\ref{5.18}), which is conservative due to the mass conservation on the variable $\rho_1-\rho_2$. Then, one can still ``borrow" the first order derivative from $F_3(n_1,v_1,n_2,v_2)$ to the Green's function $\mathbb{G}_{21}$. From these facts above, we can obtain a refined
pointwise estimate of the solution to unipolar Navier-Stokes-Poisson system as in (\ref{1.7(4)}). Then, all of the pointwise estimates for the other variables are based on (\ref{1.7(4)}).

Lastly, the presence of the factor $\varepsilon$ in the pointwise estimate for $\nabla\phi$ in Proposition \ref{l 53} will bring us new difficulty when dealing with the nonlinear problem, since it decays slower than all of the variables in the Navier-Stokes system \cite{LS,LW}. In fact, owing to the presence of several different wave profiles: Riesz wave, Huygens' wave, diffusion wave in long wave of the Green's function of the Navier-Stokes system,  to close the ansatz (\ref{5.20}) on the variables $(n_1,w_1,n_2,w_2)$, all of which have the different (unsymmetrical) asymptotic profiles in the pointwise estimates, we have to make some subtle estimations for the convolutions between each term in the Green's functions and the coupled nonlinear terms. Specifically, sometimes we will exchange time decay rate with the space decay rate to avoid producing the factor $\varepsilon$ in the time decay rate, and sometimes we will exchange the D-wave with the H-wave though it will waste the time decay rate to close the ansatz on the different pointwise profiles between $(\rho_1-\rho_2,J_1-J_2)$ and $D_x^\alpha(\rho_1-\rho_2,J_1-J_2)$ with $1\leq |\alpha|\leq 2$.

Throughout this paper, $C$, $C_i$ and $c_0$ denote general positive constants which may vary in different estimates. We use
$H^s(\mathbb{R}^n)=W^{s,2}(\mathbb{R}^n)$, where
$W^{s,p}(\mathbb{R}^n)$ is the usual Sobolev space with its norm
\bess
\|f\|_{W^{s,p}(\mathbb{R}^n)}=\sum_{k=0}^s\|\partial_x^kf\|_{L^p(\mathbb{R}^n)}.
\eess
The following existence result is from \cite{Li3, WX, Zou}, the unique difference is the regularity of the initial data in $H^6(\mathbb{R}^3)$ here.

\begin{theorem}\lbl{l a}{\rm (Existence)} Let $P_1'(\rho_1)>0$ and $P_2'(\rho_2)>0$ for $\rho_1>0$ and $\rho_2>0$ respectively, and $\bar{\rho}>0$ be a constant. Assume that $(\rho_i-\bar{\rho},J_{i0},\nabla\phi_0)\in H^6(\mathbb{R}^3)$
for $i=1,2$, with
$\varepsilon_0=:\|(\rho_{i0}-\bar{\rho},J_{i0},\nabla\phi_0)\|_{H^6(\mathbb{R}^3)}$
small. Then there is a unique global classical solution
$(\rho_i,J_i,\nabla\phi)$ of the Cauchy problem (\ref{1.1})-(\ref{1.2})
satisfying
\bes
\|(\rho_i-\bar{\rho},J_i,\nabla\phi)\|_{H^6(\mathbb{R}^3)}^2\leq
C\varepsilon_0.\lbl{1.4}
\ees
\end{theorem}

Main results of this paper are stated as the following theorem.

\begin{theorem}\lbl{1 b}{\rm (Pointwise\  estimates)} Under the assumptions of
Theorem 1.1. If further, for $|\alpha|\leq 2$
\bes
&&|D_x^\alpha(\rho_{1,0}+\rho_{2,0}-2\bar{\rho},J_{1,0}+J_{2,0})|\leq C\varepsilon_0(1+|x|^2)^{-r_1},\ r_1\geq\frac{21}{10},\\
&&|D_x^\alpha(\rho_{1,0}-\rho_{2,0},J_{1,0}-J_{2,0}),\nabla\phi_0|\leq C\varepsilon_0\big(1+|x|^2\big)^{-r},\ r>\frac{3}{2},\lbl{1.5(0)}
\ees
then for the base sound speed $c:=\sqrt{P_i'(\bar{\rho})}$ and a constant $\varepsilon$ satisfying $0<\varepsilon\ll1$, it holds that
\bes
&&|\rho_1+\rho_2-2\bar{\rho}|\leq C\varepsilon_0(1+t)^{-2}\left\{(1+\frac{(|x|-ct)^2}{1+t})^{-(\frac{3}{2}-\varepsilon)}
+\bigg(1+\frac{|x|^2}{1+t}\bigg)^{-(\frac{3}{2}-\varepsilon)}\right\}, \\ \lbl{1.7(1)}
&&|J_1+J_2|\leq C\varepsilon_0(1+t)^{-2}\bigg(1+\frac{(|x|-ct)^2}{1+t}\bigg)^{-(\frac{3}{2}-\varepsilon)}\!\!
+C\varepsilon_0(1+t)^{-\frac{3}{2}}\bigg(1+\frac{|x|^2}{1+t}\bigg)^{-(\frac{3}{2}-\varepsilon)}\!\!\!,\\ \lbl{1.7(2)}
&&|\rho_1-\rho_2|\leq  C\varepsilon_0(1+t)^{-2}\bigg(1+\frac{|x|^2}{1+t}\bigg)^{-\frac{3}{2}},\\ \lbl{1.7(3)}
&&|J_1-J_2|\leq C\varepsilon_0(1+t)^{-\frac{3}{2}}\bigg(1+\frac{|x|^2}{1+t}\bigg)^{-\frac{3}{2}}, \lbl{1.7(4)}
\ees
which imply that
\bes
&&|(\rho_1-\bar{\rho},\rho_2-\bar{\rho})|\leq C\varepsilon_0(1+t)^{-2}\Bigg\{\bigg(1+\frac{(|x|-ct)^2}{1+t}\bigg)^{-(\frac{3}{2}-\varepsilon)}+\bigg(1+\frac{|x|^2}{1+t}\bigg)^{-(\frac{3}{2}-\varepsilon)}\Bigg\},\\
&&|(J_1,J_2)|\leq C\varepsilon_0(1+t)^{-2}\bigg(1+\frac{(|x|-ct)^2}{1+t}\bigg)^{-(\frac{3}{2}-\varepsilon)}\!\!+C\varepsilon_0(1+t)^{-\frac{3}{2}}\bigg(1+\frac{|x|^2}{1+t}\bigg)^{-(\frac{3}{2}-\varepsilon)}.
\lbl{1.10}
\ees
\end{theorem}
\!\!\noindent{\bf Remark 1.1.} We find that the pointwise estimate of the densities is better than those of the momentums, which is resulted from two facts. On the one hand, the first row of the Green's function of the Navier-Stokes system corresponding to the variable $\rho_1+\rho_2$ only contains the Huygens' wave and the rest rows of the Green's function corresponding to the variable $J_1+J_2$ contain both the Riesz wave and the Huygens' wave. On the other hand, the assumption (\ref{1.5(0)}) on $\nabla\phi$ can enhance the decay rate of the variable $\rho_1-\rho_2$ in the Navier-Stokes-Poisson system in $(\ref{2.4})_{3,4,5}$.\\
\noindent{\bf Remark 1.2.} Comparing with the pointwise estimates for the Navier-Stokes system, the solution is also in $L^p(\mathbb{R}^3)$ with $p>1$ in spite of an extra factor $\epsilon$ in Theorem \ref{1 b}. That is, the factor  $\epsilon$ does not matter the range of $p$
and the pointwise estimates with optimal decay rate above for $(\rho_1-\bar{\rho},\rho_2-\bar{\rho},J_1,J_2)$ are almost the same as in \cite{LS,LW}. Here the factor $\epsilon$ is caused by the pointwise estimate for $\nabla\phi$, which does not exist in the Navier-Stokes system \cite{LS,LW}.\\
\noindent{\bf Remark 1.3.} The result in Theorem \ref{1 b} also holds for other odd dimensions $n\geq5$ without any new difficulty. On the other hand, it is more difficult to deduce the pointwise estimates for the problem (\ref{1.1})-(\ref{1.2}) in even dimensions, especially in two dimensions,  which will be considered in future.

From Theorem \ref{1 b} and Remark 1.3, we have the following $L^p$-decay rates in odd dimensions $n\geq 3$.

\begin{corollary}\lbl{1 c}
  Under the assumptions in Theorem \ref{1 b},
we have the following optimal $L^p(\mathbb{R}^n)$ estimates of the solution in odd dimensions $n\geq3$
\bess
&&\|(\rho_1-\bar{\rho}, \rho_2-\bar{\rho},\rho_1+\rho_2-2\bar{\rho})(\cdot,t)\|_{L^p(\mathbb{R}^n)}\leq
C\varepsilon_0(1+t)^{-\left(\frac{n}{2}(1-\frac{1}{p})+\frac{n-1}{4}\big(1-\frac{2}{p}\big)\right)},\ \ \ 1<p\leq \infty,\ \ \ \ \ \ \\[1.5mm]
&&\|(\rho_1-\rho_2)(\cdot,t)\|_{L^p(\mathbb{R}^n)}\leq C\varepsilon_0(1+t)^{-\frac{n}{2}(1-\frac{1}{p})-\frac{1}{2}},\ \ \ \ \ 1<p\leq \infty,\\[1.5mm]
&&\|(J_1,J_2,J_1+J_2)(\cdot,t)\|_{L^p(\mathbb{R}^n)}\leq
\left\{\begin{array}{ll}
  C\varepsilon_0(1+t)^{-\left(\frac{n}{2}(1-\frac{1}{p})+\frac{n-1}{4}\big(1-\frac{2}{p}\big)\right)}, & \ \ \ 1<p\leq2, \\[1.5mm]
 C\varepsilon_0(1+t)^{-\frac{n}{2}(1-\frac{1}{p})}, & \ \ \ 2\leq p\leq\infty,
\end{array}\right.\\
&&\|(J_1-J_2)(\cdot,t)\|_{L^p(\mathbb{R}^n)}\leq C\varepsilon_0(1+t)^{-\frac{n}{2}(1-\frac{1}{p})},\ \ \ \ \ \ \ \ 1<p\leq \infty.
\eess
\end{corollary}
\!\!\!\noindent{\bf Remark 1.4.} The usual $L^2$-estimates conceal the hyperbolic feature for the hyperbolic-parabolic coupled system. Our $L^p$-estimates above imply that the dominant part of $(\rho_1,\rho_2)$ is always the H-wave for all of $p>1$, the dominant part of $(\rho_1-\rho_2,J_1-J_2)$ is always the D-wave for all of $p>1$, and the dominant part of $(J_1,J_2)$ is the H-wave when $p<2$ and the D-wave when $p\geq2$. In fact, the $L^p$-decay rate of $(\rho_1,\rho_2)$ is completely new, which can hardly be deduced by  using  $L^2$-estimates together with Sobolev inequalities.

The rest of the paper is arranged as follows. In Section 2, we reformulate the system (\ref{1.1}) into a new system that we will considered in the next sections. In Section 3 and Section 4, we give the pointwise estimates for the Green's functions of the Navier-Stokes system and the Navier-Stokes-Poisson system. The pointwise estimates of the solution to the nonlinear Cauchy problem (\ref{1.1})-(\ref{1.2}) will be deduced in Section 5. Lastly, some useful lemmas on the nonlinear estimates are given in Appendix.

\section{Reformulation of original problem}\ \

\quad Assume $\bar{\rho}=1$ and $\sqrt{P_i'(\bar{\rho})}=c$ without loss
of generality. Then the Cauchy problem (1.1)-(1.2) is reformulated
as
\bes
\left\{\begin{array}{l}
\partial_t \rho_1+{\rm div} J_1=0, \\[0.5mm]
\partial_t J_1\!-\!\mu_1\Delta J_1\!-\!\mu_2\nabla{\rm div} J_1+c^2\nabla \rho_1-\nabla\phi=-{\rm div}(\frac{J_1\otimes J_1}{\rho_1})\!-\!\nabla(P_1(\rho_1)\!-\!c^2\rho_1)\!+\!(\rho_1\!-\!1)\nabla\phi, \\[0.5mm]
\partial_t \rho_2+{\rm div} J_2=0, \\[0.5mm]
\partial_t J_2\!-\!\mu_1\Delta J_2\!-\!\mu_2\nabla{\rm div} J_2+c^2\nabla \rho_2+\nabla\phi=-{\rm div}(\frac{J_2\otimes J_2}{\rho_2})\!-\!\nabla(P_2(\rho_2)\!-\!c^2\rho_2)-(\rho_2\!-\!1)\nabla\phi, \\[0.5mm]
\Delta\phi=\rho_1-\rho_2,\\[0.5mm]
(\rho_1,J_1,\rho_2,J_2,\nabla\phi)(x,0)=(\rho_{1,0},J_{1,0},\rho_{2,0},J_{2,0},\nabla\phi_0)(x).
\end{array}
        \right.
        \lbl{2.1}
\ees
Next, set
\bes
n_1=\rho_1+\rho_2-2,\ n_2=\rho_1-\rho_2,\ w_1=J_1+J_2,\
w_2=J_1-J_2,\lbl{2.2}
\ees
which give equivalently
\bes
\rho_1=\frac{n_1+n_2}{2}+1,\ \rho_2=\frac{n_1-n_2}{2}+1,\
J_1=\frac{w_1+w_2}{2},\ J_2=\frac{w_1-w_2}{2}.\lbl{2.3}
\ees

From (\ref{2.2}) and (\ref{2.3}), it follows that the Cauchy problem (\ref{2.1}) can
be reformulated into the following Cauchy problem for the unknown
$(n_1,w_1,n_2,w_2,\nabla\phi)$
\bes
\left\{\begin{array}{l}
\partial_t n_1+{\rm div} w_1=0, \\
\partial_t w_1-\mu_1\Delta w_1-\mu_2\nabla{\rm div} w_1+c^2\nabla n_1=F_1(n_1,w_1,n_2,w_2), \\
\partial_t n_2+{\rm div} w_2=0, \\
\partial_t w_2-\mu_2\Delta w_2-\mu_2\nabla{\rm div} w_2+c^2\nabla n_2-2\nabla\phi=F_2(n_1,w_1,n_2,w_2), \\
\Delta\phi=n_2,\\
(n_1,w_1,n_2,w_2,\nabla\phi)(x,0)=(n_{1,0},w_{1,0},n_{2,0},w_{2,0},\nabla\phi_0)(x),
\end{array}
        \right.
        \lbl{2.4}
\ees
where
$(n_{1,0},w_{1,0},n_{2,0},w_{2,0}):=(\rho_{1,0}+\rho_{2,0}-2,J_{1,0}+J_{2,0},\rho_{1,0}-\rho_{2,0},J_{1,0}-J_{2,0})$,
and
\bes  \begin{array}[b]{rl}
 F_1=&-{\rm div}\Big(\frac{(w_1+w_2)\otimes(w_1+w_2)}{2(n_1+n_2)+4}+\frac{(w_1-w_2)\otimes(w_1-w_2)}{2(n_1-n_2)+4}\Big)\\
 &-\nabla\Big(P_1(\frac{n_1+n_2}{2}+1)-c^2\frac{n_1+n_2}{2}+P_2(\frac{n_1-n_2}{2}+1)-c^2\frac{n_1-n_2}{2}\Big)+n_2\nabla\phi\\
 &-\mu_1\Delta\bigg(\frac{(n_1-n_2)(w_1-w_2)}{2(n_1-n_2)+4}+\frac{(n_1+n_2)(w_1+w_2)}{2(n_1+n_2)+4}\bigg)\\
 &-\mu_2\nabla{\rm div}\bigg(\frac{(n_1-n_2)(w_1-w_2)}{2(n_1-n_2)+4}+\frac{(n_1+n_2)(w_1+w_2)}{2(n_1+n_2)+4}\bigg),\\
 F_2=&-{\rm div}\Big(\frac{(w_1+w_2)\otimes(w_1+w_2)}{2(n_1+n_2)+4}-\frac{(w_1-w_2)\otimes(w_1-w_2)}{2(n_1-n_2)+4}\Big),\\
 &-\nabla\Big(P_1(\frac{n_1+n_2}{2}+1)-c^2\frac{n_1+n_2}{2}-P_2(\frac{n_1-n_2}{2}+1)+c^2\frac{n_1-n_2}{2}\Big)+n_1\nabla\phi\\
 &-\mu_1\Delta\bigg(\frac{(n_1+n_2)(w_1+w_2)}{2(n_1+n_2)+4}-\frac{(n_1-n_2)(w_1-w_2)}{2(n_1-n_2)+4}\bigg)\\
 &-\mu_2\nabla{\rm div}\bigg(\frac{(n_1+n_2)(w_1+w_2)}{2(n_1+n_2)+4}-\frac{(n_1-n_2)(w_1-w_2)}{2(n_1-n_2)+4}\bigg).
\end{array}
\lbl{2.5}
\ees

\section{Green's function of linearized Navier-Stokes system}\ \

In this section, we shall give an analysis on the Green's function for the Navier-Stokes system by using complex method,
which have been studied in \cite{LW} by using real method. In fact, recently Liu $et\ al.$ \cite{LS} reconsidered it by using complex method. For completeness, we would like to restate it in another form with some slight differences.

First, the domain is divided into two parts: the finite Mach number region$\{|x|\le4ct\}$ and the outside finite Mach number region$\{|x|>3ct\}$. In the outside finite Mach number region, we will use the weighted energy estimate method to obtain the pointwise estimate. Inside the finite Mach number region, we will use the long-wave short-wave decomposition method. Such weighted method was introduced in \cite{U} to study the boundary layer problem, and was used by other authors \cite{LY1,Z}.

Let
$$
\chi_1(\xi)=\left\{\begin{array}{ll}
1, &|\xi|<\varepsilon_1, \\
0, &|\xi|>2\varepsilon_1,
\end{array}\right.
$$
and
$$\chi_3(\xi)=\left\{\begin{array}{ll}
1, &|\xi|>K+1, \\
0, &|\xi|<K,
\end{array}\right.
$$
be the smooth cut-off functions with $2\varepsilon_1<K$, and
$\chi_2=1-\chi_1-\chi_3$. Then we can divide the Green's function into three parts: the long wave when $|\xi|$ small, the short wave when $|\xi|$ large and  the middle part when $|\xi|$ is in the middle:
 \bes
D_x^{\alpha}G(x,t)&=&\frac{1}{(2\pi)^n}\left(\int_{|\xi|\le\varepsilon_1}+\int_{\varepsilon_1\le|\xi|\le K}+\int_{|\xi|\ge K}\right)(i\xi)^{\alpha}\hat{G}(\xi,t)e^{ix\cdot\xi}d\xi\nm\\
&:=&D_x^{\alpha}(\chi_1(D))G(x,t)+D_x^{\alpha}(\chi_2(D))G(x,t)+D_x^{\alpha}(\chi_3(D))G(x,t).
\lbl{3.1}
\ees

Recall the linearized system on $(n_1,w_1)$ in (\ref{2.4})
\bes
\left\{\begin{array}{l}
\partial_t n_1+{\rm div} w_1=0, \\
\partial_t w_1-\mu_1\Delta w_1-\mu_2\nabla{\rm div} w_1+c^2\nabla n_1=0. \\
\end{array}\right.
\lbl{3.2}
\ees
The symbol of the operator in (\ref{3.2}) with $\mu=\mu_1+\mu_2$ is
\bes
\tau^2+\mu|\xi|^2\tau+c^2|\xi|^2=0.\lbl{3.3}
\ees
The eigenvalues of (\ref{3.3}) for $\lambda_\pm$ are
\bes
\lambda_{\pm}(\xi)=\frac{-\mu|\xi|^2\pm\sqrt{\mu^2|\xi|^4-4c^2|\xi|^2}}{2}.
 \lbl{2.22}
\ees
Then we consider the Green's function for (\ref{3.2}), i.e.,
\bes
\left\{\begin{array}{l} (\frac{\partial}{\partial t}+A(D_x))G(x,t)=0, \\
     G(x,0)=\delta(x),
       \end{array}
        \right.
         \lbl{3.4}
\ees
where $\delta(x)$ is the Dirac function, the symbols of operator
$A(D_x)$ are
$$A(\xi)=\left(
           \begin{array}{cc}
             0 & \sqrt{-1}\xi^\tau \\
             \sqrt{-1}\xi & \mu_1|\xi|^2I+\mu_2\xi\xi^\tau \\
           \end{array}
         \right).
$$

Thus by direct calculation, we get
\bes
\hat{G}(\xi,t)=\left(
                 \begin{array}{cc}
                   \hat{G}_{11} & \hat{G}_{12}  \\
                   \hat{G}_{21} & \hat{G}_{22} \\
                 \end{array}
               \right), \lbl{3.5}
\ees
where
$$
\hat{G}_{11}=\frac{\lambda_+e^{\lambda_-t}-\lambda_-e^{\lambda_+t}}{\lambda_+-\lambda_-},\ \
\hat{G}_{12}=-\sqrt{-1}c^2\frac{e^{\lambda_+t}-e^{\lambda_-t}}{\lambda_+-\lambda_-}\xi^\tau,
$$
$$
\hat{G}_{21}=-\sqrt{-1}c^2\frac{e^{\lambda_+t}-e^{\lambda_-t}}{\lambda_+-\lambda_-}\xi,\ \
\hat{G}_{22}=\frac{\xi\xi^\tau}{|\xi|^2}\frac{\lambda_+e^{\lambda_+t}-\lambda_-e^{\lambda_-t}}{\lambda_+-\lambda_-}
-e^{-\mu_1|\xi|^2t}\bigg(I-\frac{\xi\xi^\tau}{|\xi|^2}\bigg).
$$

Sometime, we also rewrite the Fourier transform of the Green's function as
follows
\bess
&&\hat{G}^+=e^{\lambda_+t}L_1=\left(
            \begin{array}{cc}
              -\eta_- & -\sqrt{-1}\eta_0\xi^\tau \\
              -\sqrt{-1}c^2\eta_0\xi & \eta_+\frac{\xi\xi^\tau}{|\xi|^2}\\
            \end{array}
          \right)e^{\lambda_+t},\\[0.5mm]
&&\hat{G}^-=e^{\lambda_-t}L_2=\left(
            \begin{array}{cc}
              \eta_+ & \sqrt{-1}\eta_0\xi^\tau  \\
              \sqrt{-1}c^2\eta_0\xi  & -\eta_-\frac{\xi\xi^\tau}{|\xi|^2}  \\
            \end{array}
          \right)e^{\lambda_-t},\\[0.5mm]
&&\hat{G}^0=e^{-\mu_1|\xi|^2t}L_3=\left(
            \begin{array}{cc}
              0 & 0 \\
              0 & I-\frac{\xi\xi^\tau}{|\xi|^2}  \\
            \end{array}
          \right)e^{-\mu_1|\xi|^2t},
\eess
where $\eta_0(\xi)=(\lambda_+(\xi)-\lambda_-(\xi))^{-1},
\eta_\pm(\xi)=\lambda_{\pm}(\xi)\eta_0(\xi)$.

\begin{lemma}\lbl{l 31} \cite{LW}
For a sufficiently small number $\varepsilon_1>0$ and a sufficiently large number $K>0$, we have the following:

(i) when $|\xi|<\varepsilon_1$, $\lambda_\pm$ are complex conjugates and have the following expansion:
\bess
\lambda_+&=&-\frac{\mu}{2}|\xi|^2+\sqrt{-1}c|\xi|\bigg(1+\sum_{j=1}^{\infty}d_j|\xi|^{2j}\bigg),\\
\lambda_-&=&-\frac{\mu}{2}|\xi|^2-\sqrt{-1}c|\xi|\bigg(1+\sum_{j=1}^{\infty}d_j|\xi|^{2j}\bigg);
\eess

 (ii) when $\varepsilon_1 \le |\xi|\le K$, $\lambda_\pm$ have the following spectrum gap property:
 \bess
 {\rm Re}(\lambda_\pm)\le-C,\ \  \mbox{for some constant} \ \ C>0;
 \eess

(iii) when $|\xi|>K\gg1$,
 $\lambda_\pm$ are real and has the following expansion:
 \bess
\lambda_+&=&-\frac{c^2}{\mu}+\frac{\mu}{2}\sum_{j=1}^{\infty}e_j|\xi|^{-2j},\\
\lambda_-&=&-\mu|\xi|^2+\frac{c^2}{\mu}-\sum_{j=1}^{\infty}e_{j}|\xi|^{-2j}.
\eess
Here $c=\sqrt{P_i'(\bar{\rho})}$ is defined as the equilibrium sound speed, and all $d_j,e_j$ are real constants.
 \end{lemma}

  \subsection{ Pointwise estimate in the finite Mach number region}

 \subsubsection{Long wave estimate}
  We first give the Kirchhoff method through the Fourier transform method, see also in \cite{LS, LW, LY3, Z}:

 \begin{lemma}\lbl{l 32}
 Let $w(x,t)$ be a function given by its Fourier transformation in $\mathbb{R}^3$:
 \bess
 \hat{w}=\frac{\sin(c|\xi |t)}{c|\xi|},\ \ \ \ \
 \hat{w}_t=\cos(c|\xi|t).
 \eess

 Then, for any function $g(x)$ one has that
 \bess
 w\ast g(x)&=&\frac{t}{4\pi}\int_{|y|=1}g(x+cty)dS_y,\\
  w_t\ast g(x)&=&\frac{1}{4\pi} \int_{|y|=1}g(x+cty)dS_y+\frac{ct}{4\pi}\int_{|y|=1}\nabla g(x+cty)\cdot ydS_y.
 \eess
 \end{lemma}

  To study the dissipation of different waves in the Navier-Stokes system, we need the following lemma with a slight modification.
    \begin{lemma}\lbl{l 33} \cite{LY5, Z}
  Let $w(x,t)$ be the inverse Fourier transformation of $\sin(c|\xi|t)$ given in Lemma \ref{l 32}. Then one has that
  \bess
 \left|w\ast\frac{e^{-\frac{|x|^2}{4D(1+t)}}}{(4D\pi (1+t))^{3/2}}\right|\le\mathcal{O}(1)W_3(x,1+t,2D),\\
 \left|w_t\ast\frac{e^{-\frac{|x|^2}{4D(1+t)}}}{(4D\pi (1+t))^{3/2}}\right|\le\frac{\mathcal{O}(1)}{1+t}W_3(x,1+t,2D),
  \eess
  where
    \bess
W_3(x,1+t,D)= \left\{\begin{array}{lll}\medskip
\frac{1}{(1+t)^{3/2}\sqrt{D}}, &&||x|-ct|\le\sqrt{D (1+t)},\\
\frac{e^{-\frac{(|x|-ct)^2}{4(1+t)D}}}{(1+t)^{3/2}\sqrt{D }}, &&||x|-ct|\ge \sqrt{D (1+t)}.
 \end{array}\right.
\eess
 \end{lemma}
 The  following lemma reveals the process of the coupling of isotropic and shear wave  within a cone.  We modify the relevant results from \cite{LY3,LY4}, and we omit its proof for simplicity.
  \begin{lemma}\lbl{l 34}
  There exists a constant $C$, such that

 When $0<t\le1$,
\bess
&&\left|\int_0^ts\int_{|y|=1}\frac{e^{-\frac{|x+csy|^2}{C t}}}{t^{5/2}}dS_yds\right|=C\frac{e^{-\frac{|x|^2}{C t}}}{(1+t)^{3/2}},\ \ \mbox{for} \ \  |x|\le ct^{1/2},\\
&&\left|\int_0^ts\int_{|y|=1}\frac{e^{-\frac{|x+csy|^2}{C t}}}{t^{5/2}}dS_yds\right|=C\frac{e^{-\frac{(|x|-ct)^2}{2C t}}}{(1+t)^2}, \ \mbox{for} \  \  |x|\ge ct^{1/2};
\eess

 When $t\ge1$,
\bess
&&\left|\int_0^ts\int_{|y|=1}\frac{e^{-\frac{|x+csy|^2}{C t}}}{t^{5/2}}dS_yds\right|=C\frac{e^{-\frac{|x|^2}{C t}}}{t^{3/2}},\ \ \mbox{for} \ \  |x|\le ct^{1/2},\\
&&\left|\int_0^ts\int_{|y|=1}\frac{e^{-\frac{|x+csy|^2}{C t}}}{t^{5/2}}dS_yds\right|=C\left(t^{-\frac{3}{2}}\left(1+\frac{|x|^2}{t}\right)^{-\frac{3}{2}}+\frac{e^{-\frac{(|x|-ct)^2}{2C t}}}{t^2}\right), \ \mbox{for} \  \  ct^{1/2}\le|x|\le ct,\\
&&\left|\int_0^ts\int_{|y|=1}\frac{e^{-\frac{|x+csy|^2}{C t}}}{t^{5/2}}dS_yds\right|=C\frac{e^{-\frac{(|x|-ct)^2}{2C t}}}{t^2}, \ \mbox{for} \  \ |x|\ge ct.
\eess
\end{lemma}

Now we invert each part of Green's function separately.
When $|\xi|\ll1$,  the Taylor expansion gives the following explicit expression of Green's function
\bes
&&\chi_1(\xi)\hat{G}(\xi,t)=e^{\lambda_+t}L_1+e^{\lambda_-t}L_2+e^{-\mu_1|\xi|^2t}L_3\nm\\
&=&e^{-\frac{1}{2}\mu|\xi|^2t+\mathcal{O}(1)|\xi|^3t}(e^{\sqrt{-1}|\xi|(c+\beta(|\xi|^2))t} L_1+e^{-\frac{1}{2}\mu|\xi|^2t+\mathcal{O}(1)|\xi|^3t}(e^{-\sqrt{-1}|\xi|(c+\beta(|\xi|^2))t} L_2+e^{-\mu_1|\xi|^2t}L_3\nm\\
&=&e^{-\mu_1|\xi|^2t}L_3+e^{-\frac{1}{2}\mu|\xi|^2t+\mathcal{O}(1)|\xi|^3t}(\cos((c|\xi|+|\xi|\beta(|\xi|^2))t)(L_1+L_2)+\sqrt{-1}\sin((c|\xi|+|\xi|\beta(|\xi|^2))t)(L_1-L_2))\nm\\
&=&e^{-\mu_1|\xi|^2t}L_3+e^{-\frac{1}{2}\mu|\xi|^2t+\mathcal{O}(1)|\xi|^3t}\cdot\left\{\left(\cos(c|\xi|t)\cos(|\xi|\beta(|\xi|^2)t)\!-\!\frac{\sin(c|\xi| t)}{c|\xi|}c|\xi|\sin(|\xi|\beta(|\xi|^2)t)\right)(L_1\!+\!L_2)\right.\nm\\
&&  \ \ \ \ \ \ \ \ \ \  \ \ \ \ \ \ \ \  \ \ \ \ \ \ \ \ \ \ \left.\ \ \ +\sqrt{-1}\left(\frac{\sin(c|\xi| t)}{c|\xi| }c|\xi|\cos(|\xi|\beta(|\xi|^2)t)+\cos(c|\xi| t)\sin(|\xi|\beta(|\xi|^2)t)\right)(L_1-L_2)\right\}\nm\\
&=&e^{-\mu_1|\xi|^2t}L_3+\cos(c|\xi| t )e^{-\frac{1}{2}\mu|\xi|^2t+\mathcal{O}(1)|\xi|^3t}
\Big(\cos(|\xi|\beta(|\xi|^2)t)(L_1+L_2)+\sqrt{-1}\sin(|\xi|\beta(|\xi|^2)t)(L_1\!-\!L_2)\Big)\nm\\
&&\!+\frac{\sin(c|\xi| t)}{c|\xi|}e^{\!-\frac{1}{2}\mu|\xi|^2t+\mathcal{O}(1)|\xi|^3t}\Big(\!-c|\xi|\sin(|\xi|\beta(|\xi|^2)t)(L_1\!+\!L_2)\!+\sqrt{-1} c|\xi|\cos(|\xi|\beta(|\xi|^2)t)(L_1\!-\!L_2)\Big).\lbl{2.25}
\ees

\begin{lemma}\lbl{l 35}
For a sufficiently small $\xi$, we have the following estimates for $L_1$ and $L_2$:
\bess
L_1+L_2=\left( \begin{array}{ccc}\medskip 1 \ \  & 0 \\
0\  \  & \frac{\xi\xi^\tau}{|\xi|^2}  \end{array}\right),\ \ \
L_1-L_2=\left( \begin{array}{ccc}\medskip i\mathcal{O}(1) \  \ & \mathcal{O}(|\xi|)\\
\mathcal{O}(|\xi|)\  \ & i\mathcal{O}(1)\frac{\xi\xi^\tau}{|\xi|^2}
 \end{array}\right) .
\eess
Here $\mathcal{O}(\xi)$ is an analytic function of $\xi$.
\end{lemma}

Then, we define
 \bess
L_a&=&\left( \begin{array}{ccc}\medskip 0 \ \ &0_{1\times 3}\\
0_{3\times 1}\ \  &\frac{\xi\xi^T}{|\xi|^2}\\
\end{array}\right),
L_b=\left( \begin{array}{ccc}\medskip 0 \ \ &0_{1\times 3}\\
0_{3\times 1}\ \  &\frac{\xi\xi^T}{|\xi|}
\end{array}\right) ,
L_c=\left( \begin{array}{ccc}\medskip 0 \ \ &0_{1\times 3}\\
0_{3\times 1}\ \  &I_{3\times3}\\
\end{array}\right),
\eess
\bess
L_{r_1} = \left( \begin{array}{ccc}\medskip 0 \ \ \ \  \  \ \ \ \ \  &0_{1\times 3}\\
0_{3\times 1}\ \ \ \  & (\cos(c|\xi| t)-1)\frac{\xi\xi^T}{|\xi|^2}\\
 \end{array}\right),\ \
L_{r_2} =\left( \begin{array}{ccc}\medskip 0 \ \   &0_{1\times 3}\\
0_{3\times 1}\ \ \ \ &\cos(c|\xi| t)\frac{\xi\xi^T}{|\xi|^2}
 \end{array}\right).
 \lbl{2.26}
  \eess

First, we divide the third term in the right hand side of (\ref{2.25}) into three parts
   \bess
  e^{-\mu_1|\xi|^2t}L_3= e^{-\mu_1|\xi|^2t}L_c+ e^{-\mu_1|\xi|^2t}L_{r_1}- e^{-\mu_1|\xi|^2t}L_{r_2},
  \eess
where the first part is identified as the rotational pairing:
  $\hat{\mathcal{E}}(\xi,t)=e^{-\mu_1|\xi|^2t}L_c$,
 and the second part is identified as the  Riesz pairing-I:
 $\hat{\mathcal{R}}_1(\xi,t)=e^{-\mu_1|\xi|^2t}L_{r_1}$.
We also define the  Riesz pairing-II as
$$
\hat{\mathcal{R}}_2(\xi,t)=(e^{\!-\!\frac{1}{2}\mu|\xi|^2t+\mathcal{O}(1)|\xi|^3t}-e^{-\mu_1|\xi|^2t})L_{r_2}.
$$
The Huygens' pair is defined as
 \bes
&&\hat{\mathcal{H}}(\xi,t)
=\cos(c|\xi| t )e^{\!-\!\frac{1}{2}\mu|\xi|^2t+\mathcal{O}(1)|\xi|^3t}
\left\{\cos(|\xi|\beta(|\xi|^2)t)(L_1\!+\!L_2\!-\!L_a)+\sqrt{-1}\sin(|\xi|\beta(|\xi|^2)t)(L_1\!-\!L_2\!-\!L_b)\right\}\nm\\ \ \ \
&&\ +\frac{\sin(c|\xi| t)}{c|\xi|}e^{\!-\!\frac{1}{2}\mu|\xi|^2t+\mathcal{O}(1)|\xi|^3t}\!\left\{\!-\! c|\xi|\sin(|\xi|\beta(|\xi|^2)t)(L_1\!+\!L_2\!-\!L_a)\!+\!\sqrt{-1} c|\xi|\cos(|\xi|\beta(|\xi|^2)t)(L_1\!-\!L_2\!-\!L_b)\right\}.\ \ \ \ \ \ \ \
\lbl{2.27}
\ees

Obviously,
\bes
\chi_1(\xi)\hat{G}(\xi,t)=\hat{\mathcal{E}}+\hat{\mathcal{H}}+\hat{\mathcal{R}}_1+\hat{\mathcal{R}}_2+\widehat{\mathcal{RE}},
\ees
where $\widehat{\mathcal{RE}}$ is the rest term with higher order of $|\xi|$.
Next, we will give the estimate of each term above.

 \begin{lemma}\lbl{l 36}
 The rotational wave (or called entropy wave) have the following estimate:
\bess
  \left |D_x^\alpha\mathcal{E}(x,t)\right|&\le&C(1+t)^{-\frac{|\alpha|}{2}}\frac{e^{-\frac{|x|^2}{C(1+t)}}}{(1+t)^{3/2}}+Ce^{-(|x|+t)/C}.
\eess
\end{lemma}
\begin{proof} The estimate of the rotational wave relies on the complex analysis. First, we have
 \bes
  D_x^\alpha \mathcal{E}(x,t)&:=&\frac{1}{(2\pi)^3}\int_{|\xi|\le\varepsilon_1}(\sqrt{-1}\xi)^\alpha e^{\sqrt{-1} \xi\cdot x}e^{-\mu_1|\xi|^2t+\mathcal{O}(1)|\xi|^3t}L_cd\xi.\lbl{2.28}
  \ees
We only consider the case $|\alpha|=0$  since an additional factor $\xi$   will produce extra decay rate of $(1+t)^{-1/2}$.

  For each $x\in \mathbb{R}^3$, there is an orthogonal matrix $\mathcal{Q}$ depending on $x$ such that
  $\mathcal{Q}x=(|x|,0,0)^T$.
  Let $\eta=(\eta_1,\eta_2,\eta_3)^T=\mathcal{Q}\xi$, then
   \bess
 \left|\mathcal{E}(x,t)\right|&\le&C\frac{1}{(2\pi)^3}\int_{|\eta|\le\varepsilon_1}\left|e^{\sqrt{-1} |x|\eta_1}e^{-\mu_1|\eta|^2(1+t)+\mathcal{O}(1)|\eta|^3(1+t)}L_c\right|d\eta,\lbl{2.28(0)}
  \eess
  since $\eta$ is bounded and the orthogonal matrix preserves inner products and norms.

 Let $\mathbb{B}:=[-\frac{\varepsilon_1}{\sqrt{3}},\frac{\varepsilon_1}{\sqrt{3}}]\times[-\frac{\varepsilon_1}{\sqrt{3}},
 \frac{\varepsilon_1}{\sqrt{3}}]\times[-\frac{\varepsilon_1}{\sqrt{3}},\frac{\varepsilon_1}{\sqrt{3}}]$, $\mathbb{B}\subset\{|\eta|\le\varepsilon_1\}$.
  Then the integration is divided into two regions:
  \bess
  \left|\mathcal{E}(x,t)\right|&\le&\frac{1}{(2\pi)^3}\int_{\mathbb{B}}\left|e^{\sqrt{-1} |x|\eta_1}e^{-\mu_1|\eta|^2(1+t)+\mathcal{O}(1)|\eta|^3(1+t)}L_c\right|d\eta\\
  &&+\frac{1}{(2\pi)^3}\int_{\{|\eta|\le \varepsilon_1\}\cap\mathbb{B}^c}\left|e^{\sqrt{-1} |x|\eta_1}e^{-\mu_1\mu|\eta|^2(1+t)+\mathcal{O}(1)|\eta|^3(1+t)}L_c\right|d\eta.
  \eess
  The integrand in the second term is $|e^{-\mu_1|\eta|^2(1+t)+\mathcal{O}(1)|\eta|^3(1+t)}L_c|= \mathcal{O}(1)e^{-\varepsilon_1^2t/C}$ since
  $|\eta|\ge\frac{\varepsilon_1}{\sqrt{3}}$. Thus the second term is $ \mathcal{O}(1)e^{-\varepsilon_1^2t/C}$.
  The first term can be estimated as follows
   \bess
  \left|\mathcal{E}(x,t)\right|&\le&\frac{1}{(2\pi)^3}\int_{\mathbb{B}}\left|e^{\sqrt{-1} |x|\eta_1}e^{-\mu_1|\eta|^2(1+t)+\mathcal{O}(1)|\eta|^3(1+t)}L_c\right|d\eta\\
  &=&\mathcal{O}(1)\int_{-\frac{\varepsilon_1}{\sqrt{3}}}^{\frac{\varepsilon_1}{\sqrt{3}}}
  \int_{-\frac{\varepsilon_1}{\sqrt{3}}}^{\frac{\varepsilon_1}{\sqrt{3}}}
  \int_{-\frac{\varepsilon_1}{\sqrt{3}}}^{\frac{\varepsilon_1}{\sqrt{3}}}\left|e^{\sqrt{-1} |x|\eta_1} e^{-\mu_1|\eta|^2(1+t)+\mathcal{O}(1)|\eta|^3(1+t)}L_c\right|d\eta_1d\eta_2d\eta_3.
  \eess
   Since $L_c$ is analytic in $\eta_1$ near the origin, we choose the contour  as
   $\Gamma=\Gamma_1\cup \Gamma_2\cup \Gamma_3\cup \Gamma_4$:
  \bess
 \Gamma_1&=&\{z: Im z=0,\ \ |Re z|\le \frac{1}{\sqrt{3}}\varepsilon_1\},
 \ \ \ \Gamma_2=\{z:Re z=-\frac{1}{\sqrt{3}}\varepsilon_1, \ \ 0\le Im z\le \sigma\},\\
\Gamma_3&=&\{z:  Im z=\sigma, \ \ |Re z|\le\frac{1}{\sqrt{3}}\varepsilon_1\},
\ \ \ \Gamma_4=\{z:Re z=\frac{1}{\sqrt{3}}\varepsilon_1,  \ \ 0\le Im z\le \sigma  \},\ {\rm with}\ \sigma=\frac{\varepsilon_1}{\sqrt{3}}\frac{x_1}{t}.
   \eess
   One the path  $\Gamma_3$,
   \bess
&&\int_{-\frac{\varepsilon_1}{\sqrt{3}}}^{\frac{\varepsilon_1}{\sqrt{3}}}
\int_{-\frac{\varepsilon_1}{\sqrt{3}}}^{\frac{\varepsilon_1}{\sqrt{3}}}\int_{\Gamma_3}\left| e^{\sqrt{-1} |x|\eta_1} e^{-\mu_1(\eta_1^2+\eta_2^2+\eta_3^2)(1+t)+\mathcal{O}(1)(\eta_1^3+\eta_2^3+\eta_3^3)(1+t)}L_c\right|d\eta_1d\eta_2d\eta_3\\
   &\le&\mathcal{O}(1)\int_{-\frac{\varepsilon_1}{\sqrt{3}}}^{\frac{\varepsilon_1}{\sqrt{3}}}
   \int_{-\frac{\varepsilon_1}{\sqrt{3}}}^{\frac{\varepsilon_1}{\sqrt{3}}}e^{-\mu_1(\eta_2^2+\eta_3^2)(1+t)
   +\mathcal{O}(1)(\eta_2^3+\eta_3^3)(1+t)}\int_{\Gamma_3}e^{-\mu_1\eta_1^2(1+t)+\mathcal{O}(1)\eta_1^3(1+t)}d\eta_1d\eta_2d\eta_3\\
   &\le&C\frac{e^{-\frac{|x|^2}{C(1+t)}}}{(1+t)^{3/2}}.
   \eess

Along the path $\Gamma_2$ and $\Gamma_4$, we can get the exponential decay rate
 \bess
&&\int_{-\frac{\varepsilon_1}{\sqrt{3}}}^{\frac{\varepsilon_1}{\sqrt{3}}}
\int_{-\frac{\varepsilon_1}{\sqrt{3}}}^{\frac{\varepsilon_1}{\sqrt{3}}}\int_{\Gamma_2,\Gamma_4}\left|e^{\sqrt{-1} |x|\eta_1}  e^{-\mu_1(\eta_1^2+\eta_2^2+\eta_3^2)(1+t)+\mathcal{O}(1)(\eta_1^3+\eta_2^3+\eta_3^3)(1+t)}L_c\right|d\eta_1d\eta_2d\eta_3\\
   &\le&\mathcal{O}(1)\int_{-\frac{\varepsilon_1}{\sqrt{3}}}^{\frac{\varepsilon_1}{\sqrt{3}}}
   \int_{-\frac{\varepsilon_1}{\sqrt{3}}}^{\frac{\varepsilon_1}{\sqrt{3}}}e^{-\mu_1(\eta_2^2+\eta_3^2)(1+t)
   +\mathcal{O}(1)(\eta_2^3+\eta_3^3)(1+t)}\int_{\Gamma_2, \Gamma_4}e^{-\mu_1\eta_1^2(1+t)+\mathcal{O}(1)\eta_1^3(1+t)}d\eta_1d\eta_2d\eta_3\\
   &\le&Ce^{-t/C}\le Ce^{-(|x|+t)/C}.
   \eess
    This gives the estimate of the rotational wave.
\end{proof}

 \begin{lemma}\lbl{l 37}
The Huygens' wave $\mathcal{H}(x,t)$ has the following estimate:
\bess
|D_x^\alpha\mathcal{H}(x,t)|\le C(1+t)^{-\frac{|\alpha|}{2}}\frac{e^{-\frac{(|x|-ct)^2}{C(1+t)}}}{(1+t)^{2}}+Ce^{-(|x|+t)/C}.
\eess
\end{lemma}
\begin{proof} First, for convenience's sake we rewrite (\ref{2.27}) as
\bess
\mathcal{H}(x,t)=w_t(\cdot,t)\ast K_1(\cdot,t)+w(\cdot,t)\ast K_2(\cdot,t),
\eess
where $w(x,t)$ and $w_t(x,t)$ is defined in Lemma \ref{l 32}. The estimates of  $K_1$ and $K_2$  are similar to that of $ \mathcal{E}(x,t)$. Note that on the contour $\Gamma$, one also has the growth rate
  \bess
 |\cos(|\xi|\beta(|\xi|^2)t)|, |\sin(|\xi|\beta(|\xi|^2)t)|\le Ce^{\mathcal{O}(1)|\xi|^3t}.
 \eess
 Thus, by the proof of Lemma \ref{l 36}, we have
 \bess
| K_1(x,t)|\le C\frac{e^{-\frac{|x|^2}{C(1+t)}}}{(1+t)^{3/2}}+Ce^{-(|x|+t)/C},\\
| \nabla K_1(x,t)|+| K_2(x,t)|\le C\frac{e^{-\frac{|x|^2}{C(1+t)}}}{(1+t)^{2}}+Ce^{-(|x|+t)/C},
 \eess
where an additional decay rate of $(1+t)^{-1/2}$ is gained due to the extra factor $|\xi|$ in  $\nabla K_1$ and $K_2$.

Using Lemma \ref{l 33} and  \ref{l 34}, we have
\bes
\left|w\ast K_2(x,t)\right|&=&\left|\frac{t}{4\pi}\int_{|y|=1}K_5(x+cty)dS_y\right|
\le\frac{t}{4\pi}\int_{|y|=1}\left|K_2(x+cty)\right|dS_y\nm\\
&\le&\frac{t}{4\pi}\int_{|y|=1}\left|\frac{e^{-\frac{|x+cty|^2}{C(1+t)}}}{(1+t)^{2}}\right|dS_y
\le C\frac{1}{\sqrt{1+t}}W_3(x,1+t,2D);\nm\\
\left|w_t\ast  K_1(x,t)\right|&\le&\frac{1}{4\pi} \int_{|y|=1}\left|K_1(x+cty)\right|dS_y+\frac{ct}{4\pi}\int_{|y|=1}\left|\nabla K_1(x+cty)\cdot y\right|dS_y\nm\\
&\le& C\frac{W_3(x,1+t,2D)}{1+t}+\frac{ct}{4\pi}C(1+t)^{-3}e^{-\frac{(|x|-ct)^2}{C(1+t)}}\le C\frac{1}{1+t}W_3(x,1+t,2D).
\lbl{2.29}
\ees
 Here we  have used the important estimate of
 \bess
\left|\int_{|y|=1}e^{-\frac{|x+cty|^2}{C(1+t)}}\cdot y^\beta dS_y\right|
 \le C(1+t)^{-1}e^{-\frac{(|x|-ct)^2}{C(1+t)}},  \ \ \mbox{for}  \ \ \beta>0,
 \eess
  which was proved by Lemma 2.3 in \cite{LW}.

Hence $$\left|D_x^\alpha\mathcal{H}(x,t)\right|=\left|D_x^\alpha w_t\ast K_1\right|+\left|D_x^\alpha w\ast K_2\right|\le C(1+t)^{-\frac{|\alpha|}{2}}\frac{e^{-\frac{(|x|-ct)^2}{C(1+t)}}}{(1+t)^{2}}+Ce^{-(|x|+t)/C} ,$$
 and we prove this lemma.
  \end{proof}

 \begin{lemma}\lbl{l 38}
The Riesz wave-I has the following estimate:
\bess
\left|D_x^\alpha \mathcal{R}_1(x,t)\right| \le C(1+t)^{-|\alpha|/2}\Bigg\{\chi_{|x|\le ct}(1+t)^{-\frac{3}{2}}\left(1+\frac{x^2}{1+t}\right)^{-\frac{3+|\alpha|}{2}}+\frac{e^{-\frac{(|x|-ct)^2}{2C (1+t)}}}{(1+t)^2}+\frac{e^{-\frac{x^2}{C(1+t)}}}{(1+t)^{3/2}}\Bigg\}.
\eess
\end{lemma}
\begin{proof} Recall that
 \bess
 \mathcal{R}_1(x,t)&=&\frac{1}{(2\pi)^3}\int_{\mathbb{R}^3}e^{\sqrt{-1} \xi\cdot x}e^{-\mu_1|\xi|^2t}L_{r_1}d\xi.
\eess
Since
 \bess
 &&\left|\frac{1}{(2\pi)^3}\int_{\mathbb{R}^3}e^{\sqrt{-1} \xi\cdot x} (\cos(c|\xi| t)-1)\frac{\xi\xi^T}{|\xi|^2}e^{-\mu_1|\xi|^2t}d\xi\right|\nm\\
 &=&\left|\frac{\partial^2}{\partial x_j\partial x_k}\frac{1}{(2\pi)^3}\int_{\mathbb{R}^3}e^{\sqrt{-1} \xi\cdot x} (\cos(c|\xi| t)-1)\frac{1}{|\xi|^2}e^{-\mu_1|\xi|^2t}d\xi\right|\nm\\
&=&\left|\frac{\partial^2}{\partial x_j\partial x_k}\frac{1}{(2\pi)^3}\int_{\mathbb{R}^3}e^{\sqrt{-1} \xi\cdot x} \int_0^t\frac{-c\sin(c|\xi| s)}{|\xi|} ds e^{-\mu_1|\xi|^2t}d\xi\right|\nm\\
&=&\left|-c^2\frac{\partial^2}{\partial x_j\partial x_k}\int_0^t\frac{1}{(2\pi)^3}\int_{\mathbb{R}^3}e^{\sqrt{-1} \xi\cdot x}\frac{\sin(c|\xi| s)}{c|\xi|} e^{-\mu_1|\xi|^2t}d\xi ds\right|\nm\\
&=&\left|c^2\frac{\partial^2}{\partial x_j\partial x_k}\int_0^t \frac{s}{4\pi}\frac{1}{(4\varepsilon\pi t)^{3/2} }
\int_{|y|=1}e^{-\frac{|x+csy|^2}{4\varepsilon t}}dS_yds\right|\\
&=&C\left|\int_0^t s\int_{|y|=1}\frac{e^{-\frac{|x+csy|^2}{4\varepsilon t}}}{t^{5/2}}dS_yds\right|,
 \eess
then by using Lemma \ref{l 34}, we prove this lemma.
\end{proof}

\begin{lemma}\lbl{l 39}
The  Riesz wave-II has the following  estimate:
\bess
|D_x^\alpha \mathcal{R}_2(x,t)|\le C(1+t)^{-\frac{|\alpha|}{2}}\frac{ e^{-\frac{(|x|-ct)^2}{C(1+t)}}}{(1+t)^{2}}.
\eess
\end{lemma}
\begin{proof}After a direct computation, we have
 \bess
&&\frac{1}{(2\pi)^3}\int_{\mathbb{R}^3}\chi_{|\xi|\le \varepsilon_1}e^{\sqrt{-1} \xi\cdot x}(e^{\!-\!\frac{1}{2}\mu|\xi|^2t+\mathcal{O}(1)|\xi|^3t}-e^{-\mu_1|\xi|^2t})L_{r_2}d\xi\nm\\
&=&\frac{1}{(2\pi)^{3}}\int_{\mathbb{R}^3}\chi_{|\xi|\le \varepsilon_1}(e^{\!-\frac{1}{2}\mu|\xi|^2t+\mathcal{O}(1)|\xi|^3t}-e^{-\mu_1|\xi|^2t})\cos(c|\xi| t)\frac{\xi\xi^T}{|\xi|^2}e^{\sqrt{-1}\xi\cdot x}d\xi\nm\\
&=&w_t\ast \frac{1}{(2\pi)^{3}}\int_{\mathbb{R}^3}\chi_{|\xi|\le \varepsilon_1} (e^{\!-\!\frac{1}{2}\mu|\xi|^2t+\mathcal{O}(1)|\xi|^3t}-e^{-\mu_1|\xi|^2t})\frac{\xi\xi^T}{|\xi|^2}e^{\sqrt{-1}\xi\cdot x}d\xi\nm\\
&=&w_t\ast \frac{1}{(2\pi)^{3}}\int_{|\xi|\le\varepsilon_1}\int_{\mu_1 t}^{\frac{1}{2}\mu t-\mathcal{O}(1)|\xi| t} \frac{\partial}{\partial_s}e^{-|\xi|^2s}ds \frac{\xi\xi^T}{|\xi|^2}e^{\sqrt{-1}\xi\cdot x}d\xi\nm\\
&=&w_t\ast\left(\int_{\mu_1 t}^{\frac{1}{2}\mu t-\mathcal{O}(1)|\xi| t}\frac{1}{(2\pi)^{3}}\int_{|\xi|\le\varepsilon_1} e^{-|\xi|^2s}\xi\xi^Te^{\sqrt{-1}\xi\cdot x}d\xi ds\right).
\eess
The contour integral gives
\bess
\left| \int_{|\xi|\le\varepsilon_1} e^{-|\xi|^2s}e^{\sqrt{-1}\xi\cdot x}\xi\xi^Td\xi \right|\le C\frac{ e^{-\frac{x^2}{4(s+1)}}}{(4\pi (s+1))^{5/2}},
 \eess
which imply that
\bess
 |\mathcal{R}_2(x,t)|&=&\left|\frac{1}{(2\pi)^3}\int_{|\xi|\le\varepsilon_1}e^{\sqrt{-1} \xi\cdot x}(e^{\!-\!\frac{1}{2}\mu|\xi|^2t+\mathcal{O}(1)|\xi|^3t}-e^{-\mu_1|\xi|^2t})L_{r_2}d\xi\right|\nm\\
 &=&\left|\frac{1}{(2\pi)^3}\int_{\mathbb{R}^3}\chi_{|\xi|\le \varepsilon_1}e^{\sqrt{-1} \xi\cdot x}(e^{\!-\!\frac{1}{2}\mu|\xi|^2t+\mathcal{O}(1)|\xi|^3t}-e^{-\mu_1|\xi|^2t})L_{r_2}d\xi\right|
\le C\frac{ e^{-\frac{x^2}{C(1+t)}}}{(1+t)^{5/2} }.
\eess
We complete the proof of this lemma.
 \end{proof}

 The rest is
\bess
\left|\widehat{\mathcal{RE}}(\xi,t)\right|&=&\left|\cos(c|\xi| t )e^{-\frac{1}{2}\mu|\xi|^2t+\mathcal{O}(1)|\xi|^3t}
\sqrt{-1}\sin(|\xi|\beta(|\xi|^2)t)L_b\right.\nm\\
&&\left.+\frac{\sin(c|\xi| t)}{c|\xi|}e^{-\frac{1}{2}\mu|\xi|^2t+\mathcal{O}(1)|\xi|^3t}\left(-c|\xi|\sin(|\xi|\beta(|\xi|^2)t)L_a+\sqrt{-1} c|\xi|\cos(|\xi|\beta(|\xi|^2)t)L_b\right)\right|\nm\\
&=&\left|\cos(c|\xi| t )e^{-\frac{1}{2}\mu|\xi|^2t+\mathcal{O}(1)|\xi|^3t}
\sqrt{-1}\sin(|\xi|\beta(|\xi|^2)t)\frac{\xi\xi^T}{|\xi|^2}\right.\nm\\
&&\left.+\frac{\sin(c|\xi| t)}{c|\xi|}e^{-\frac{1}{2}\mu|\xi|^2t+\mathcal{O}(1)|\xi|^3t}\left(-c|\xi|\sin(|\xi|\beta(|\xi|^2)t)\frac{\xi\xi^T}{|\xi|^2}+\sqrt{-1} c|\xi|\cos(|\xi|\beta(|\xi|^2)t)\frac{\xi\xi^T}{|\xi|}\right)\right|.
\eess
  For any analytic function $\beta(|\xi|^2)$, we have $\sin(|\xi|\beta(|\xi|^2)t)=|\xi|\beta(|\xi|^2)t-\frac{1}{6}(|\xi|\beta(|\xi|^2)t)^3+\cdot\cdot\cdot$,
  then $\mathcal{RE}(x,t)$ has the following estimate:
  \bess
  |\mathcal{RE}(x,t)|
 \le C w_t\ast \frac{ e^{-\frac{|x|^2}{C(1+t)}}}{(1+t)^2 }+Cw\ast\frac{ e^{-\frac{|x|^2}{C(1+t)}}}{(1+t)^{5/2}}
 \le C\frac{e^{-\frac{(|x|-ct)^2}{C(1+t)}}}{(1+t)^{5/2}}.
  \eess

Now we conclude the estimate  of the long wave:
\begin{proposition}\lbl{l 310}
The long wave component contains several waves
$$
\chi_1(D)G(x,t)=\mathcal{E}(x,t)+\mathcal{H}(x,t)+\mathcal{R}_1(x,t)+\mathcal{R}_2(x,t)+\mathcal{RE}(x,t),
$$
 and has the following estimate
\bess
    &&\left|D_x^{\alpha}(\chi_1(D)(G_{11},G_{12},G_{21}))(x,t)\right|\le \!C(1+t)^{-\frac{4+|\alpha|}{2}}\frac{e^{-\frac{(|x|-ct)^2}{C(1+t)}}}{(1+t)^{2}}+Ce^{-(|x|+t)/C},\\
    &&\left|D_x^{\alpha}(\chi_1(D)G_{22})(x,t)\right|\! \le \!C(1\!+\!t)^{-\frac{|\alpha|}{2}}\!\!\left(\!\chi_{\small{|x|\le ct}}(1\!+\!t)^{-\frac{3}{2}}\!\!\left(\!1\!+\!\frac{|x|^2}{1\!+\!t}\!\right)^{-\!\frac{3+|\alpha|}{2}}\!\!\!\!\!
    \!\!+\!\frac{e^{-\!\frac{|x|^2}{C(1+t)}}}{(1\!+\!t)^{3/2}}
    \!+\!\frac{e^{-\!\frac{(|x|\!-\!ct)^2}{C(1+t)}}}{(1\!+\!t)^{2}}\!\right) \!+\!Ce^{-(|x|+t)/C}.
 \eess
\end{proposition}

 \subsubsection{Short wave estimate}\ \

\begin{lemma}\lbl{l 311}
When $|\xi|\ge K$, $L_i$ has the following estimates:
\bess
&&L_1=\left( \begin{array}{ccc}\medskip 1+\mathcal{O}(|\xi|^{-2}) \ \  &\sqrt{-1}\xi^T \mathcal{O}(|\xi|^{-2})\\
\sqrt{-1}\xi\mathcal{O}(|\xi|^{-2})\ \  &\xi\xi^T\mathcal{O}(|\xi|^{-4})\\
\end{array}\right),\\
&&L_2=\left( \begin{array}{ccc}\medskip \mathcal{O}(|\xi|^{-2}) \  & \sqrt{-1}\xi^T\mathcal{O}(|\xi|^{-2})\\
\sqrt{-1}\xi\mathcal{O}(|\xi|^{-2})\ \  & \mathcal{O}(|\xi|^{-2})+\mathcal{O}(|\xi|^{-4})\xi\xi^T\\
\end{array}\right),\\
&&L_3=\left( \begin{array}{ccc}\medskip \mathcal{O}(|\xi|^{-2}) \  & \sqrt{-1}\xi^T\mathcal{O}(|\xi|^{-2})\\
\sqrt{-1}\xi\mathcal{O}(|\xi|^{-2})\ \  &\xi\xi^T\mathcal{O}(|\xi|^{-4})
\end{array}\right) .
\eess
\end{lemma}

\begin{proposition}\lbl{l 312}
There exists a constant $C$ such that the short wave part has the following estimates:
\bess
\left| D_{x}^{\alpha}(\chi_3(D)G(x,t)-G_{S_2}(x,t))\right|\le C\frac{e^{-\frac{t}{2C}}e^{-\frac{|x|^2}{2Ct}}}{t^{(3+|\alpha|)/2}}+ Ce^{-(|x|+t)/C}.
\eess
\end{proposition}

Here $G_{S_2}$ is defined in the following Lemma \ref{l 313}, which contains a Dirac function and a Dirac-like function with exponential decay rate on the time $t$. And it is from the terms in short wave of Green's function corresponding to the eigenvalue $\lambda_+\sim -\frac{c^2}{\mu}$ for the short wave. For convenience, we denote the term $\frac{e^{-\frac{t}{2C}}e^{-\frac{|x|^2}{2Ct}}}{t^{(3+|\alpha|)/2}}$ as $G_{S_1}(x,t)$, which is from the terms in short wave of Green's function corresponding to the eigenvalue $\lambda_-\sim -\mu|\xi|^2$.

\begin{lemma} \lbl{l 313} \cite{HZ,LS}
The singular term $G_{S_2}$ in short wave satisfies
\bes
&&G_{S_2}(x,t)=e^{-\frac{t}{2C}}\left(
                                      \begin{array}{cc}
                                        1 & 0 \\
                                        0 & 0 \\
                                      \end{array}
                                    \right)\delta(x)+f(x),\\
&&|f(x)|\leq e^{-bt}\tilde{f}(x),\ \tilde{f}(x)\in L^1,\ \tilde{f}(x)=\left\{\begin{array}{lcr}
C|x|^{-2},\ |x|\leq R,\\
C|x|^{-N},\ |x|>R,
\end{array}\right.
\ees
for some positive constant $b$ and any given positive integer $N$.
\end{lemma}

 \subsubsection{ Middle part estimate}\ \

 The following results is standard for the middle part of the Green's function as in \cite{Wang2,LS}, we omit the details for simplicity.

   \begin{lemma}\lbl{l 314}
When $\{\varepsilon_1\le|\xi|\le K\}$,
$L_1+L_2+L_3$ is bounded and analytic. The only possible pole for $L_1+L_2+L_3$ is $\xi=0$ and has been excluded here.
\end{lemma}

 \begin{proposition}\lbl{l 315}
There exists a constant $C$ such that the middle part of the Green's function has the following estimate:
\bess
\left|D_x^{\alpha}\chi_2(D)G(x,t)\right|\le Ce^{-(|x|+t)/C}.
\eess
\end{proposition}

\subsection{Energy estimate outside the finite Mach number region}\ \

 Consider the initial value problem for the linear part of the system (\ref{2.4})$_{1,2}$ outside the finite Mach number region$\{|x|>3ct\}$:
   \bes
 \left\{\begin{array}{lll}\medskip
\partial_tn_1+{\rm div} w_1=0,\\
\partial_t w_1+c^2\nabla n_1=\mu_1\Delta w_1+\mu_2\nabla {\rm div} \  w_1,\\
n_1(x,0)=\rho_{1,0}(x), w_1(x,0)=w_{1,0}(x).
 \end{array}\right.
 \lbl{2.30}
  \ees

   Multiplying each side of the two equations in (\ref{2.30}) by $ce^{a(|x|-M ct)}n_1$ and $e^{a(|x|-Mc t)}w_1$ respectively and integrating with respective to $x$ on $\mathbb{R}^3$,
    we have the $L^2$ estimate of $n_1$, $w_1$. Similarly, one can derive the $L^2$ estimate of $\nabla_x^\alpha n_1$ and $\nabla_x^\alpha w_1$ when $|\alpha|\geq 1$.  Hence,  with the help of Sobolev embedding theorem, we can prove the following Proposition.

  \begin{proposition}\lbl{l 316} For $|x|>4ct$, there exists a positive constant $C$ independent of $|x|$ and $t$ such that
  \bess
  |D_x^\alpha G(x,t)\ast n_{1,0}(x)|+|D_x^\alpha G(x,t)\ast w_{1,0}(x)|\leq Ce^{-(|x|+t)/C}.
\eess
\end{proposition}

Up to now, from Proposition \ref{l 310}-\ref{l 316}, we get the following pointwise estimates for each term in the Green's function $G(x,t)$, which contains the H-wave and D-wave mentioned in \S1.

\begin{proposition}\lbl{l 317} For any $|\alpha|\geq0$, we have for some constant $C>0$
\bess
|D_x^\alpha(G_{11}\!-\!G_{S_2})|\!+\!|D_x^\alpha(G_{12}\!-\!G_{S_2})|\!+\!|D_x^\alpha(G_{21}\!-\!G_{S_2})|
\leq  C(1+t)^{-\frac{4+|\alpha|}{2}}e^{-\frac{(|x|-ct)^2}{1+t}}+\frac{e^{-\frac{t}{2C}}e^{-\frac{|x|^2}{2Ct}}}{t^{(3+|\alpha|)/2}},\\
|D_x^\alpha (G_{22}\!-\!G_{S_2})|
\leq C(1+t)^{-\frac{3+|\alpha|}{2}}\left\{(1+t)^{-\frac{1}{2}}e^{-\frac{(|x|-ct)^2}{1+t}}+\bigg(1+\frac{|x|^2}{1+t}\bigg)^{-\frac{3+|\alpha|}{2}}\right\}+\frac{e^{-\frac{t}{2C}}e^{-\frac{|x|^2}{2Ct}}}{t^{(3+|\alpha|)/2}},\ \
\eess
where $G_{S_2}$ is defined in Lemma \ref{l 313}.
\end{proposition}

\section{Green's function of linearized Navier-Stokes-Poisson system}\ \

In this section, we shall review the pointwise estimates for the Green's function for
the linearized Navier-Stokes-Poisson equations, which has been considered in Wang and Wu \cite{Wang} by using real analysis method. Here we do not reconsider this term by using the complex method as in \cite{LS} for the Navier-Stokes system for simplicity.

Recall the linearized system on $(n_2,w_2)$ in (\ref{2.4}) with $c=P_1'(\bar{\rho_1})=P_2'(\bar{\rho_2})$
\bes
\left\{\begin{array}{l}
\partial_tn_2+{\rm div}w_2=0,\\
\partial_tw_2+c^2\nabla n_2-\mu_1\Delta w_2-\mu_2\nabla{\rm div}w_2+2\nabla\phi=0,\\
\Delta\phi=n_2.
\end{array}\right.
\lbl{4.1}
\ees

The symbol of the operator in (\ref{4.1}) is
\bes
\lambda^2+(\mu_1+\mu_2)|\xi|^2\lambda+(c^2|\xi|^2+2)=0.\lbl{4.2}
\ees
Here, $\lambda$ and $\xi^\tau=(\xi_1,\xi_2,\cdots,\xi_n)$ correspond to $\frac{\partial}{\partial_t}$ and $(D_{x_1},\cdots,D_{x_n})$ respectively, where $D_{x_j}=\frac{1}{\sqrt{-1}\partial/\partial_{x_j}}$ with $j=1,2,\cdots,n$. It is easy to find that the eigenvalues of (\ref{4.2}) for $\lambda$ are
\bes
\lambda=\lambda_{\pm}(\xi)=\frac{-\mu|\xi|^2\pm\sqrt{\mu^2|\xi|^4-4(c^2|\xi|^2+2)}}{2},\lbl{4.3}
\ees
where $\mu=\mu_1+\mu_2$.

Now we consider the Green's function for (\ref{4.1}) defined as
\bes
\left\{\begin{array}{l}
\left(\frac{\partial}{\partial t}+A(D_x)\right)\mathbb{G}(x,t)=0,\\
\mathbb{G}(x,0)=\delta(x),
\end{array}\right.
\lbl{4.4}
\ees
where $\delta(x)$ is the Dirac function, the symbols of operator $A(D_x)$ are
\bess
A(\xi)=\left(
           \begin{array}{cc}
             0 & \sqrt{-1}\xi^\tau \\
             \sqrt{-1}\xi(c^2+\frac{1}{|\xi|^2}) & \mu_1|\xi|^2+\mu_2\xi\xi^\tau \\
           \end{array}
         \right).
\eess
Applying the Fourier transform with respect to the variable $x$ to (\ref{4.4}), we get by a direct calculation
\bes
\begin{array}[b]{rl}
\hat{\mathbb{G}}(\xi,t)\!=&\!\left(\!
                 \begin{array}{cc}
                   \hat{\mathbb{G}}_{11} & \hat{\mathbb{G}}_{12}  \\
                   \hat{\mathbb{G}}_{21} & \hat{\mathbb{G}}_{22} \\
                 \end{array}
               \!\right)
               \!=\!\left(\!
                 \begin{array}{cc}
                   \frac{\lambda_+e^{\lambda_-t}-\lambda_-e^{\lambda_+t}}{\lambda_+-\lambda_-} & -\sqrt{-1}\frac{e^{\lambda_+t}-e^{\lambda_-t}}{\lambda_+-\lambda_-}\xi^\tau  \\
                   -\sqrt{-1} \frac{e^{\lambda_+t}-e^{\lambda_-t}}{\lambda_+-\lambda_-}\xi & e^{-\mu_1|\xi|^2t}I\!+\!\left(\frac{\lambda_+e^{\lambda_+t}-\lambda_-e^{\lambda_-t}}{\lambda_+-\lambda_-}
\!-\!e^{-\mu|\xi|^2t}\right)\frac{\xi\xi^\tau}{|\xi|^2}
                 \end{array}
               \!\right)\\[1mm]
               =&\!\left(\!
                 \begin{array}{cc}
                   \eta_+e^{\lambda_-t}-\eta_-e^{\lambda_+t} & -\sqrt{-1}\eta_0\xi^\tau(-e^{\lambda_+t}+e^{\lambda_-t})  \\
                   -\sqrt{-1} \eta_0\xi(-e^{\lambda_+t}+e^{\lambda_-t}) & e^{-\mu_1|\xi|^2t}(I-\frac{\xi\xi^\tau}{|\xi|^2})\!+\!\frac{\xi\xi^\tau}{|\xi|^2}(\eta_+e^{\lambda_+t}-\eta_-e^{\lambda_-t})
                 \end{array}
               \!\right),
               \end{array} \lbl{4.5}
\ees
where $\eta_0(\xi)=(\lambda_+(\xi)-\lambda_-(\xi))^{-1},
\eta_\pm(\xi)=\lambda_{\pm}(\xi)\eta_0(\xi)$.

Sometime, we also rewrite the Fourier transform of the Green's function as
follows
$$
\hat{\mathbb{G}}^+=\left(
            \begin{array}{cc}
              -\eta_- & -\sqrt{-1}\eta_0\xi^\tau \\
              -\sqrt{-1}(c^2-|\xi|^{-2})\eta_0\xi & \eta_+\frac{\xi\xi^\tau}{|\xi|^2}  \\
            \end{array}
          \right)e^{\lambda_+t},
$$
$$
\hat{\mathbb{G}}^-=\left(
            \begin{array}{cc}
              -\eta_+e^{\lambda_-t} & \sqrt{-1}\eta_0\xi^\tau e^{\lambda_-t} \\
              \sqrt{-1}(c^2-|\xi|^{-2})\eta_0\xi e^{\lambda_-t} & \eta_-e^{\lambda_-t}\frac{\xi\xi^\tau}{|\xi|^2}+(I-\frac{\xi\xi^\tau}{|\xi|^2})e^{-\mu_1|\xi|^2t}  \\
            \end{array}
          \right),
$$
$$
\hat{\mathbb{G}}^0=\left(
            \begin{array}{cc}
              0 & 0 \\
              0 & I-\frac{\xi\xi^\tau}{|\xi|^2}  \\
            \end{array}
          \right)e^{-\mu_1|\xi|^2t},
$$

\begin{lemma}\lbl{l 41} \cite{Wang}
For a sufficiently small number $\varepsilon_1>0$ and a sufficiently large number $K>0$, we have the following:

(i) when $|\xi|<\varepsilon_1$, $\lambda_\pm$ has the following expansion:
\bess
\lambda_+&=&-\frac{\mu}{2}|\xi |^2+\sum_{j=2}^{\infty}a_j|\xi|^{2j}+\sqrt{-1}(\sqrt{2}+\sum_{j=1}^{\infty}b_j|\xi|^{2j}),\\
\lambda_-&=&-\frac{\mu}{2}|\xi |^2+\sum_{j=2}^{\infty}a_j|\xi|^{2j}-\sqrt{-1}(\sqrt{2}+\sum_{j=1}^{\infty}b_j|\xi|^{2j});
\eess

 (ii) when $\varepsilon_1 \le |\xi|\le K$, $\lambda_\pm$ has the following spectrum gap property:
 \bess
 Re(\lambda_\pm)\le-C,\ \  \mbox{for some constant} \ \ C>0;
 \eess

(iii) when $|\xi|>K$,
 $\lambda_\pm$ has the following expansion:
 \bess
\lambda_+&=&-\frac{c^2}{\mu}+\sum_{j=1}^{\infty}c_j|\xi|^{-2j},\\
\lambda_-&=&-\mu|\xi|^2+\frac{c^2}{\mu}-\sum_{j=1}^{\infty}c_{j}|\xi|^{-2j}.
\eess
Here $c=\sqrt{P_i'(\bar{\rho})}$ is defined as the equilibrium sound speed, and all $a_j,b_j,c_j$ are real constants.
 \end{lemma}

From Lemma \ref{l 41}, we have
$$
\eta_+(\xi)=1/2+\sqrt{-1}O(|\xi|^2),\ \ \eta_-(\xi)=-1/2+\sqrt{-1}O(|\xi|^2),
\ \ \eta_0(\xi)=-\sqrt{-1}\frac{\sqrt{2}}{4}+\sqrt{-1}O(|\xi|^2).
$$
Thus, we get
\bes
\chi_1(\xi)\widehat{\mathbb{G}}^{+}=\left(
            \begin{array}{cc}
              \frac{1}{2}+\sqrt{-1}O(|\xi|^2) & -\frac{1}{2}\xi^\tau+O(|\xi|^3) \\
              -\frac{1}{2}(c^2+\frac{2}{|\xi|^2})\xi+O(|\xi|^3) & \frac{1}{2}\xi\xi^\tau/|\xi|^2+O(|\xi|^3) \\
            \end{array}
          \right)e^{-\frac{\mu}{2}|\xi|^2t+O(|\xi|^3)t}, \lbl{4.6}
\ees
\bes
\chi_1(\xi)\widehat{\mathbb{G}}^{-}=\left(
            \begin{array}{cc}
              -\frac{1}{2}+\sqrt{-1}O(|\xi|^2) & -\frac{1}{2}\xi^\tau+O(|\xi|^3) \\
             -\frac{1}{2}(c^2+\frac{2}{|\xi|^2})\xi+O(|\xi|^3) & -\frac{3}{2}\xi\xi^\tau/|\xi|^2+I+O(|\xi|^3)\\
            \end{array}
          \right)e^{-\frac{\mu}{2}|\xi|^2t+O(|\xi|^3)t}, \lbl{4.7}
\ees
and
\bes
\chi_1(\xi)\widehat{\mathbb{G}}^{0}=\left(
            \begin{array}{cc}
              0 & 0 \\
             0 & I-\frac{\xi\xi^\tau}{|\xi|^2}\\
            \end{array}
          \right)e^{-\mu_1|\xi|^2t}. \lbl{4.8}
\ees

By differentiating both sides of $\chi_1(\xi)\widehat{\mathbb{G}}_{11}^+$ and $\chi_1(\xi)\widehat{\mathbb{G}}_{12}^-$
with respect to $\xi$, because of the smoothness of
$\chi_1(\xi)\widehat{\mathbb{G}}_{11}^+$ and $\chi_1(\xi)\widehat{\mathbb{G}}_{11}^-$ near $|\xi|=0$, we can
immediately obtain the following lemma.

\begin{lemma}\lbl{l 42}  If $|\xi|$ is sufficiently small, then there exists a constant
$b>0$, such that
\bes
&&|D_\xi^\beta(\xi^\alpha\chi_1(\xi)\hat{\mathbb{G}}_{11}^+(\xi,t))|+|D_\xi^\beta(\xi^\alpha\chi_1(\xi)\hat{\mathbb{G}}_{11}^-(\xi,t))|\nm\\
&\leq&
C(|\xi|^{(|\alpha|-|\beta|)_+}+|\xi|^{|\alpha|}t^{|\beta|/2})(1+t|\xi|^2)^{|\beta|+1}e^{-b|\xi|^2t},\nm\\
&&|D_\xi^\beta(\xi^\alpha\chi_1(\xi)\hat{\mathbb{G}}_{12}^+(\xi,t))|+|D_\xi^\beta(\xi^\alpha\chi_1(\xi)\hat{\mathbb{G}}_{12}^-(\xi,t))|\nm\\
&\leq&
C(|\xi|^{(|\alpha|-|\beta|+1)_+}+|\xi|^{|\alpha|+1}t^{|\beta|/2})(1+t|\xi|^2)^{|\beta|+1}e^{-b|\xi|^2t}.
\ees
\end{lemma}
Then, we will use the following key estimates for the Green's function like a Heat kernel.

\begin{lemma}\lbl{l 43}
If $\hat{f}(\xi,t)$ has compact support on
$\xi$ in $\mathbb{R}^n$, and satisfies for a constant $b>0$ that
$$
|D_\xi^\beta(\xi^\alpha\hat{f}(\xi,t)|\leq
C(|\xi|^{(|\alpha|+k-|\beta|)_+}
+|\xi|^{|\alpha|+k}t^{|\beta|/2})(1+(t|\xi|^2))^m \exp(-b|\xi|^2t),
$$
for any multi-indexes $\alpha, \beta$ with
$|\beta|\leq 2N$\ (the integer $N$ could be arbitrary large), then
$$
|D_x^\alpha f(x,t)|\leq C_N (1+t)^{-\frac{n+|\alpha|+k}{2}}\bigg(1+\frac{|x|^2}{1+t}\bigg)^{-N},
$$
where $k$ and $m$ are any fixed integers and $(a)_+=\max(0,a)$.
\end{lemma}

\begin{proof} If $|\beta|<k+|\alpha|$, then we have by
direct calculation that
\bess
\begin{array}[b]{rl}
|x^\beta D_x^\alpha f(x,t)|=&\D C\bigg|\int e^{\sqrt{-1}x\cdot\xi}D^\beta\xi^\alpha\hat{f}(\xi,t)d\xi\bigg|\\
\leq &\D C\int|\xi|^{|\alpha|+k}(|\xi|^{-|\beta|}+t^{|\beta|/2})(1+(t|\xi|^2))^me^{-b|\xi|^2(1+t)}e^{b|\xi|^2}d\xi\\
\leq &\D C(1+t)^{-(|\alpha|+n+k-|\beta|)/2}.
\end{array}
\eess
If $|\beta|\geq k+|\alpha|$, one also can find that
\bess
\begin{array}[b]{rl}
|x^\beta D_x^\alpha f(x,t)|=&\D C|\int e^{\sqrt{-1}x\cdot\xi}D^\beta\xi^\alpha\hat{f}(\xi,t)d\xi|\\
\leq &\D C\int(|\xi|^{(|\alpha|+k-|\beta|)_+}+|\xi|^{|\alpha|+k}t^{|\beta|/2}))(1+(t|\xi|^2))^me^{-b|\xi|^2t}d\xi\\
= &\D C\int(|\xi|^{(|\alpha|+k-|\beta|)_+}+|\xi|^{|\alpha|+k}t^{|\beta|/2}))(1+(t|\xi|^2))^me^{-b|\xi|^2(1+t)}e^{b|\xi|^2}d\xi\\
\leq &\D C(1+t^{-(|\alpha|+k-|\beta|)/2})(1+t)^{-n/2}\\
=&\D C(1+t)^{|\beta|/2}(1+t)^{-(|\alpha|+n+k)/2}\bigg(\frac{t}{1+t}\bigg)^{(|\alpha|+k)/2}\\
\leq &\D C(1+t)^{|\beta|/2}(1+t)^{-(|\alpha|+n+k)/2}
\end{array}
\eess

Let $|\beta|=0$ when $|x|^2\leq 1+t$, and $|\beta|=2N$ when
$|x|^2>1+t$, we get
$$
|D_x^\alpha f(x,t)|\leq C(1+t)^{-(|\alpha|+n+k)/2}\min\bigg\{1,\bigg(\frac{1+t}{|x|^2}\bigg)^N\bigg\}.
$$
Since
$$
1+\frac{|x|^2}{1+t}\leq 2\left\{\begin{array}{l}1,\ \ \ \ \ \ \ |x|^2\leq1+t,\\
\frac{|x|^2}{1+t},\ \ \ |x|^2>1+t,
\end{array}
\right.
$$
we have
$$
\min\bigg\{1,\bigg(\frac{1+t}{|x|^2}\bigg)^N\bigg\}\leq \frac{2^N}{(1+\frac{|x|^2}{1+t})^N}=2^N\bigg(1+\frac{|x|^2}{1+t}\bigg)^{-N}.
$$
This completes the proof.
\end{proof}

Notice that when $\hat{f}(\xi,t)$ has not compact support in
$\xi$, the result in Lemma \ref{l 43} should be
\bes
|D_x^\alpha f(x,t)|\leq C_N t^{-\frac{n+|\alpha|+k}{2}}\bigg(1+\frac{|x|^2}{1+t}\bigg)^{-N}.\lbl{4.9}
\ees

From Lemma \ref{l 42} and \ref{l 43}, we can immediately give the following estimates for $\chi_1(D)\widehat{\mathbb{G}}_{11}$ and $\chi_1(D)\widehat{\mathbb{G}}_{12}$.

\begin{proposition}\lbl{l 44} When $|\xi|$ is sufficiently small, we have for any $|\alpha|\geq0$
and any integer $N>0$,
\bes
&&|D_x^\alpha (\chi_1(D)\mathbb{G}_{11}(x,t))|\leq C
(1+t)^{-(n+|\alpha|)/2}\bigg(1+\frac{|x|^2}{1+t}\bigg)^{-N}, \\
&&|D_x^\alpha (\chi_1(D)\mathbb{G}_{12}(x,t))|\leq C
(1+t)^{-(n+1+|\alpha|)/2}\bigg(1+\frac{|x|^2}{1+t}\bigg)^{-N}.
\ees
\end{proposition}

Two remain terms $\chi_1(D)\hat{\mathbb{G}}_{21}(\xi,t)$ and
$\chi_1(D)\hat{\mathbb{G}}_{22}(\xi,t)$ containing Calderon-Zygmund operators $R_{i}$ and  $R_{ij}$
with the symbols $\frac{\xi_i}{|\xi|^2}$ and $\frac{\xi_i\xi_j}{|\xi|^2}$ respectively can be estimated by the following two lemmas.

\begin{lemma}\lbl{l 45} \cite{HZ}
When $|\xi|$ is sufficiently small in $\mathbb{R}^n$, it holds that
$C>0$ that
\bess
 |D_x^\alpha (R_{ij}\ast \chi_1(D)H(x,t))|\leq C(\alpha)(1+t)^{-\frac{n+|\alpha|}{2}}\bigg(1+\frac{|x|^2}{1+t}\bigg)^{-\frac{n+|\alpha|}{2}},
\eess
 where $H(x,t)$ is the heat kernel.
\end{lemma}

Next, for the term $R_i(x,t)$ in $\chi_1(D)\widehat{\mathbb{G}}_{21}$, by using the same method in Lemma \ref{l 44}, we have the following.

\begin{lemma}\lbl{l 46}
When $|\xi|$ is sufficiently small and $x\in\mathbb{R}^n$, we have for some
$C>0$ that
 \bess
 |D_x^\alpha (R_{i}\ast \chi_1(D)H(x,t))|\leq C(1+t)^{-\frac{n-1+|\alpha|}{2}}\bigg(1+\frac{|x|^2}{1+t}\bigg)^{-\frac{n-1+|\alpha|}{2}}.
\eess
 \end{lemma}

\noindent{\bf Remark 4.1.} This estimate is crucial to help us deduce the asymptotic profile $\big(1+\frac{|x|^2}{1+t}\big)^{-\frac{n}{2}}$ in the pointwise estimates, which is an improvement of those in \cite{Wang} for the Navier-Stokes-Poisson system in $\mathbb{R}^n$, where the asymptotic profile for the velocity (or the momentum) is $(1+t)^{-\frac{n-1}{2}}\big(1+\frac{|x|^2}{1+t}\big)^{-\frac{n-1}{2}}$.

Hence, we have the following pointwise estimates for $\chi_1(D)\mathbb{G}_{21}$ and $\chi_1(D)\mathbb{G}_{22}$.

 \begin{proposition}\lbl{l 47}
When $|\xi|$ is sufficiently small in $\mathbb{R}^n$, we have for any $|\alpha|\geq0$ that
\bess
&&|D_x^\alpha (\chi_1(D)\mathbb{G}_{21}(x,t))|\leq C
(1+t)^{-\frac{n-1+|\alpha|}{2}}\bigg(1+\frac{|x|^2}{1+t}\bigg)^{-\frac{n-1+|\alpha|}{2}}, \nm\\
&&|D_x^\alpha (\chi_1(D)\mathbb{G}_{22}(x,t))|\leq C
(1+t)^{-\frac{n+|\alpha|}{2}}\bigg(1+\frac{|x|^2}{1+t}\bigg)^{-\frac{n+|\alpha|}{2}}.
\lbl{2.19}
\eess
\end{proposition}

The following proposition gives the pointwise estimates of the Green's function in the middle part, which is a standard result for the pointwise estimates as in \cite{Wang, Wu2}.

\begin{proposition}\lbl{l 48} For fixed $\varepsilon$ and $R$ defined in the cut-off functions, there exists a positive constant $c_0$ and $C$ such that
$$
|D_x^\alpha (\chi_2(D_x)\mathbb{G}(x,t))|\leq Ce^{-c_0t}\bigg(1+\frac{|x|^2}{1+t}\bigg)^{-N},\ {\rm for\ any\ integer}\ N>0.
$$
\end{proposition}

Next, it is absolutely the same as in \cite{LS,LW} for the short wave of the Green's function for the Navier-Stokes system, one can immediately get the following result.

\begin{proposition}\lbl{l 49} For $|\xi|$ being sufficiently large, there exists distribution $\mathbb{G}_{S_2}$ such that
 for some $c_0>0$
$$
|D_x^\alpha(\chi_3(D)(\mathbb{G}(x,t)-\mathbb{G}_{S_2}(x,t)))|\leq Ce^{-c_0t}\bigg(1+\frac{|x|^2}{1+t}\bigg)^{-N}.
$$
\end{proposition}
\ \ \ \ \ In summary, we have the following pointwise results for the Green's function $\mathbb{G}(x,t)$.
\begin{proposition}\lbl{l 410} For any $|\alpha|\geq0$ and $x\in\mathbb{R}^3$, we have
\bes
|D_x^\alpha(\mathbb{G}_{11}-\mathbb{G}_{S_2})|\leq C(1+t)^{-\frac{3+|\alpha|}{2}}\bigg(1+\frac{|x|^2}{1+t}\bigg)^{-N}+e^{-\frac{t}{2C}}t^{-\frac{3+|\alpha|}{2}}\bigg(1+\frac{|x|^2}{1+t}\bigg)^{-N},\\
|D_x^\alpha(\mathbb{G}_{12}-\mathbb{G}_{S_2})|\leq C(1+t)^{-\frac{4+|\alpha|}{2}}\bigg(1+\frac{|x|^2}{1+t}\bigg)^{-N}+e^{-\frac{t}{2C}}t^{-\frac{3+|\alpha|}{2}}\bigg(1+\frac{|x|^2}{1+t}\bigg)^{-N},\\
|D_x^\alpha(\mathbb{G}_{21}-\mathbb{G}_{S_2})|\leq C(1+t)^{-\frac{2+|\alpha|}{2}}\bigg(1+\frac{|x|^2}{1+t}\bigg)^{-\frac{2+|\alpha|}{2}}+e^{-\frac{t}{2C}}t^{-\frac{3+|\alpha|}{2}}\bigg(1+\frac{|x|^2}{1+t}\bigg)^{-N},\\
|D_x^\alpha(\mathbb{G}_{22}-\mathbb{G}_{S_2})|\leq C(1+t)^{-\frac{3+|\alpha|}{2}}\bigg(1+\frac{|x|^2}{1+t}\bigg)^{-\frac{3+|\alpha|}{2}}+e^{-\frac{t}{2C}}t^{-\frac{3+|\alpha|}{2}}\bigg(1+\frac{|x|^2}{1+t}\bigg)^{-N},
\ees
where $\mathbb{G}_{S_2}$ is defined in Proposition \ref{l 49} and the integer $N$ could be arbitrary large. Similar to Proposition \ref{l 312}, we also denote $\mathbb{G}_{S_1}=e^{-\frac{t}{2C}}t^{-\frac{3+|\alpha|}{2}}(1+\frac{|x|^2}{1+t})^{-N}$.
\end{proposition}

\section{Pointwise estimates for nonlinear system}\ \

\quad This section devotes to the pointwise estimates of the
solution to the nonlinear system. The following lemma will be used to estimate the convolutions between the Green's function and the initial data.

\begin{lemma}\lbl{l 51}
There exists a constant $C>0$, such that
 \bess
\mathcal{I}_1&:=&\int_{\mathbb{R}^3} e^{-\frac{|x-y|^2}{C(1+t)}}(1+|y|^2)^{-r_1}dy\leq C\left(1+\frac{x^2}{1+t}\right)^{-r_1},\ {\rm for}\ r_1>\frac{3}{2};\\
\mathcal{I}_2&:=&\int_{\mathbb{R}^3}\left(1+\frac{|x-y|^2}{1+t}\right)^{-\frac{3}{2}}(1+|y|^2)^{-r_1}dy\leq C\left(1+\frac{x^2}{1+t}\right)^{-\frac{3}{2}},\ {\rm for}\ r_1>\frac{3}{2};\\
\mathcal{I}_3&:=&\int_{\mathbb{R}^3} e^{-\frac{(|x-y|-ct)^2}{C(1+t)}}(1+|y|^2)^{-r_1}dy\leq C\left(1+\frac{(|x|\!-\!ct)^2}{1+t}\right)^{-(\frac{3}{2}-\varepsilon)},\ {\rm for}\ r_1\geq\frac{21}{10}.
\eess
\end{lemma}
\begin{proof}
We only prove $\mathcal{I}_3$. When $(|x|-ct)^2\leq 4(1+t)$,
$$
\mathcal{I}_3\leq C\leq C\left(1+\frac{(|x|-ct)^2}{1+t}\right)^{-r_1}.
$$
When $(|x|-ct)^2\geq 4(1+t)$, we break integration into two parts. If $|y|\geq\frac{||x|-ct|}{2}$,
\bess
\mathcal{I}_3&\leq & C(1+(|x|-ct)^2)^{-(\frac{3}{2}-\varepsilon)}\int_{\mathbb{R}^3} e^{-\frac{(|x-y|-ct)^2}{C(1+t)}}(1+|y|^2)^{-(r_1-\frac{3}{2}+\varepsilon)}dy\nm\\
&\leq& C(1+(|x|-ct)^2)^{-(\frac{3}{2}-\varepsilon)}(1+t)^{\frac{5}{2p}}\left(\int_{\mathbb{R}^3}(1+|y|^2)^{-\frac{p}{p-1}(r_1-\frac{3}{2}+\varepsilon)}dy\right)^{\frac{p-1}{p}}\nm\\
&\leq& C(1+(|x|-ct)^2)^{-(\frac{3}{2}-\varepsilon)}(1+t)^{(\frac{3}{2}-\varepsilon)}\left(\int_{\mathbb{R}^3}(1+|y|^2)^{-\frac{5}{2}(r_1-\frac{3}{2}+\varepsilon)}dy\right)^{\frac{2}{5}}\nm\\
&\leq& C\left(1+\frac{(|x|\!-\!ct)^2}{1+t}\right)^{-(\frac{3}{2}-\varepsilon)},\ {\rm since}\ r_1\geq\frac{21}{10}.
\eess
Here we have used Young's inequality with $p=\frac{5}{3}$ and the fact
\bes
\int_{\mathbb{R}^n}\left(1+\frac{(|y|-a)^2}{b}\right)^{-N}dy\leq C(b^{\frac{n}{2}}+b^{\frac{1}{2}}a^{n-1}),\ {\rm for}\ N>\frac{n}{2}.
\lbl{5.0}
\ees
If $|y|<\frac{||x|-ct|}{2}$, then $|x-y|-ct\geq\frac{||x|-ct|}{2}$. Thus, it holds that
\bess
\mathcal{I}_3\leq Ce^{-\frac{(|x|-ct)^2}{2C(1+t)}}\int_{\mathbb{R}^3} e^{-\frac{(|x-y|-ct)^2}{2C(1+t)}}(1+|y|^2)^{-r_1}dy
\leq C\left(1+\frac{(|x|\!-\!ct)^2}{1+t}\right)^{-\frac{3}{2}}.
\eess
This proves the third estimate.
\end{proof}

\newpage

We begin to study the pointwise estimates of the solution $(n_1,w_1,n_2,w_2)$ for (2.4). By Duhamel's principle,
the solution $(n_1,w_1)$ can be expressed as
\bes
 D_x^\alpha\left(\!\!
            \begin{array}{c}
              \!\!n_1 \\
              w_1 \\
            \end{array}
          \!\!\right)\!=\!D_x^\alpha \left(\!
                            \begin{array}{cc}
                              G_{11}\! & \!G_{12} \\
                              G_{21}\! & \!G_{22} \\
                            \end{array}
                          \!\right)\!\!\ast\!\!\left(\!\!
            \begin{array}{c}
              n_{1,0} \\
              w_{1,0} \\
            \end{array}
          \!\!\right)\!\!+\!\!\displaystyle\int_0^t\!D_x^\alpha\left(\!\!
                            \begin{array}{cc}
                              G_{11}\! & \!G_{12} \\
                              G_{21}\! & \!G_{22} \\
                            \end{array}
                          \!\!\right)(\cdot,t\!-\!s)\!\ast\!\left(\!\!
                                            \begin{array}{c}
                                              0 \\
                                              F_1(n_1,w_1,n_2,w_2)\! \\
                                            \end{array}
                                          \!\!\right)\!(\cdot,s)ds\ \ \ \ \ \
\lbl{5.1}
\ees
and $(n_2,w_2)$ can be expressed as
\bes
 \begin{array}[b]{ll}
 D_x^\alpha\left(\!\!
            \begin{array}{c}
              n_2 \\
              w_2 \\
            \end{array}
          \!\!\right)\!=\!D_x^\alpha \left(\!\!
                            \begin{array}{cc}
                              \mathbb{G}_{11}\! & \!\mathbb{G}_{12} \\
                              \mathbb{G}_{21}\! & \!\mathbb{G}_{22} \\
                            \end{array}
                          \!\!\right)\!\ast\!\left(\!\!
            \begin{array}{c}
              n_{2,0} \\
              w_{2,0} \\
            \end{array}
          \!\right)\!\!+\!\!\displaystyle\int_0^t\!\!D_x^\alpha\left(\!\!\!
                            \begin{array}{cc}
                              \mathbb{G}_{11} & \mathbb{G}_{12} \\
                              \mathbb{G}_{21} & \mathbb{G}_{22} \\
                            \end{array}
                          \!\!\right)(\cdot,t\!-\!s)\!\ast\!\left(\!\!
                                            \begin{array}{c}
                                              0 \\
                                              F_2(n_1,w_1,n_2,w_2) \\
                                            \end{array}
                                          \!\!\!\right)\!(\cdot,s)ds.\ \ \ \ \ \
\end{array}
\lbl{5.2}
\ees

Without loss of generality, we assume that for $|\alpha|\leq 2$
\bes
&&|D_x^\alpha(n_{1,0},w_{1,0})|\leq C\varepsilon_0\big(1+|x|^2\big)^{-r_1}, \ r_1\geq\frac{21}{10},\nm\\
&&|D_x^\alpha(n_{2,0},w_{2,0})|+|\nabla\phi_0|\leq C\varepsilon_0\big(1+|x|^2\big)^{-r_2}, \ r_2>\frac{3}{2}.
\lbl{5.3}
\ees
Here $\varepsilon_0$ is sufficiently small.

Now we first study the pointwise estimates for the solution $(\overline{n_1},\overline{w_1})$ to the linearized system of  (\ref{2.4})$_{1,2}$. From (\ref{3.5}), we have
\bes
D^\alpha \overline{n_1}=  D_x^\alpha G_{1,1}(t)\ast n_{1,0}+D_x^\alpha G_{1,2}\ast w_{1,0}
:=  R_1+R_2.
\lbl{5.5}
\ees

We rewrite $R_1$ as
\bes
R_1=D_x^\alpha(G_{11}-G_{S_1}-G_{S_2})\ast n_{1,0}+D_x^\alpha G_{S_1}\ast  n_{1,0}+D_x^\alpha G_{S_2}\ast  n_{1,0}:=R_1^1+R_1^2+R_1^3,\lbl{5.6}
\ees
where the Delta-like function $G_{S_2}$ is generated from the higher frequency part in Green's function defined in Lemma \ref{l 313}, and $G_{S_1}$ has exponential decay rate and is only singular as $t\rightarrow0$.

From Lemma \ref{l 51}, the assumption (\ref{5.3}) and Proposition \ref{l 317},  we have for $R_1^1$ that
\bes
R_1^1\leq C(1+t)^{-\frac{4+|\alpha|}{2}}\bigg(1+\frac{(|x|-ct)^2}{1+t}\bigg)^{-(\frac{3}{2}-\varepsilon)}.\lbl{5.7}
\ees

For $R_1^3$, when $|x|^2\leq 1+t$, from Lemma \ref{l 313}, we have
\bes
\left|\int G_{S_2}(x-y)n_{1,0}(y)dy\right|&\leq& e^{-\frac{c^2t}{\mu}}n_{1,0}(x)+\int_{|x-y|\leq R}\frac{Ce^{-bt}}{|x-y|^2}(1+|y|^2)^{-\frac{21}{10}}dy\nm\\
&&+\int_{|x-y|>R}\frac{C(N)e^{-bt}}{|x-y|^N}(1+|y|^2)^{-\frac{21}{10}}dy\nm\\ \ \ \ \ \
&\leq& Ce^{-C_1t}\leq Ce^{-C_1t}\bigg(1+\frac{|x|^2}{1+t}\bigg)^{-\frac{3}{2}}.
\ees
When $|x|^2>1+t$, we have
\bes
\left|\int G_{S_2}(x-y)n_{1,0}(y)dy\right|&\leq& e^{-\frac{c^2t}{\mu}}n_{1,0}(x)+\int_{|x-y|\leq \frac{|x|}{2}}\frac{Ce^{-bt}}{|x-y|^2}(1+|y|^2)^{-\frac{21}{10}}dy\nm\\
&&+\int_{|x-y|>\frac{|x|}{2}}\frac{C(N)e^{-bt}}{|x-y|^N}(1+|y|^2)^{-\frac{21}{10}}dy\nm\\
&\leq& Ce^{-\frac{c^2t}{\mu}}(1+|x|^2)^{-\frac{21}{10}}+Ce^{-bt}\bigg(1+\frac{|x|^2}{4}\bigg)^{-\frac{21}{10}}+C(N)e^{-bt}(1+|x|)^{-N}\nm\\ \ \ \
&\leq& Ce^{-C_1t}\bigg(1+\frac{|x|^2}{1+t}\bigg)^{-\frac{21}{10}}.
\ees

For $R_1^2$, we divide the time interval into $0\leq t\leq 1$ and $t>1$. When $t>1$, there is not singularity in $R_1^2$, so one can deal with this term as for $R_1^1$. When $0\leq t\leq 1$ and $|x|\leq 1+t$, by using $L^2$-energy estimate for the existence result and Sobolev inequality, we have
\bes
|R_1^2|=|G_{S_2}\ast  D_x^\alpha n_{1,0}|\leq C\leq Ce^{-|x|-t}\leq Ce^{-C_1t}\bigg(1+\frac{|x|^2}{1+t}\bigg)^{-\frac{3}{2}}.
\ees
When $t>1$ and $|x|>1+t$,
\bes
|R_1^2|&=&|G_{S_1}\ast  D_x^\alpha n_{1,0}|\leq e^{-\frac{t}{2C}}{t^{-3/2}\int_{\mathbb{R}^3} e^{-\frac{|x-y|^2}{2Ct}}}(1+|y|^2)^{-\frac{21}{10}}dy\nm\\
&=&e^{-\frac{t}{2C}}t^{-3/2}\left\{\int_{|y|\geq\frac{|x|}{2}} e^{-\frac{|x-y|^2}{2Ct}}(1+|y|^2)^{-\frac{21}{10}}dy\nm+\int_{|y|<\frac{|x|}{2}} e^{-\frac{|x-y|^2}{2Ct}}(1+|y|^2)^{-\frac{21}{10}}dy\right\}\nm\\
&\leq& Ce^{-\frac{t}{2C}}(1+|x|^2)^{-\frac{21}{10}}+Ce^{-\frac{t}{2C}}e^{-\frac{|x|^2}{2Ct}}\nm\\
&\leq& Ce^{-C_1t}\bigg(1+\frac{|x|^2}{1+t}\bigg)^{-\frac{3}{2}}.
\ees
Then, we have completed the estimate of $R_1$ as
\bes
|R_1|=|D_x^\alpha G_{1,1}(t)\ast n_{1,0}|\leq C(1+t)^{-\frac{4+|\alpha|}{2}}\Bigg\{\bigg(1+\frac{(|x|-ct)^2}{1+t}\bigg)^{-(\frac{3}{2}-\varepsilon)}
+\bigg(1+\frac{|x|^2}{1+t}\bigg)^{-\frac{3}{2}}\Bigg\},
\lbl{5.8}
\ees
synchronously, we have the same estimate for $R_2$.

As a result, we have
\bes
|D_x^\alpha \overline{n_1}|\leq C\varepsilon_0(1+t)^{-\frac{4+|\alpha|}{2}}\Bigg\{\bigg(1+\frac{(|x|-ct)^2}{1+t}\bigg)^{-(\frac{3}{2}-\varepsilon)}
+\bigg(1+\frac{|x|^2}{1+t}\bigg)^{-\frac{3}{2}}\Bigg\}.\lbl{5.11}
\ees
Similarly, we also have
\bes
|D_x^\alpha \overline{w_1}|\leq C\varepsilon_0\left\{(1+t)^{-\frac{4+|\alpha|}{2}}\bigg(1+\frac{(|x|-ct)^2}{1+t}\bigg)^{-(\frac{3}{2}-\varepsilon)}
+(1+t)^{-\frac{3+|\alpha|}{2}}\bigg(1+\frac{|x|^2}{1+t}\bigg)^{-\frac{3}{2}}\right\}.\lbl{5.12}
\ees

Let us emphasize here that the method above on dealing with the convolution between
the short wave component of the Green's function and the nonlinear terms is standard, see \cite{Wang,Wu1} and the references therein. Thus, in the following we mainly deal with the convolution between nonlinear terms and the leading part of the Green's function $G(x,t)-G_{S_1}-G_{S_2}$.

Next, we consider the $(\overline{n_2},\overline{w_2})$.  Firstly, we have
\bes
D^\alpha \overline{n_2}=  D_x^\alpha \mathbb{G}_{11}(t)\ast n_{2,0}+D_x^\alpha \mathbb{G}_{12}\ast w_{2,0}
:=  R_3+R_4,\\
D^\alpha \overline{w_2}=  D_x^\alpha \mathbb{G}_{21}(t)\ast n_{2,0}+D_x^\alpha \mathbb{G}_{22}\ast w_{2,0}
:=  R_5+R_6.
\lbl{5.14}
\ees

When $|\alpha|=0$, from the assumption
$$
|\nabla\phi_0|\leq C\varepsilon_0\bigg(1+|x|^2\bigg)^{-r}\ {\rm with}\ r>\frac{3}{2},
$$
and the fact $n_{2,0}={\rm div}\nabla\phi_0$, we find that for $|\alpha|\geq0$
\bes
|D_x^\alpha\mathbb{G}_{11}\ast n_{2,0}|&=&|D_x^\alpha(\mathbb{G}_{11}-\mathbb{G}_{S_1}-\mathbb{G}_{S_2})\ast n_{2,0}|+|(\mathbb{G}_{S_1}+\mathbb{G}_{S_2})\ast D_x^\alpha n_{2,0}|\nm\\
&\leq&|D_x^\beta(\mathbb{G}_{11}-\mathbb{G}_{S_1}-\mathbb{G}_{S_2})\ast \phi_0|+C\varepsilon_0e^{-c_0t}\bigg(1+\frac{|x|^2}{1+t}\bigg)^{-\frac{3}{2}}\nm\\
&\leq & C\varepsilon_0(1+t)^{-\frac{4+|\alpha|}{2}}\bigg(1+\frac{|x|^2}{1+t}\bigg)^{-\frac{3}{2}}
+C\varepsilon_0e^{-c_0t}\bigg(1+\frac{|x|^2}{1+t}\bigg)^{-\frac{3}{2}}\nm\\
&\leq & C\varepsilon_0(1+t)^{-\frac{4+|\alpha|}{2}}\bigg(1+\frac{|x|^2}{1+t}\bigg)^{-\frac{3}{2}}, \ {\rm where}\ |\beta|=|\alpha|+1.
\lbl{5.15}
\ees
Similarly, by using Proposition \ref{l 410}, we have
\bes
&&|R_4|=|D_x^\alpha\mathbb{G}_{22}\ast w_{2,0}|\leq C\varepsilon_0(1+t)^{-\frac{4+|\alpha|}{2}}\bigg(1+\frac{|x|^2}{1+t}\bigg)^{-\frac{3}{2}},\\
&&|R_5|+|R_6|\leq C\varepsilon_0(1+t)^{-\frac{3+|\alpha|}{2}}\bigg(1+\frac{|x|^2}{1+t}\bigg)^{-\frac{3}{2}}.
\lbl{5.17}
\ees

Up to now, we have obtained the following pointwise estimates for the linearized system.

\begin{proposition}\lbl{l 52}
For $0\leq |\alpha|\leq2$, we have
\bes
&|D_x^\alpha \overline{n_1}|\leq C\varepsilon_0(1+t)^{-\frac{4+|\alpha|}{2}}\Bigg\{\bigg(1+\frac{(|x|-ct)^2}{1+t}\bigg)^{-(\frac{3}{2}-\varepsilon)}
+\bigg(1+\frac{|x|^2}{1+t}\bigg)^{-\frac{3}{2}}\Bigg\},\nm\\
&|D_x^\alpha \overline{w_1}|\leq C\varepsilon_0(1+t)^{-\frac{3+|\alpha|}{2}}\left\{(1+t)^{-\frac{1}{2}}\bigg(1+\frac{(|x|-ct)^2}{1+t}\bigg)^{-(\frac{3}{2}-\varepsilon)}+\bigg(1+\frac{|x|^2}{1+t}\bigg)^{-\frac{3}{2}}\right\},\nm\\
&|D_x^\alpha \overline{n_2}|\leq C\varepsilon_0(1+t)^{-\frac{4+|\alpha|}{2}}\bigg(1+\frac{(|x|-ct)^2}{1+t}\bigg)^{-\frac{3}{2}},\nm\\
&|D_x^\alpha \overline{w_2}|\leq C\varepsilon_0(1+t)^{-\frac{3+|\alpha|}{2}}\bigg(1+\frac{|x|^2}{1+t}\bigg)^{-\frac{3}{2}}.
\ees
\end{proposition}

As mentioned in \S1, to deduce the pointwise estimates for the solution $(n_1,w_1)$ of the nonlinear problem, we have to first obtain the pointwise estimate of the electric field $\nabla\phi$. To this end, we shall consider the following new system on the variable $n_2$ and $v_2=u_1-u_2$, to avoid dealing with the term $\nabla\phi$ in the nonlinear term of this system. In fact, we can rewrite the system $(2.4)_{3,4}$ as
\bes
\left\{\begin{array}{lll}\medskip
\!\!\partial_tn_2+{\rm div}v_2=-{\rm div}(\frac{n_1v_2+n_2v_1}{2}):=F_3(n_1,v_1,n_2,v_2)\\
\!\!\partial_t v_2\!+\!h'(1)n_2\!-\!\mu_1\Delta v_2\!-\!\mu_2\!\nabla\!{\rm div}v_2\!+\!2\nabla\phi\!=\!\nabla \{h'(1)n_2\!-\!h(1\!+\!\frac{n_1+n_2}{2})\!+\!h(1\!+\!\frac{n_1-n_2}{2})\}\!+\!(v_1\!\cdot\!\nabla)v_2\\[1mm]
  \ \ \ \ \ \ \!+(v_2\!\cdot\!\nabla)v_1\!+\!\frac{\frac{n_1+n_2}{2}}{1\!+\!\frac{n_1+n_2}{2}}(\mu_1\Delta v_2\!+\!\mu_2\nabla{\rm div}v_2) \!+\!\left(\frac{\frac{n_1\!+\!n_2}{2}}{1\!+\!\frac{n_1\!+\!n_2}{2}}\!-\!\frac{\frac{n_1\!-\!n_2}{2}}{1\!+\!\frac{n_1\!-\!n_2}{2}}\right)
(\mu_1\Delta\frac{v_1+v_2}{2}\!+\!\mu_2\nabla{\rm div}\frac{v_1-v_2}{2})\ \ \ \ \ \\
\ \ \ \ \ \ \ \ \ \ \ \ \ \ \ \ \ \ \ \ \ \ \ \ \ \ \ \ \ \ \ \ \ \ \ \ \ \ \ \ \ \ \ \ \ \ \ \ :=F_4(n_1,v_1,n_2,v_2),
\end{array}\right.
\lbl{5.18}
\ees
where $h'(\rho_i)=\frac{P_i'(\rho_i)}{\rho_i}$ and $v_1=u_1+u_2$.\\
Furthermore, we find that
\bes
F_4(n_1,v_1,n_2,v_2)=\mathcal{O}(1)(D(n_1n_2)+v_1Dv_2+v_2Dv_1+n_1D^2v_2+n_2D^2v_2+n_2D^2v_1),
\ees
that is, there is a factor $n_2$ or $v_2$ in each term of $F_3(n_1,v_1,n_2,v_2)$ and $F_4(n_1,v_1,n_2,v_2)$. This fact is key for us to deduce the pointwise estimate of $n_2$ and $v_2$.

Due to Proposition \ref{l 52}, we first introduce the following ansatz for $|\alpha|\leq 2$:
\bes
M(T)&=&\sup\limits_{0\leq t\leq T}\left\{\|D_x^\alpha n_1(\cdot,t)\psi_1^{-1}(\cdot,t)\|_{L^\infty}+\|D_x^\alpha w_1(\cdot,t)\psi_2^{-1}(\cdot,t)\|_{L^\infty}+\|n_2(\cdot,t)\psi_3^{-1}(\cdot,t)\|_{L^\infty}\right.\nm\\
&&\ \ \ \ \ \ \ \ \ +\|w_2(\cdot,t)\psi_4^{-1}(\cdot,t)\|_{L^\infty}+\|D_x^\alpha n_2(\cdot,t)\psi_5^{-1}(\cdot,t)\|_{L^\infty}+\|D_x^\alpha w_2(\cdot,t)\psi_6^{-1}(\cdot,t)\|_{L^\infty}\nm\\
&&\ \ \ \ \ \ \ \ \ +(1+t)^3(\|D^3(n_1,w_1,n_2,w_2)\|_{L^\infty}+\|D^4(n_1,w_1,n_2,w_2)\|_{L^\infty})\},
\lbl{5.20}
\ees
where for any $0<\varepsilon\ll1$ that
\bess
&&\psi_1(x,t)=(1+t)^{-2}\left\{\left(1+\frac{x^2}{1+t}\right)^{-(\frac{3}{2}-\varepsilon)}
+\left(1+\frac{(|x|-ct)^2}{1+t}\right)^{-(\frac{3}{2}-\varepsilon)}\right\},\\
&&\psi_2(x,t)=(1+t)^{-\frac{3}{2}}\left\{\left(1+\frac{x^2}{1+t}\right)^{-(\frac{3}{2}-\varepsilon)}
+(1+t)^{-\frac{1}{2}}\left(1+\frac{(|x|-ct)^2}{1+t}\right)^{-(\frac{3}{2}-\varepsilon)}\right\},\\
&&\psi_3(x,t)=(1+t)^{-2}\bigg(1+\frac{x^2}{1+t}\bigg)^{-\frac{3}{2}},\ \ \psi_4(x,t)=(1+t)^{-\frac{3}{2}}\bigg(1+\frac{x^2}{1+t}\bigg)^{-\frac{3}{2}},\\
&&\psi_5(x,t)=(1+t)^{-2}\bigg(1+\frac{x^2}{1+t}\bigg)^{-(\frac{3}{2}-\varepsilon)},\ \ \psi_6(x,t)=(1+t)^{-\frac{3}{2}}\bigg(1+\frac{x^2}{1+t}\bigg)^{-(\frac{3}{2}-\varepsilon)}.
\eess
In what follows, we mainly prove that $M(T)\leq C$.\\
 In fact,  in \cite{WX} the authors have deduced the following decay rate of the solution $(\rho_1,u_1,\rho_2,u_2)$ when the initial data is in $H^l(\mathbb{R}^3)\cap L^1(\mathbb{R}^3)$ with $l\geq4$ that
\bes
\|D_x^\alpha(\rho_1-1,u_1,\rho_2-1,u_2)\|_{L^2(\mathbb{R}^3)}\leq C(1+t)^{-\frac{3}{4}-\frac{|\alpha|}{2}},\ |\alpha|\leq l.
\ees
After a similar procedure as in \cite{WX}, when the initial data is in $H^6(\mathbb{R}^3)\cap L^1(\mathbb{R}^3)$, one can immediately get the similar results for the higher order derivative of the solution. In addition, by using Sobolev inequality in $\mathbb{R}^3$
 \bess
 ||f||_{L^\infty}\le C||Df||_{L^2}^{\frac{1}{2}}||D^2f||_{L^2}^{\frac{1}{2}},
\eess
we know the ansatz on $D_x^\beta(n_1,w_1,n_2,w_2)$  with $3\leq |\beta|\leq 4$ is reasonable in the present paper. Thus, in the following we mainly focus on the pointwise estimates for the $D_x^\alpha(n_1,w_1,n_2,w_2)$ when $|\alpha|\leq 2$.

From the ansatz (\ref{5.20}) and  the relation $w_1=(\rho_1+1)u_1+(\rho_2+1)u_2$ and $w_2=(\rho_1+1)u_1-(\rho_2+1)u_2$, we have
\bes
|v_1|\leq C(|w_1|+|n_2|), \ |v_2|\leq C(|w_2|+|n_2|),\ |w_1|\leq C(|v_1|+|n_2|), \ |w_2|\leq C(|v_2|+|n_2|).\lbl{5.21}
\ees
In addition, for $1\leq |\alpha|\leq 2$ we get
\bes
&&|D_x^\alpha v_2(x,t)|\leq CM(T)(\psi_3(x,t)+\psi_4(x,t)),\nm\\
&&|D_x^\alpha v_1(x,t)|\leq CM(T)(\psi_2(x,t)+\psi_3(x,t)).\lbl{5.22}
\ees

Now, we begin to derive the pointwise estimates for $n_2$ and $v_2$. by Duhamel's principle, the nonlinear part of $|D_x^\alpha n_2|$ is as follows
\bes
&&\int_0^tD_x^\alpha \mathbb{G}_{11}(\cdot,t-s)\ast F_3(\cdot,s)ds+\int_0^tD_x^\alpha \mathbb{G}_{12}(\cdot,t-s)\ast F_4(\cdot,s)ds\nm\\
&=&\int_0^tD_x^\alpha (\mathbb{G}_{11}-\mathbb{G}_{S_1}-\mathbb{G}_{S_2})(\cdot,t-s))\ast F_3(\cdot,s)ds+\int_0^tD_x^\alpha (\mathbb{G}_{S_1}+\mathbb{G}_{S_2})(\cdot,t-s)\ast F_3(\cdot,s)ds\nm\\
&&\ \ +\int_0^tD_x^\alpha (\mathbb{G}_{12}-\mathbb{G}_{S_1}-\mathbb{G}_{S_2})(\cdot,t-s))\ast F_4(\cdot,s)ds+\int_0^tD_x^\alpha (\mathbb{G}_{S_1}+\mathbb{G}_{S_2})(\cdot,t-s)\ast F_4(\cdot,s)ds\nm\\
&:=&R_7+R_8+R_9+R_{10}.
\lbl{5.24}
\ees

We first consider the case $|\alpha|=0$ to obtain the estimate for $n_2$. We divide the time interval into two parts: short time $[0,\frac{t}{2}]$ and long time $(\frac{t}{2}, t]$. For $R_7$, since the nonlinear term $F_3$ has the conservative structure, and from (\ref{5.21}), we have for the short time
\bes
&&\left|\int_0^{\frac{t}{2}} (\mathbb{G}_{11}-\mathbb{G}_{S_1}-\mathbb{G}_{S_2})(\cdot,t-s)\ast F_3(\cdot,s)ds\right|\nm\\
&=&\left|\int_0^{\frac{t}{2}}D_x (\mathbb{G}_{11}-\mathbb{G}_{S_1}-\mathbb{G}_{S_2})(\cdot,t-s)\ast (n_1v_2+n_2v_1)(\cdot,s)ds\right|\nm\\
&\leq& \int_0^{\frac{t}{2}}|D_x (\mathbb{G}_{11}-\mathbb{G}_{S_1}-\mathbb{G}_{S_2})(\cdot,t-s)|\ast |n_1v_2+n_2v_1|(\cdot,s)ds\nm\\
&\leq& CM^2(T)\int_0^{\frac{t}{2}}\!\!\int_{\mathbb{R}^3}(1+t-s)^{-2}\bigg(1+\frac{|x-y|^2}{1+t-s}\bigg)^{-N}
(1+s)^{-\frac{7}{2}}\bigg(1+\frac{|y|^2}{1+s}\bigg)^{-\frac{3}{2}}\nm\\
\ \ \ \ \ \ \ &&\cdot\left\{\bigg(1+\frac{|y|^2}{1+s}\bigg)^{-(\frac{3}{2}-\varepsilon)}
+\bigg(1+\frac{(|y|-cs)^2}{1+s}\bigg)^{-(\frac{3}{2}-\varepsilon)}\right\}dyds
:=R_7^1+R_7^2.
\lbl{5.25}
\ees
When $|x|^2\leq 1+t$, by using (\ref{5.1}) and Young's inequality, we have
\bess
R_7^1+R_7^2&\leq & CM^2(T)(1+t)^{-2}\left\{\int_0^{\frac{t}{2}}(1+s)^{-\frac{7}{2}}(1+s)^{\frac{3}{2}}ds
+\int_0^{\frac{t}{2}}(1+s)^{-\frac{7}{2}}(1+s)^{\frac{3}{4}}(1+t-s)^{\frac{5}{4}}ds\right\}\nm\\
&\leq& CM^2(T)(1+t)^{-2}\left\{\!1\!+\!\!\int_0^{\frac{t}{2}}\!\!\!\int_{\mathbb{R}^3}\!\bigg(1\!+\!\frac{|x-y|^2}{1\!+\!t\!-\!s}\bigg)^{\!\!-N}
\!\!\!\!(1+s)^{-\frac{7}{2}}\bigg(1\!+\!\frac{|y|^2}{1+s}\bigg)^{\!\!-\frac{3}{2}}\!\!\bigg(1\!+\!\frac{(|y|\!-\!cs)^2}{1\!+\!s}\bigg)^{\!\!-\frac{1}{4}}\!\!dyds\!\right\}\nm\\ \ \
&\leq& CM^2(T)(1+t)^{-2}\left\{1+\int_0^{\frac{t}{2}}\!\!\int_{\mathbb{R}^3}(1+s)^{-\frac{7}{2}}\bigg(1+\frac{|y|^2}{1+s}\bigg)^{-\frac{3}{2}}
(1+s)^{\frac{1}{2}}\bigg(1+\frac{|y|^2}{1+s}\bigg)^{-\frac{1}{4}}dyds\right\}\nm\\
&\leq& CM^2(T)(1+t)^{-2}\left\{1+\int_0^{\frac{t}{2}}(1+s)^{-\frac{7}{2}}\!(1+s)^{\frac{1}{2}}(1+s)^{\frac{3}{2}}ds\right\}\nm\\
&\leq& CM^2(T)(1+t)^{-2}\leq C(1+t)^{-2}\left(1+\frac{|x|^2}{1+t}\right)^{-\frac{3}{2}},
\eess
where we have used the estimate (\ref{5.26(2)}).\\
When $|x|^2> 1+t$ and $|y|>\frac{|x|}{2}$, by using (\ref{5.1}), Young's inequality and Lemma \ref{A1}, we have
\bess
R_7^1&=&CM^2(T)\int_0^{\frac{t}{2}}\!\!\!\int_{\mathbb{R}^3}(1\!+\!t\!-\!s)^{-2}\bigg(1+\frac{|x-y|^2}{1\!+\!t\!-\!s}\bigg)^{-N}
(1+s)^{-\frac{7}{2}}\bigg(1+\frac{|y|^2}{1+s}\bigg)^{-(3-\varepsilon)}dyds\nm\\
&\leq& CM^2(T)(1+t)^{-2}\bigg(1+\frac{|x|^2}{1+t}\bigg)^{-\frac{3}{2}}(1+t)^{-\frac{3}{2}}\int_0^{\frac{t}{2}}(1+s)^{-2}(1+t-s)^{\frac{3}{4}}(1+s)^{\frac{3}{4}}ds\nm\\
&\leq& CM^2(T)(1+t)^{-\frac{11}{4}}\bigg(1+\frac{|x|^2}{1+t}\bigg)^{-\frac{3}{2}},
\eess
and
\bess
R_7^2&=&CM^2(T)\!\int_0^{\frac{t}{2}}\!\!\!\!\int_{\mathbb{R}^3}(1\!+\!t\!-\!s)^{-2}\bigg(1+\frac{|x-y|^2}{1\!+\!t\!-\!s}\bigg)^{-N}\!\!\!(1+s)^{-\frac{7}{2}}
\bigg(1+\frac{|y|^2}{1\!+\!s}\bigg)^{-\frac{3}{2}}\!\!
\bigg(1\!+\!\frac{(|y|\!-\!cs)^2}{1+s}\bigg)^{-(\frac{3}{2}-\varepsilon)}\!\!dyds\nm\\
&\leq & CM^2(T)(1+t)^{-2}\bigg(1+\frac{|x|^2}{1+t}\bigg)^{-\frac{3}{2}}(1+t)^{-\frac{3}{2}}
\int_0^{\frac{t}{2}}(1+s)^{-2}(1+t-s)^{\frac{3}{4}}(1+s)^{\frac{5}{4}}ds\nm\\
&\leq& CM^2(T)(1+t)^{-\frac{5}{2}}\bigg(1+\frac{|x|^2}{1+t}\bigg)^{-\frac{3}{2}}.
\eess
Similarly, when $|x|^2> 1+t$ and $|y|\leq\frac{|x|}{2}$, it also holds that
\bess
|R_7|\leq CM^2(T)(1+t)^{-2}\bigg(1+\frac{|x|^2}{1+t}\bigg)^{-\frac{3}{2}}.
\eess

Next, we consider the long time part of the nonlinear term. We have
\bes
&&\left|\int_{\frac{t}{2}}^t (\mathbb{G}_{11}-\mathbb{G}_{S_1}-\mathbb{G}_{S_2})(\cdot,t-s))\ast F_3(\cdot,s)ds\right|\nm\\
&=&\left|\int_{\frac{t}{2}}^tD_x (\mathbb{G}_{11}-\mathbb{G}_{S_1}-\mathbb{G}_{S_2})(\cdot,t-s)\ast (n_1v_2+n_2v_1)(\cdot,s)ds\right|\nm\\
&\leq& \int_{\frac{t}{2}}^t|D_x (\mathbb{G}_{11}-\mathbb{G}_{S_1}-\mathbb{G}_{S_2})(\cdot,t-s)|\ast (|n_1v_2|+|n_2v_1|)(\cdot,s)ds\nm\\
&\leq& CM^2(T)\int_{\frac{t}{2}}^t\!\int_{\mathbb{R}^3}(1+t-s)^{-2}\bigg(1+\frac{|x-y|^2}{1+t-s}\bigg)^{-N}(1+s)^{-\frac{7}{2}}\bigg(1+\frac{|y|^2}{1+s}\bigg)^{-\frac{3}{2}}\nm\\
\ \ \ \ \ \ \ &&\cdot\left\{\bigg(1+\frac{|y|^2}{1+s}\bigg)^{-(\frac{3}{2}-\varepsilon)}+\bigg(1+\frac{(|y|-cs)^2}{1+s}\bigg)^{-(\frac{3}{2}-\varepsilon)}\right\}dyds
:=R_7^3+R_7^4.
\lbl{5.26}
\ees
When $|x|^2\leq 1+t$, we have
\bess
R_7^3+R_7^4&\leq & CM^2(T)(1+t)^{-\frac{7}{2}}\int_{\frac{t}{2}}^t(1+t-s)^{-2}(1+t-s)^{\frac{3}{2}}ds\nm\\
&\leq& CM^2(T)(1+t)^{-3}\leq CM^2(T)(1+t)^{-3}\left(1+\frac{|x|^2}{1+t}\right)^{-\frac{3}{2}}.
\eess
When $|x|^2> 1+t$ and $|y|>\frac{|x|}{2}$,
\bess
R_7^3&=&CM^2(T)\int_{\frac{t}{2}}^t\!\!\int_{\mathbb{R}^3}(1+t-s)^{-2}\bigg(1+\frac{|x-y|^2}{1+t-s}\bigg)^{-N}(1+s)^{-\frac{7}{2}}
\bigg(1+\frac{|y|^2}{1+s}\bigg)^{-(3-\varepsilon)}dyds\nm\\
&\leq& CM^2(T)(1+t)^{-\frac{7}{2}}\left(1+\frac{|x|^2}{1+t}\right)^{-(3-\varepsilon)}\int_{\frac{t}{2}}^t(1+t-s)^{-2}(1+t-s)^{\frac{3}{2}}ds\nm\\
&\leq& CM^2(T)(1+t)^{-3}\bigg(1+\frac{|x|^2}{1+t}\bigg)^{-(3-2\varepsilon)}.
\eess
\bess
R_7^4&=&CM^2(T)\!\int_{\frac{t}{2}}^t\!\!\int_{\mathbb{R}^3}(1\!+\!t\!-\!s)^{-2}\bigg(1+\frac{|x-y|^2}{1\!+\!t\!-\!s}\bigg)^{-N}\!\!
(1+s)^{-\frac{7}{2}}\bigg(1+\frac{|y|^2}{1+s}\bigg)^{-\frac{3}{2}}\!\!
\bigg(1\!+\!\frac{(|y|\!-\!cs)^2}{1\!+\!s}\bigg)^{-(\frac{3}{2}-\varepsilon)}\!dyds\nm\\
&\leq & CM^2(T)(1+t)^{-\frac{7}{2}}\bigg(1+\frac{|x|^2}{1+t}\bigg)^{-\frac{3}{2}}\int_{\frac{t}{2}}^t(1+t-s)^{-2}(1+t-s)^{\frac{5}{2}}ds\nm\\
&\leq& CM^2(T)(1+t)^{-3}\bigg(1+\frac{|x|^2}{1+t}\bigg)^{-\frac{3}{2}}.
\eess
Similarly, one can get for $|x|^2> 1+t$ and $|y|\leq\frac{|x|}{2}$ that
\bess
R_7^3+R_7^4\leq CM^2(T)(1+t)^{-2}\bigg(1+\frac{|x|^2}{1+t}\bigg)^{-\frac{3}{2}}.
\eess

Notice that there exists an additional derivative in $\mathbb{G}_{12}$ though the nonlinear term $F_4$ has not a divergence form. Then, after a similar calculus as for $R_7$, and by using the ansatz (\ref{5.20}), one can immediately obtain
\bess
R_9\leq CM^2(T)(1+t)^{-2}\bigg(1+\frac{|x|^2}{1+t}\bigg)^{-\frac{3}{2}}.
\eess

Next we consider $R_8$ and $R_{10}$, In spite of singularity of $\mathbb{G}_{S_1}$ as $t\rightarrow0$, the convolution between $\mathbb{G}_{S_1}$ and the nonlinear terms has not this singularity when $|\alpha|\leq 1$. Thus, one can similarly deduce the estimate for this convolution as $R_7$ and $R_9$. Then, we consider the Dirac-like function $\mathbb{G}_{S_2}$. We need the following standard nonlinear estimates as in \cite{LW,LS}.

\begin{lemma}\lbl{l 52(0)} There exists a constant $C>0$ such that
\bes
&&\int_0^t\!\!\int_{\mathbb{R}^3} e^{-b(t-s)}\tilde{f}(x\!-\!y)(1+s)^{-\frac{3+\gamma}{2}}\bigg(1+\frac{|y|^2}{1+s}\bigg)^{-a}\!\!\!dyds\leq C(1+t)^{-\frac{3+\gamma}{2}}\bigg(1+\frac{|x|^2}{1+t}\bigg)^{-a}\!\!,\\
&&\int_0^t\!\!\int_{\mathbb{R}^3} e^{-b(t-s)}\tilde{f}(x\!-\!y)(1+s)^{-\frac{3+\gamma}{2}}\bigg(1+\frac{(|y|-cs)^2}{1+s}\bigg)^{-a}\!\!\!dyds\leq C(1+t)^{-\frac{3+\gamma}{2}}\bigg(1+\frac{(|x|-ct)^2}{1+t}\bigg)^{-a}\!\!,\ \ \
\ees
\end{lemma}
where $b>0$ and $\tilde{f}(x)$ is defined in Proposition \ref{l 49}, which is a Dirac-like function.

This lemma implies that the convolution between $\mathbb{G}_{S_2}$ and the nonlinear terms is almost the same as the convolution between $e^{-bt}\delta(x)$ and the nonlinear terms. Thus, in what follows, we regard $\mathbb{G}_{S_2}$ as $e^{-bt}\delta(x)$. Notice the difference between $\psi_3$, $\psi_4$ and $\psi_5$, $\psi_6$ in the ansatz (\ref{5.20}). We have to carefully estimate the convolutions between $\mathbb{G}_{S_2}$ and $F_3$, $F_4$ corresponding to the variables $n_2$, $w_2$. In fact, from the relation (\ref{5.22}) and the equations (\ref{5.18}), we can rewrite $F_3$ and $F_4$ as
\bes
|F_3|+F_4|&=&\mathcal{O}(1)\{Dn_1(w_2+n_2)+n_1D(w_2+n_2)+Dn_2w_1+n_2Dv_1+Dn_2n_1+w_1Dw_2\nm\\
&&\ \ \ \ \ \ +w_1Dn_2+n_1D^2w_2+n_1D^2n_2+n_1D^2w_2+n_1D^2n_2+\cdots\},
\lbl{5.26(0)}
\ees
where ``$\cdots$" denotes the  terms containing $n_2$ or $w_2$. Hence, for the term ``$\cdots$" above one can use the ansatz (\ref{5.20}) and Lemma \ref{l 52(0)} to get the pointwise profile with the factor $a=\frac{3}{2}$ in Lemma \ref{l 52(0)}. For the other terms above, from the ansatz (\ref{5.20}), it is not so obvious to obtain the same result. We just take the term $w_1Dw_2$ in (\ref{5.26(0)}) for example. In fact, we have
\bes
&&|w_1Dw_2|(x,t)\leq CM^2(T)\left\{(1+t)^{-\frac{3}{2}}\left(1+\frac{|x|^2}{1+t}\right)^{-(\frac{3}{2}-\varepsilon)}
\!\!+\!(1+t)^{-2}\left(1+\frac{(|x|-ct)^2}{1+t}\right)^{-(\frac{3}{2}-\varepsilon)}\right\}\nm\\
&&\ \ \ \ \ \ \ \ \ \ \ \ \ \ \ \ \ \ \ \ \ \ \ \ \ \cdot(1+t)^{-\frac{3}{2}}\bigg(1+\frac{|x|^2}{1+t}\bigg)^{-(\frac{3}{2}-\varepsilon)}\nm\\
&\leq& CM^2(T)(1+t)^{-3}\left(1+\frac{|x|^2}{1+t}\right)^{-(3-2\varepsilon)}
\!\!+CM^2(T)(1+t)^{-\frac{7}{2}}\left(1+\frac{|x|^2}{1+t}\right)^{-(\frac{3}{2}-\varepsilon)}\!\left(1+\frac{(|x|-ct)^2}{1+t}\right)^{-(\frac{3}{2}-\varepsilon)}\nm\\
&\leq& CM^2(T)(1+t)^{-3}\left(1+\frac{|x|^2}{1+t}\right)^{-(3-2\varepsilon)}
\!\!+CM^2(T)(1+t)^{-\frac{7}{2}}\left(1+\frac{|x|^2}{1+t}\right)^{-(\frac{3}{2}-\varepsilon)}(1+t)^{2\varepsilon}\!\left(1+\frac{|x|^2}{1+t}\right)^{-\varepsilon}\nm\\
&\leq& CM^2(T)(1+t)^{-3}\left(1+\frac{|x|^2}{1+t}\right)^{-(3-2\varepsilon)}\!+(1+t)^{-(\frac{7}{2}-2\varepsilon)}\left(1+\frac{|x|^2}{1+t}\right)^{-\frac{3}{2}}\nm\\
&\leq& CM^2(T)(1+t)^{-3}\left(1+\frac{|x|^2}{1+t}\right)^{-\frac{3}{2}}.
\lbl{5.26(1)}
\ees
Here we have used the following estimate for $a_1>0$ that
\bes
\left(1+\frac{(|x|-ct)^2}{1+t}\right)^{-a_1}\leq C(1+t)^{2a_1}\left(1+\frac{|x|^2}{1+t}\right)^{-a_1},\lbl{5.26(2)}
\ees
which can be proved by dividing the domain into two parts: $|x|<2ct$ and $|x|\geq 2ct$. For simplicity, we omit the details. Lastly, the other terms in
(\ref{5.26(0)}) could be treated similarly.

In a conclusion, we have the following estimate for $n_2$
\bes
|n_2(x,t)|\leq C(\varepsilon_0+M^2(T))(1+t)^{-2}\bigg(1+\frac{|x|^2}{1+t}\bigg)^{-\frac{3}{2}}.\lbl{5.27}
\ees
Synchronously, one can have the estimate of $v_2$ that
\bes
|v_2(x,t)|\leq C(\varepsilon_0+M^2(T))(1+t)^{-\frac{3}{2}}\bigg(1+\frac{|x|^2}{1+t}\bigg)^{-\frac{3}{2}},\lbl{5.27(0)}
\ees
which together with the relation (\ref{5.21}) yields
\bes
|w_2(x,t)|\leq C(\varepsilon_0+M^2(T))(1+t)^{-\frac{3}{2}}\bigg(1+\frac{|x|^2}{1+t}\bigg)^{-\frac{3}{2}}.\lbl{5.27(1)}
\ees
That is, we have completed the estimate for $(n_2,w_2)$ as
\bes
|n_2(x,t)|&\leq& C(\varepsilon_0+M^2(T))\psi_3,\nm\\
|w_2(x,t)|&\leq& C(\varepsilon_0+M^2(T))\psi_4.\lbl{5.27(2)}
\ees

Now, from the relation $\nabla\phi=\frac{\nabla}{-\Delta}n_2$, we can immediately give the following proposition on $\nabla\phi$, which is key for us to deduce the optimal pointwise estimate for $n_1$ and $w_1$.

\begin{proposition}\lbl{l 53} We have for a constant $\varepsilon$ satisfying $0<\varepsilon\ll1$ that
\bes
|\nabla\phi(x,t)|\leq C(\varepsilon_0+M^2(T))(1+t)^{-\frac{3-\varepsilon}{2}}\bigg(1+\frac{|x|^2}{1+t}\bigg)^{-\frac{2-\varepsilon}{2}}.\lbl{5.28}
\ees
\end{proposition}
\begin{proof} The relation $\nabla\phi=\frac{\nabla}{-\Delta}n_2$ yields that
\bes
|\nabla\phi(x,t)|=(1+t)^{-2}\int_{\mathbb{R}^3}|x-y|^{-2}\bigg(1+\frac{|y|^2}{1+t}\bigg)^{-\frac{3}{2}}dy.
\ees
When $|x|^2\leq 1+t$, if $|x-y|\leq1$, we know
\bes
|\nabla\phi(x,t)|\leq C(1+t)^{-2}\leq C(1+t)^{-2}\bigg(1+\frac{|x|^2}{1+t}\bigg)^{-\frac{3}{2}},
\ees
and if $|x-y|>1$, by using Young's inequality we have
\bes
|\nabla\phi(x,t)|\leq C(1+t)^{-\frac{3-\varepsilon}{2}}\leq C(1+t)^{-\frac{3-\varepsilon}{2}}\bigg(1+\frac{|x|^2}{1+t}\bigg)^{-\frac{3}{2}}.
\ees
When $|x|^2> 1+t$, if $|y|<\frac{|x|}{2}$, we know $|x-y|\geq\frac{|x|}{2}$, which implies that
\bes
|\nabla\phi(x,t)|&\leq& C(1+t)^{-2}(1+\frac{|x|^2}{1+t})^{-\frac{2-\varepsilon}{2}}(1+t)^{-\frac{2-\varepsilon}{2}}\int_{\mathbb{R}^3}|x-y|^{-\varepsilon}\bigg(1+\frac{|y|^2}{1+t}\bigg)^{-\frac{3}{2}}dy\nm\\
&\leq &C(1+t)^{-\frac{3-\varepsilon}{2}}\bigg(1+\frac{|x|^2}{1+t}\bigg)^{-\frac{2-\varepsilon}{2}}.
\ees
When $|y|\geq\frac{|x|}{2}$, we also obtain
\bes
|\nabla\phi(x,t)|\leq C(1+t)^{-\frac{3-\varepsilon}{2}}\bigg(1+\frac{|x|^2}{1+t}\bigg)^{-\frac{2-\varepsilon}{2}}.
\ees
This proves Proposition \ref{l 53}.
\end{proof}\\

  Next, we shall go back to consider the conservation system on $n_1$ and $w_1$. We rewrite
\bes
&&\int_0^tD_x^\alpha (G_{12}(\cdot,t-s)\ast F_1(\cdot,s)ds\nm\\
&=&\int_0^tD_x^\alpha (G_{12}-G_{S_1}-G_{S_2})(\cdot,t-s))\ast F_1(\cdot,s)ds+\int_0^tD_x^\alpha (G_{S_1}+G_{S_2})(\cdot,t-s)\ast F_1(\cdot,s)ds\nm\\
&:=&R_{15}+R_{16}.\lbl{5.33}
\ees
For $R_{15}$, we mainly focus on the ``worst" nonlinear term $n_2\nabla\phi$ in $F_1$ of  (\ref{2.4}). In fact, we have
\bes
\bar{R}_{15}:=\int_0^tD_x^\alpha (G_{12}(\cdot,t-s)\ast (n_2\nabla\phi)(\cdot,s)ds
=-\int_0^tD_x^\beta G_{12}(\cdot,t-s)\ast(|\nabla\phi|^2)(\cdot,s)ds,\lbl{5.34}
\ees
where $|\beta|=|\alpha|+1$.\\
Since the Green's function of the Navier-Stokes system contains the diffusion wave, the Riesz wave and the Hyngens' wave, to estimate the interaction of these different waves, we should divide both the time $t$ and the space $x$ into several parts. In particular, as in \cite {LS} we define
\bess
&&D_1=\{x^2\leq 1+t\},\ \ D_2=\{(|x|-ct)^2\leq 1+t\},\ \ D_3=\{|x|\geq ct+\sqrt{1+t}\},\nm\\
&&D_4=\{\sqrt{1+t}\leq|x|\leq\frac{ct}{2}\},\ \ D_5=\{\frac{ct}{2}\leq|x|\leq ct-\sqrt{1+t}\}.
\lbl{5.34(0)}
\eess

In fact, we have the following estimate for the nonlinear term $n_2\nabla\phi$ in $F_1(n_1,w_1,n_2,w_2)$.

\begin{lemma}\lbl{l 54} We have
\bes
\mathcal{N}_1:&=&\Big|\int_0^t H(\cdot,t-s)\ast(n_2\nabla\phi)(\cdot,s)ds\Big|\nm\\
&\leq& CM^2(T)(1+t)^{-2}\left\{\bigg(1+\frac{(|x|-ct)^2}{1+t}\bigg)^{-(\frac{3}{2}-\varepsilon)}
+\bigg(1+\frac{|x|^2}{1+t}\bigg)^{-(\frac{3}{2}-\varepsilon)}\right\},
\ees
\end{lemma}
where $H(x,t)$ is the H-wave defined in Proposition \ref{l 317}.

\begin{proof} When $(x,t)\in D_1\cup D_2$, we divide the interval $[0,t]$ into two parts. In fact, by using Young's inequality and the ansatz (\ref{5.20}), we have
\bes
&&\Big|\int_0^{\frac{t}{2}}H(\cdot,t-s)\ast(n_2\nabla\phi)(\cdot,s)ds\Big|\nm\\
&=&\Big|-\int_0^{\frac{t}{2}}(1+t-s)^{-\frac{5}{2}}e^{-\frac{(|x-y|-c(t-s))^2}{C(t-s)}}\ast(|\nabla\phi|^2)(\cdot,s)ds\Big|\nm\\
&\leq&CM^2(T)\int_0^{\frac{t}{2}}\!\int_{\mathbb{R}^3}(1+t-s)^{-\frac{5}{2}}e^{-\frac{(|x-y|-c(t-s))^2}{C(t-s)}}
(1+s)^{-(3-\varepsilon)}\bigg(1+\frac{|y|^2}{1+s}\bigg)^{-(2-\varepsilon)}dyds\nm\\
&\leq& CM^2(T)(1+t)^{-\frac{5}{2}}\int_0^{\frac{t}{2}}(1+s)^{-(3-\varepsilon)}(1+s)^{\frac{3}{2}}ds\nm\\
&\leq& CM^2(T)(1+t)^{-\frac{5}{2}}\leq C\left\{\begin{array}{lll}\medskip
   (1+t)^{-\frac{5}{2}}\bigg(1+\frac{|x|^2}{1+t}\bigg)^{-\frac{3}{2}},\ (x,t)\in D_1;\\
   (1+t)^{-\frac{5}{2}}\bigg(1+\frac{(|x|-ct)^2}{1+t}\bigg)^{-\frac{3}{2}},\ (x,t)\in D_2,
\end{array}\right.
\ees
and
\bes
&&\Big|\int_{\frac{t}{2}}^tH(\cdot,t-s)\ast(n_2\nabla\phi)(\cdot,s)ds\Big|
=\Big|\int_{\frac{t}{2}}^t(1+t-s)^{-2}e^{-\frac{(|x-y|-c(t-s))^2}{C(t-s)}}\ast(n_2\nabla\phi)(\cdot,s)ds\Big|\nm\\
&\leq&CM^2(T)\int_{\frac{t}{2}}^t\int_{\mathbb{R}^3}(1+t-s)^{-2}e^{-\frac{(|x-y|-c(t-s))^2}{C(t-s)}}
(1+s)^{-(\frac{7}{2}-\varepsilon)}\bigg(1+\frac{|y|^2}{1+s}\bigg)^{-(2-\varepsilon)}dyds\nm\\
&\leq& CM^2(T)(1+t)^{-(\frac{7}{2}-\varepsilon)}\int_{\frac{t}{2}}^t(1+t-s)^{-2}(1+t-s)^{-\frac{5}{4}}(1+s)^{\frac{3}{4}}ds\nm\\
&\leq& CM^2(T)(1+t)^{-(\frac{7}{2}-\varepsilon)}\cdot(1+t)\leq C(1+t)^{-(\frac{5}{2}-\varepsilon)}\nm\\
&\leq& \left\{\begin{array}{lll}\medskip
   CM^2(T)(1+t)^{-\frac{5-\varepsilon}{2}}(1+\frac{|x|^2}{1+t})^{-\frac{3}{2}},\ (x,t)\in D_1;\\
   CM^2(T)(1+t)^{-\frac{5-\varepsilon}{2}}(1+\frac{(|x|-ct)^2}{1+t})^{-\frac{3}{2}},\ (x,t)\in D_2.
\end{array}\right.
\ees
When $(x,t)\in D_3$,  and $|y|\geq \frac{|x|-ct}{2}$, by using Lemma \ref{A1} and the ansatz (\ref{5.20}), we can get
\bes
&&\Big|\int_0^{\frac{t}{2}}H(\cdot,t-s)\ast(n_2\nabla\phi)(\cdot,s)ds\Big|\nm\\
&\leq&CM^2(T)\int_0^{\frac{t}{2}}\!\!\int_{\mathbb{R}^3}(1+t-s)^{-\frac{5}{2}}e^{-\frac{(|x-y|-c(t-s))^2}{C(t-s)}}
(1+s)^{-(3-\varepsilon)}\bigg(1+\frac{|y|^2}{1+s}\bigg)^{-(2-\varepsilon)}dyds\nm\\
&\leq& CM^2(T)(1+t)^{-\frac{5}{2}}\bigg(1+\frac{(|x|-ct)^2}{1+t}\bigg)^{-(\frac{3}{2}-\varepsilon)}(1+t)^{-(\frac{3}{2}-\varepsilon)}
\int_0^{\frac{t}{2}}(1+s)^{-\frac{3}{2}}(1+t-s)^{\frac{5}{2p}}(1+s)^{\frac{3}{2q}}ds\nm\\
&\leq & CM^2(T)(1+t)^{-\frac{5}{2}}\bigg(1+\frac{(|x|-ct)^2}{1+t}\bigg)^{-(\frac{3}{2}-\varepsilon)}(1+t)^{-(\frac{3}{2}-\varepsilon)}\int_0^{\frac{t}{2}}(1+s)^{-\frac{3}{2}}(1+t-s)^{\frac{15}{8}}(1+s)^{\frac{3}{8}}ds\nm\\
&\leq& CM^2(T)(1+t)^{-\frac{5}{2}}\bigg(1+\frac{(|x|-ct)^2}{1+t}\bigg)^{-(\frac{3}{2}-\varepsilon)}(1+t)^{-(\frac{3}{2}-\varepsilon)}(1+t)^{\frac{15}{8}}\nm\\
&\leq& CM^2(T)(1+t)^{-(\frac{17}{8}-\varepsilon)}\bigg(1+\frac{(|x|-ct)^2}{1+t}\bigg)^{-(\frac{3}{2}-\varepsilon)},\ {\rm where}\ p=\frac{4}{3},\ q=4,
\ees
and similarly
\bes
&&\Big|\int_{\frac{t}{2}}^t(1+t-s)^{-2}e^{-\frac{(|x-y|-c(t-s))^2}{C(t-s)}}\ast(n_2\nabla\phi)(\cdot,s)ds\Big|\nm\\
&\leq&CM^2(T)\int_{\frac{t}{2}}^t\!\int_{\mathbb{R}^3}(1+t-s)^{-2}e^{-\frac{(|x-y|-c(t-s))^2}{C(t-s)}}
(1+s)^{-(\frac{7}{2}-\varepsilon)}\bigg(1+\frac{|y|^2}{1+s}\bigg)^{-(2-\varepsilon)}dyds\nm\\
&\leq& CM^2(T)(1+t)^{-(\frac{7}{2}-\varepsilon)}\bigg(1+\frac{(|x|-ct)^2}{1+t}\bigg)^{-(\frac{3}{2}-\varepsilon)}
\!\!\int_{\frac{t}{2}}^t\!\!\int_{\mathbb{R}^3}(1+t-s)^{-2}e^{-\frac{(|x-y|-c(t-s))^2}{2C(t-s)}}\bigg(1+\frac{|y|^2}{1+s}\bigg)^{-\frac{1}{2}}dyds\nm\\
&\leq & CM^2(T)(1+t)^{-(\frac{7}{2}-\varepsilon)}\bigg(1+\frac{(|x|-ct)^2}{1+t}\bigg)^{-(\frac{3}{2}-\varepsilon)}
\int_0^{\frac{t}{2}}(1+t-s)^{-2}(1+t-s)^{\frac{15}{8}}(1+s)^{\frac{3}{8}}ds\nm\\
&\leq& CM^2(T)(1+t)^{-(\frac{7}{2}-\varepsilon)}\bigg(1+\frac{(|x|-ct)^2}{1+t}\bigg)^{-(\frac{3}{2}-\varepsilon)}(1+t)\nm\\
&\leq& CM^2(T)(1+t)^{-(\frac{5}{2}-\varepsilon)}\bigg(1+\frac{(|x|-ct)^2}{1+t}\bigg)^{-(\frac{3}{2}-\varepsilon)}.
\ees
When $(x,t)\in D_3$ and $|y|< \frac{|x|-ct}{2}$, by using the same method above, one can immediately obtain
\bes
\Big|\int_{\frac{t}{2}}^t(1+t-s)^{-2}e^{-\frac{(|x-y|-c(t-s))^2}{C(t-s)}}\ast(n_2\nabla\phi)(\cdot,s)ds\Big|\leq CM^2(T)(1+t)^{-(\frac{5}{2}-\varepsilon)}\bigg(1+\frac{(|x|-ct)^2}{1+t}\bigg)^{-\frac{3}{2}}.\ \ \
\ees
When $(x,t)\in D_4: \sqrt{1+t}\leq|x|\leq \frac{ct}{2}$, we have to divide the interval into several parts.\\
Case 1: $0\leq s\leq\frac{t}{2}-\frac{|x|}{2c}$. When $|y|\leq \frac{ct-|x|}{4}$, it holds that $c(t-s)-|x-y|\geq\frac{ct-|x|}{4}$. Then,
\bes
&&\int_0^{\frac{t}{2}-\frac{|x|}{2c}}\!\!\int_{|y|\leq \frac{ct-|x|}{4}}(1+t-s)^{-\frac{5}{2}}e^{-\frac{(|x-y|-c(t-s))^2}{C(t-s)}}
(1+s)^{-(3-\varepsilon)}\bigg(1+\frac{|y|^2}{1+s}\bigg)^{-(2-\varepsilon)}dyds\nm\\
&\leq& C(1+t)^{-\frac{5}{2}}\bigg(1+\frac{(|x|-ct)^2}{1+t}\bigg)^{-\frac{3}{2}}\int_0^{\frac{t}{2}-\frac{|x|}{2c}}(1+s)^{-(3-\varepsilon)}(1+s)^{\frac{3}{2}}ds\nm\\
&\leq& C(1+t)^{-\frac{5}{2}}\bigg(1+\frac{(|x|-ct)^2}{1+t}\bigg)^{-\frac{3}{2}}.
\ees
When $|y|> \frac{ct-|x|}{4}>\frac{ct}{8}>\sqrt{1+t}$, we have
\bes
&&\int_0^{\frac{t}{2}-\frac{|x|}{2c}}\!\!\!\int_{|y|\leq \frac{ct-|x|}{4}}(1+t-s)^{-\frac{5}{2}}e^{-\frac{(|x-y|-c(t-s))^2}{C(t-s)}}
(1+s)^{-(3-\varepsilon)}\bigg(1+\frac{|y|^2}{1+s}\bigg)^{-(2-\varepsilon)}dyds\nm\\
&\leq& C(1+t)^{-\frac{5}{2}}\bigg(1+\frac{(|x|-ct)^2}{1+t}\bigg)^{-\frac{3}{2}-\varepsilon}(1+t)^{-(\frac{3}{2}-\varepsilon)}\nm\\
&&\ \ \ \ \ \ \ \ \ \ \ \ \ \ \times\int_0^{\frac{t}{2}-\frac{|x|}{2c}}\!\!\!\int_{|y|\leq \frac{ct-|x|}{4}} e^{-\frac{(|x-y|-c(t-s))^2}{C(t-s)}}(1+s)^{-\frac{3}{2}}\bigg(1+\frac{|y|^2}{1+s}\bigg)^{-\frac{1}{2}}dyds\nm\\
&\leq& C(1+t)^{-(4-\varepsilon)}\bigg(1+\frac{(|x|-ct)^2}{1+t}\bigg)^{-(\frac{3}{2}-\varepsilon)}
\int_0^{\frac{t}{2}-\frac{|x|}{2c}}(1+s)^{-\frac{3}{2}}(1+t-s)^{\frac{15}{8}}(1+s)^{\frac{3}{8}}ds\nm\\
&\leq& C(1+t)^{-(\frac{17}{8}-\varepsilon)}\bigg(1+\frac{(|x|-ct)^2}{1+t}\bigg)^{-(\frac{3}{2}-\varepsilon)}.
\ees
Case 2: $\frac{t}{2}-\frac{|x|}{2c}\leq s\leq\frac{t}{2}$. Since $|x|\leq \frac{ct}{2}$, we know $s\geq \frac{t}{2}-\frac{|x|}{2c}>\frac{t}{4}$ and $\frac{t}{2}\leq t-s\leq t$.  Thus,
\bes
&&\int_{\frac{t}{2}-\frac{|x|}{2c}}^{\frac{t}{2}}\int_{\mathbb{R}^3}(1+t-s)^{-\frac{5}{2}}e^{-\frac{(|x-y|-c(t-s))^2}{C(t-s)}}
(1+s)^{-(3-\varepsilon)}\bigg(1+\frac{|y|^2}{1+s}\bigg)^{-(2-\varepsilon)}dyds\nm\\
&\leq& C(1+t)^{-\frac{5}{2}}\bigg(1+\frac{ct-|x|}{2c}\bigg)^{-(3-\varepsilon)}(1+t)^2\nm\\
&\leq& CM^2(T)(1+t)^{-(\frac{5}{2}+\varepsilon)}\bigg(1+\frac{ct-|x|}{4c}\bigg)^{-(3-2\varepsilon)}(1+t)^2\nm\\
&\leq& C(1+t)^{-(\frac{1}{2}+\varepsilon)}\bigg(1+\frac{(ct-|x|)^2}{1+t}\bigg)^{-\frac{3-2\varepsilon}{2}}(1+t)^{-(\frac{3}{2}-\varepsilon)}\nm\\
&\leq& C(1+t)^{-2}\bigg(1+\frac{(|x|-ct)^2}{1+t}\bigg)^{-(\frac{3}{2}-\varepsilon)},
\ees
where we have used the following estimate
\bess
\int_0^t\!\!\int_{\mathbb{R}^3}e^{-\frac{(|x-y|-c(t-s))^2}{C(t-s)}}\bigg(1+\frac{|y|^2}{1+s}\bigg)^{-a_2}dyds\leq C(1+t)^2,\ {\rm for}\ a_2>\frac{3}{2}.
\eess
Case 3: $\frac{t}{2}\leq s\leq t-\frac{|x|}{4c}$. Since $|x|\leq \frac{ct}{2}$, we have
\bes
&&\int_{\frac{t}{2}}^{t-\frac{|x|}{4c}}\int_{\mathbb{R}^3}(1+t-s)^{-2}e^{-\frac{(|x-y|-c(t-s))^2}{C(t-s)}}(1+s)^{-(\frac{7}{2}-\varepsilon)}
\bigg(1+\frac{|y|^2}{1+s}\bigg)^{-(2-\varepsilon)}dyds\nm\\
&\leq& C(1+t)^{-(\frac{7}{2}-\varepsilon)}\bigg(1+\frac{|x|}{4c}\bigg)^{-2}(1+t)^2\leq CM^2(T)(1+t)^{-(\frac{1}{2}+\varepsilon)}(1+t)^{-(1-2\varepsilon)}(1+|x|)^{-2}\nm\\
&\leq& C(1+t)^{-(\frac{1}{2}+\varepsilon)}(1+|x|)^{-(3-2\varepsilon)}
\leq C(1+t)^{-(\frac{1}{2}+\varepsilon)}\bigg(1+\frac{|x|^2}{1+t}\bigg)^{-(\frac{3}{2}-\varepsilon)}\!(1+t)^{-\frac{3-2\varepsilon}{2}}\nm\\
&\leq& C(1+t)^{-2}\bigg(1+\frac{|x|^2}{1+t}\bigg)^{-(\frac{3}{2}-\varepsilon)}.
\ees
Case 4: $t-\frac{|x|}{4c}\leq s\leq t$. When $|x-y|\geq\frac{|x|}{2}$, we have $|x-y|-c(t-s)\geq|x-y|-\frac{|x|}{4}\geq\frac{|x|}{4}$, then
\bes
&&\int_{t-\frac{|x|}{4c}}^t\int_{\mathbb{R}^3}(1+t-s)^{-2}e^{-\frac{(|x-y|-c(t-s))^2}{C(t-s)}}(1+s)^{-(\frac{7}{2}-\varepsilon)}\bigg(1+\frac{|y|^2}{1+s}\bigg)^{-(2-\varepsilon)}dyds\nm\\
&\leq& C(1+t)^{-(\frac{7}{2}-\varepsilon)}\bigg(1+\frac{|x|^2}{1+t}\bigg)^{-3}\int_{t-\frac{|x|}{4c}}^t\int_{\mathbb{R}^3}(1+t-s)^{-2}e^{-\frac{(|x-y|-c(t-s))^2}{2C(t-s)}}\bigg(1+\frac{|y|^2}{1+s}\bigg)^{-(2-\varepsilon)}dyds\nm\\
&\leq& C(1+t)^{-(\frac{7}{2}-\varepsilon)}\bigg(1+\frac{|x|^2}{1+t}\bigg)^{-3}\int_{t-\frac{|x|}{4c}}^t(1+t-s)^{-2}(1+t-s)^{\frac{5}{2}}ds\nm\\
&\leq& C(1+t)^{-(\frac{7}{2}-\varepsilon)}\bigg(1+\frac{|x|^2}{1+t}\bigg)^{-3}(1+t)^{\frac{1}{2}}\frac{|x|}{4c}\nm\\
&\leq& C(1+t)^{-(\frac{7}{2}-\varepsilon)}\bigg(1+\frac{|x|^2}{1+t}\bigg)^{-3}(1+t)^{\frac{1}{2}}\bigg(1+\frac{|x|^2}{1+t}\bigg)^{\frac{1}{2}}(1+t)^{\frac{1}{2}}\nm\\
&\leq& C(1+t)^{-(\frac{5}{2}-\varepsilon)}\bigg(1+\frac{|x|^2}{1+t}\bigg)^{-\frac{5}{2}}.
\ees
When $|x-y|<\frac{|x|}{2}$, we have $|y|\geq \frac{|x|}{2}$, which yields that
\bes
&&\int_{t-\frac{|x|}{4c}}^t\int_{\mathbb{R}^3}(1+t-s)^{-2}e^{-\frac{(|x-y|-c(t-s))^2}{C(t-s)}}(1+s)^{-(\frac{7}{2}-\varepsilon)}\bigg(1+\frac{|y|^2}{1+s}\bigg)^{-(2-\varepsilon)}dyds\nm\\
&\leq& C(1+t)^{-(\frac{7}{2}-\varepsilon)}\bigg(1+\frac{|x|^2}{1+t}\bigg)^{-(2-\varepsilon)}\int_{t-\frac{|x|}{4c}}^t(1+t-s)^{-2}(1+t-s)^{\frac{5}{2}}ds\nm\\
&\leq&C(1+t)^{-(\frac{7}{2}-\varepsilon)}\bigg(1+\frac{|x|^2}{1+t}\bigg)^{-(2-\varepsilon)}(1+t)^{\frac{1}{2}}\bigg(1+\frac{|x|^2}{1+t}\bigg)^{\frac{1}{2}}(1+t)^{\frac{1}{2}}\nm\\
&\leq& C(1+t)^{-(\frac{7}{2}-\varepsilon)}\bigg(1+\frac{|x|^2}{1+t}\bigg)^{-(2-\varepsilon)}(1+t)\bigg(1+\frac{|x|^2}{1+t}\bigg)^{\frac{1}{2}}\nm\\
&\leq& C(1+t)^{-(\frac{5}{2}-\varepsilon)}\bigg(1+\frac{|x|^2}{1+t}\bigg)^{-(\frac{3}{2}-\varepsilon)}.
\ees
Lastly,  when $(x,t)\in D_5: \frac{ct}{2}\leq |x|\leq ct-\sqrt{1+t}$, we also divide the domain into four parts. The Case 1: $0\leq s\leq \frac{t}{2}-\frac{|x|}{2c}$ and Case 2: $\frac{t}{2}-\frac{|x|}{2c}\leq s\leq\frac{t}{2}$ can be treated as those in $D_4$, we omit the details. For
the Case 3: $\frac{t}{2}\leq s\leq t-\frac{ct-|x|}{4c}$, we have
\bes
&&\int_{\frac{t}{2}}^{t-\frac{ct-|x|}{4c}}\!\!\int_{\mathbb{R}^3}(1+t-s)^{-2}e^{-\frac{(|x-y|-c(t-s))^2}{C(t-s)}}(1+s)^{-(\frac{7}{2}-\varepsilon)}
\bigg(1+\frac{|y|^2}{1+s}\bigg)^{-(2-\varepsilon)}dyds\nm\\
&\leq& C(1+ct-|x|)^{-2}(1+t)^{-(\frac{7}{2}-\varepsilon)}(1+t)^2\nm\\
&\leq& C(1+t)^{-(\frac{1}{2}+\varepsilon)}(1+t)^{-(1-2\varepsilon)}(1+ct-|x|)^{-2}\nm\\
&\leq& C(1+t)^{-(\frac{1}{2}+\varepsilon)}(1+ct-|x|)^{-(3-2\varepsilon)}\nm\\
&\leq& C(1+t)^{-(\frac{1}{2}+\varepsilon)}\bigg(1+\frac{(ct-|x|)^2}{1+t}\bigg)^{-\frac{3-2\varepsilon}{2}}(1+t)^{-\frac{3-2\varepsilon}{2}}\nm\\
&\leq& C(1+t)^{-2}\bigg(1+\frac{(|x|-ct)^2}{1+t}\bigg)^{-(\frac{3}{2}-\varepsilon)}.
\ees\\
For the Case 4: $t-\frac{ct-|x|}{4c}\leq s\leq t$, when $|x-y|\geq\frac{ct-|x|}{2}$, we have $|x-y|-c(t-s)\geq\frac{ct-|x|}{4}$. Then
\bes
&&\int_{t-\frac{ct-|x|}{4c}}^t\int_{\mathbb{R}^3}(1+t-s)^{-2}e^{-\frac{(|x-y|-c(t-s))^2}{C(t-s)}}(1+s)^{-(\frac{7}{2}-\varepsilon)}
\bigg(1+\frac{|y|^2}{1+s}\bigg)^{-(2-\varepsilon)}dyds\nm\\
&\leq& C(1+t)^{-(\frac{7}{2}-\varepsilon)}\bigg(1+\frac{(ct-|x|)^2}{1+t}\bigg)^{-3}\int_{t-\frac{ct-|x|}{4c}}^t(1+t-s)^{-2}(1+t-s)^{\frac{5}{2}}ds\nm\\
&\leq& C(1+t)^{-(\frac{7}{2}-\varepsilon)}\bigg(1+\frac{(ct-|x|)^2}{1+t}\bigg)^{-3}(1+t)^{\frac{1}{2}}\frac{ct-|x|}{4c}\nm\\
&\leq& C(1+t)^{-(\frac{7}{2}-\varepsilon)}\bigg(1+\frac{(ct-|x|)^2}{1+t}\bigg)^{-3}(1+t)^{\frac{1}{2}}\bigg(1+\frac{(ct-|x|)^2}{1+t}\bigg)^{\frac{1}{2}}
(1+t)^{\frac{1}{2}}\nm\\
&\leq& C(1+t)^{-(\frac{5}{2}-\varepsilon)}\bigg(1+\frac{(ct-|x|)^2}{1+t}\bigg)^{-\frac{5}{2}}.
\ees\\
When $|x-y|<\frac{ct-|x|}{2}$, we have $|y|\geq \frac{|x|-ct}{2}$, which yields that
\bes
&&\int_{t-\frac{ct-|x|}{4c}}^t\int_{\mathbb{R}^3}(1+t-s)^{-2}e^{-\frac{(|x-y|-c(t-s))^2}{C(t-s)}}(1+s)^{-(\frac{7}{2}-\varepsilon)}\bigg(1+\frac{|y|^2}{1+s}\bigg)^{-(2-\varepsilon)}dyds\nm\\
&\leq& C(1+t)^{-(\frac{7}{2}-\varepsilon)}\bigg(1+\frac{(ct-|x|)^2}{1+t}\bigg)^{-(2-\varepsilon)}\int_{t-\frac{ct-|x|}{4c}}^t(1+t-s)^{-2}(1+t-s)^{\frac{5}{2}}ds\nm\\
&\leq& C(1+t)^{-(\frac{7}{2}-\varepsilon)}\bigg(1+\frac{(ct-|x|)^2}{1+t}\bigg)^{-(2-\varepsilon)}(1+t)\bigg(1+\frac{(|x|-ct)^2}{1+t}\bigg)^{\frac{1}{2}}\nm\\
&\leq& C(1+t)^{-(\frac{5}{2}-\varepsilon)}\bigg(1+\frac{(ct-|x|)^2}{1+t}\bigg)^{-(\frac{3}{2}-\varepsilon)}.
\ees
Hence, we complete the proof of this lemma.
\end{proof}

Next, we consider the convolution between Riesz wave and the worst nonlinear term $n_2\nabla\phi$. We have the following lemma.

\begin{lemma}\lbl{l 55} We have
\bes
\mathcal{N}_2:&=&\Big|\int_\frac{t}{2}^t\!\int_{\mathbb{R}^3}(1+t-s)^{-\frac{3}{2}}\bigg(1+\frac{|x-y|^2}{1+t-s}\bigg)^{-\frac{3}{2}}(n_2\nabla\phi)(y,s)dyds\Big|,\nm\\
&&+\Big|\int_0^\frac{t}{2}\!\int_{\mathbb{R}^3}(1+t-s)^{-2}\bigg(1+\frac{|x-y|^2}{1+t-s}\bigg)^{-2}(|\nabla\phi|^2)(y,s)dyds\Big|\nm\\
&\leq& CM^2(T)(1+t)^{-2}\bigg(1+\frac{|x|^2}{1+t}\bigg)^{-(2-\varepsilon)}.
\ees
\end{lemma}
\begin{proof} We divide the domain into several parts.\\
Case 1: $0\leq s\leq \frac{t}{2}$. When $|x|^2\leq 1+t$, we have
\bes
 &&\Big|\int_0^\frac{t}{2}\!\int_{\mathbb{R}^3}(1+t-s)^{-2}\bigg(1+\frac{|x-y|^2}{1+t-s}\bigg)^{-2}(1+s)^{-(3-\varepsilon)}\bigg(1+\frac{|y|^2}{1+s}\bigg)^{-(2-\varepsilon)}dyds\Big|\nm\\
&\leq& C(1+t)^{-2} \int_0^\frac{t}{2}(1+s)^{-(3-\varepsilon)}(1+s)^{\frac{3}{2}}ds\nm\\
&\leq& C(1+t)^{-2}\leq C(1+t)^{-2}\bigg(1+\frac{|x|^2}{1+t}\bigg)^{-2}.
\ees
When $|x|^2>1+t$ and $|y|\geq\frac{|x|}{2}$, we have
\bes
 &&\Big|\int_0^\frac{t}{2}\!\int_{\mathbb{R}^3}(1+t-s)^{-2}\bigg(1+\frac{|x-y|^2}{1+t-s}\bigg)^{-2}(1+s)^{-(3-\varepsilon)}\bigg(1+\frac{|y|^2}{1+s}\bigg)^{-(2-\varepsilon)}dyds\Big|\nm\\
&\leq& C(1+t)^{-2}\bigg(1+\frac{|x|^2}{1+t}\bigg)^{-(2-\varepsilon)}(1+t)^{-(2-\varepsilon)}\int_0^\frac{t}{2}(1+s)^{-1}(1+t-s)^{\frac{3}{2}}ds\nm\\
&\leq& C(1+t)^{-2}\bigg(1+\frac{|x|^2}{1+t}\bigg)^{-(2-\varepsilon)}.
\ees
When $|x|^2>1+t$ and $|y|<\frac{|x|}{2}$, we have $|x-y|\geq\frac{|x|}{2}$ and
\bes
&&\Big|\int_0^\frac{t}{2}\!\int_{\mathbb{R}^3}(1+t-s)^{-2}\bigg(1+\frac{|x-y|^2}{1+t-s}\bigg)^{-2}(1+s)^{-(3-\varepsilon)}\bigg(1+\frac{|y|^2}{1+s}\bigg)^{-(2-\varepsilon)}dyds\Big|\nm\\
&\leq& C(1+t)^{-2}\bigg(1+\frac{|x|^2}{1+t}\bigg)^{-2}\int_0^\frac{t}{2}(1+s)^{-(3-\varepsilon)}(1+s)^{\frac{3}{2}}ds\nm\\
&\leq& C(1+t)^{-2}\bigg(1+\frac{|x|^2}{1+t}\bigg)^{-2}.
\ees
Case 2: $\frac{t}{2}<s\leq t$. When $|x|^2\leq 1+t$, we have
\bes
&&\Big|\int_0^\frac{t}{2}\!\int_{\mathbb{R}^3}(1+t-s)^{-2}\bigg(1+\frac{|x-y|^2}{1+t-s}\bigg)^{-2}(1+s)^{-(3-\varepsilon)}\bigg(1+\frac{|y|^2}{1+s}\bigg)^{-(2-\varepsilon)}dyds\Big|\nm\\
&\leq& C(1+t)^{-(3-\varepsilon)}\int_{\frac{t}{2}}^t(1+t-s)^{-2}(1+t-s)^{\frac{3}{2}}ds\nm\\
&\leq& C(1+t)^{-(\frac{5}{2}-\varepsilon)}\leq C(1+t)^{-(\frac{5}{2}-\varepsilon)}\bigg(1+\frac{|x|^2}{1+t}\bigg)^{-2}.
\ees
When $|x|^2>1+t$ and $|y|\geq\frac{|x|}{2}$, we have
\bes
 &&\Big|\int_{\frac{t}{2}}^t\!\int_{\mathbb{R}^3}(1+t-s)^{-2}\bigg(1+\frac{|x-y|^2}{1+t-s}\bigg)^{-2}(1+s)^{-(3-\varepsilon)}\bigg(1+\frac{|y|^2}{1+s}\bigg)^{-(2-\varepsilon)}dyds\Big|\nm\\
&\leq& C(1+t)^{-(3-\varepsilon)}\bigg(1+\frac{|x|^2}{1+t}\bigg)^{-(2-\varepsilon)}\int_{\frac{t}{2}}^t(1+t-s)^{-2}(1+t-s)^{\frac{3}{2}}ds\nm\\
&\leq& C(1+t)^{-(\frac{5}{2}-\varepsilon)}\bigg(1+\frac{|x|^2}{1+t}\bigg)^{-(2-\varepsilon)}.
\ees
When $|x|^2>1+t$ and $|y|<\frac{|x|}{2}$, we have $|x-y|\geq\frac{|x|}{2}$ and
\bes
&&\Big|\int_{\frac{t}{2}}^t\!\int_{\mathbb{R}^3}(1+t-s)^{-2}\bigg(1+\frac{|x-y|^2}{1+t-s}\bigg)^{-2}(1+s)^{-(3-\varepsilon)}\bigg(1+\frac{|y|^2}{1+s}\bigg)^{-(2-\varepsilon)}dyds\Big|\nm\\
&\leq& C(1+t)^{-2}\bigg(1+\frac{|x|^2}{1+t}\bigg)^{-2}\int_{\frac{t}{2}}^t(1+s)^{-(3-\varepsilon)}(1+s)^{\frac{3}{2}}ds\nm\\
&\leq& C(1+t)^{-2}\bigg(1+\frac{|x|^2}{1+t}\bigg)^{-2}.
\ees
Thus, we complete the proof of this lemma.
\end{proof}

As in Lemma \ref{l 54} and \ref{l 55}, one can deal with the convolutions $\int_0^t(G_{12}-G_{S_1}-G_{S_2})(\cdot,t-s)\ast \widetilde{F_1}(\cdot,s)ds$ and $\int_0^t(G_{22}-G_{S_1}-G_{S_2})(\cdot,t-s)\ast \widetilde{F_1}(\cdot,s)ds$, where $\widetilde{F_1}$ is the other terms in $F_1$ except the worst nonlinear term $n_2\nabla\phi$ by using Lemma \ref{A4}-\ref{A6}.

On the other hand, we consider the convolution between the singular part in short wave $G_{S_1}$ and $F_1(n_1,w_1,n_2,w_2)(x,t)$. Since $G_{S_1}$ is singular when $t\rightarrow0$, this singularity no longer exists when taking the convolution. Then these terms can be estimated as those for the convolution between $G-G_{S_1}-G_{S_2}$ and the nonlinear terms above.

Next, for the convolution between the singular part $G_{S_2}$ and $F_1(n_1,w_1,n_2,w_2)(x,t)$, we have
\bes
&&\Big|\int_0^tG_{S_2}(\cdot,t-s)\ast F_1(n_1,w_1,n_2,w_2)(\cdot,s)ds\Big|\nm\\
&=&\Big|\int_0^t e^{-C_1(t-s)}(\delta(\cdot)+\tilde{f}(\cdot))\ast F_1(n_1,w_1,n_2,w_2)(\cdot,s)ds\Big|\nm\\
&\leq& CM^2(T)\int_0^te^{-C_1(t-s)}(\delta(\cdot)+\tilde{f}(\cdot))\ast\{(1+t)^{-3}(\psi_1+\psi_2+\psi_3+\psi_4)
+\psi_1\psi_4+\psi_2\psi_3\}(\cdot,s)ds\nm\\
&\leq& CM^2(T)\{(1+t)^{-3}(\psi_1+\psi_2+\psi_3+\psi_4)
+\psi_1\psi_4+\psi_2\psi_3+\cdots\}\nm\\
&\leq& CM^2(T)(1+t)^{-3}\left\{\bigg(1+\frac{|x|^2}{1+t}\bigg)^{-(\frac{3}{2}-\varepsilon)}+\bigg(1+\frac{(|x|-ct)^2}{1+t}\bigg)^{-(\frac{3}{2}-\varepsilon)}\right\},
\ees
where we have used the ansatz (\ref{5.20}), Lemma \ref{l 52(0)} and the fact
\bes
F_1&=&\mathcal{O}(1)\{D(w_1^2+w_2^2+n_1^2+n_2^2+n_2w_1^2+n_2w_1w_2+n_2w_2^2)+n_2\nabla\phi\nm\\
&&+D^2(n_1w_1+n_2w_2+n_1n_2w_1+n_1n_2w_2+w_1n_2^2+n_2^2w_2)\}.
\lbl{5.35}
\ees

Up to now, from Lemma \ref{l 54}, \ref{l 55}, \ref{A4}, \ref{A5}, \ref{A6}, and Proposition \ref{l 52}, we can immediately obtain
\bes
|n_1(x,t)|&\leq& C(\varepsilon_0+M^2(T))\psi_1,\nm\\
|w_1(x,t)|&\leq&  C(\varepsilon_0+M^2(T))\psi_2.
\lbl{5.36}
\ees

For the higher order derivative of $D_x^\alpha(n_1,w_1,n_2,w_2)$ with $|\alpha|=1,2$, one can easily obtain in the same way. In fact, for the short time part, one can put derivatives of any order on the Green's function to get the same estimate as $|\alpha|=0$. For the long time part, the unique difference from the case $|\alpha|=0$ is in the convolution between $G_{S_1}$ (or $\mathbb{G}_{S_1}$) due to the singularity as $t\rightarrow0$, which only allows us to put first order derivative on the Green's function. Thus, when $|\alpha|=2$, taking the nonlinear term $D^2(n_1w_1)$ in $F_1$ for example, the term $D^3(n_1w_1)=D^3n_1+D_2n_1Dw_1+Dn_1D^2w_1+n_1D^3w_1$ can easily be estimated from the ansatz (\ref{5.20}), Lemma \ref{A4}-\ref{A6}. The convolution between $D_x^\alpha (G(x,t)-G_{S_1}-G_{S_2})$ (or $D_x^\alpha(\mathbb{G}(x,t)-\mathbb{G}_{S_1}-\mathbb{G}_{S_2})$) and the nonlinear terms $F_1$ and $F_2$ can be treated similarly.

Lastly, for the convolution between $G_{S_2}$ (or $\mathbb{G}_{S_2}$) and the nonlinear terms $D_x^\alpha F_1$ and $D_x^\alpha F_2$ with $|\alpha|=1,2$, we have to face the third-order and the fourth-order derivatives of the solution $D^3(n_1,w_1,n_2,w_2)$ and $D^4(n_1,w_1,n_2,w_2)$, which is lack of the pointwise estimates for these terms due to the ansatz (\ref{5.20}). Hence, we should estimate as follows. In fact,
\bes
&&D^2F_1=\mathcal{O}(1)\{D^3w_1w_1\!+\!D^3w_2w_2\!+\!D^3n_1n_1\!+\!D^3n_2n_2\!+\!D^4n_1w_1\!+\!D^4w_1n_1\!+\!D^4n_2w_2\!+\!D^4w_2n_2\!+\!\cdots\},\nm\\
&&D^2F_2=\mathcal{O}(1)\{D^3w_1w_2\!+\!D^3w_2w_1\!+\!D^3n_1n_2\!+\!D^3n_2n_1\!+\!D^4n_1w_2\!+\!D^4w_2n_1\!+\!D^4n_2w_1\!+\!D^4w_1n_2\!+\!\cdots\},\ \ \ \ \ \ \ \ \
\lbl{5.36(00)}
\ees
where ``$\cdots$" denotes the rest terms. We take the term above $D^3w_1w_1$ for example and have
\bes
|D^3w_1w_1|&\leq& C(1+t)^{-3}\left\{(1+t)^{-\frac{3}{2}}\bigg(1+\frac{|x|^2}{1+t}\bigg)^{-(\frac{3}{2}-\varepsilon)}+(1+t)^{-2}\bigg(1+\frac{(|x|-ct)^2}{1+t}\bigg)^{-(\frac{3}{2}-\varepsilon)}\right\}\nm\\
&\leq& C(1+t)^{-\frac{9}{2}}\bigg(1+\frac{|x|^2}{1+t}\bigg)^{-(\frac{3}{2}-\varepsilon)}+C(1+t)^{-5}\bigg(1+\frac{|x|^2}{1+t}\bigg)^{-(\frac{3}{2}-\varepsilon)}(1+t)^{3-2\varepsilon}\nm\\
&\leq& C(1+t)^{-2}\bigg(1+\frac{|x|^2}{1+t}\bigg)^{-(\frac{3}{2}-\varepsilon)},
\lbl{5.36(0)}
\ees
where we have used the inequality (\ref{5.26(2)}) again. Then, from Lemma \ref{l 52(0)}, (\ref{5.36(00)}) and (\ref{5.36(0)}), we know when $1\leq|\alpha|\leq 2$ that
\bes
&&\int_0^t G_{S_2}(\cdot,t-s)\ast D_x^\alpha(F_1,F_2)(\cdot,s)ds+\int_0^t \mathbb{G}_{S_2}(\cdot,t-s)\ast D_x^\alpha(F_1,F_2)(\cdot,s)ds\nm\\
&\leq& C(1+t)^{-2}\bigg(1+\frac{|x|^2}{1+t}\bigg)^{-(\frac{3}{2}-\varepsilon)}.
\ees

As a result, we have for $1\leq|\alpha|\leq 2$ that
\bes
|D_x^\alpha n_1(x,t)|&\leq& C(\varepsilon_0+M^2(T))\psi_1,\nm\\
|D_x^\alpha w_1(x,t)|&\leq&  C(\varepsilon_0+M^2(T))\psi_2,\nm\\
|D_x^\alpha n_2(x,t)|&\leq& C(\varepsilon_0+M^2(T))\psi_5,\nm\\
|D_x^\alpha w_2(x,t)|&\leq& C(\varepsilon_0+M^2(T))\psi_6.
\lbl{5.37}
\ees

In summary, from (\ref{5.27(2)}), (\ref{5.36}) and (\ref{5.37}), we can conclude that
$$
M(T)\leq C(\varepsilon_0+M^2(T)),
$$
which together with the smallness of $\varepsilon_0$ and the continuity of $M(T)$ implies
$$
M(T)\leq C\varepsilon_0.
$$
This closes the ansatz (\ref{5.20}) and proves Theorem \ref{1 b}.

\section {Appendix}

\ \
\begin{lemma}\lbl{A1} \cite{Wang2}\ \ (1) When $\tau\in[0,t]$ and
$a^2\geq
 1+t$, for any positive constant $l$, we have
 \bess
 \left(1+\frac{a^2}{1+\tau}\right)^{-l}\leq
 3^l\left(\frac{1+\tau}{1+t}\right)^l\left(1+\frac{a^2}{1+t}\right)^{-l}.
 \eess

(2) When $a^2\leq 1+t$, we have
 \bess
 1\leq 2^l\left(1+\frac{a^2}{1+t}\right)^{-l}.
 \eess
 \end{lemma}

The following three lemmas are almost from \cite{LS,LW}. The unique difference in the present paper is the presence of the factor $\varepsilon$ in these lemmas. In fact, the factor $\varepsilon$ can be negligible in the proof. For simplicity, we omit the details. And  $t_0$ below is defined as in \cite{LS}
\bess
t_0=\max\bigg\{\frac{t}{2},t-\frac{\sqrt{1+t}}{4}\bigg\} \ {\rm in} \ D_1\cup D_2 \cup D_3;\ \ \
t_0=t-\min\bigg\{\frac{ct-|x|}{4c},\frac{|x|}{4c}\bigg\} \ {\rm in}\ D_4\cup D_5.
\eess

\begin{lemma}\lbl{A4}
For any $\gamma\geq0$, we have
\bess
\mathcal{N}_1^3&=&\int_0^{t_0}\int_{\mathbb{R}^3}(t-s)^{-\frac{4+\gamma}{2}}e^{-\frac{|x-y|^2}{C(t-s)}}(1+s)^{-3}\left(1+\frac{|y|^2}{1+s}\right)^{-(3-\varepsilon)}dyds\nm\\
&\leq & C(1+t)^{-\min\{\frac{4+\gamma}{2},\frac{9+\gamma}{4}\}}\left(1+\frac{x^2}{1+t}\right)^{-(\frac{3}{2}-\varepsilon)},\lbl{6.20}
\eess
\bess
\mathcal{N}_2^3&=&\int_0^{t_0}\int_{\mathbb{R}^3}(t-s)^{-\frac{4+\gamma}{2}}(t-s)^{-\frac{1}{2}}e^{-\frac{(|x-y|-c(t-s))^2}{C(t-s)}}(1+s)^{-4}\left(1+\frac{(|y|-cs)^2}{1+s}\right)^{-(3-2\varepsilon)}dyds\nm\\
&\leq & C(1+t)^{-\min\{\frac{4+\gamma}{2},3\}}\left(\left(1+\frac{x^2}{1+t}\right)^{-(\frac{3}{2}-\varepsilon)}+\left(1+\frac{(|x|-ct)^2}{1+t}\right)^{-(\frac{3}{2}-\varepsilon)}\right),
\lbl{6.2}
\eess
\bess
\mathcal{N}_3^3&=&\int_0^{t_0}\int_{\mathbb{R}^3}(t-s)^{-\frac{4+\gamma}{2}}(t-s)^{-\frac{1}{2}}e^{-\frac{(|x-y|-c(t-s))^2}{C(t-s)}}(1+s)^{-3}\left(1+\frac{|y|^2}{1+s}\right)^{-(3-2\varepsilon)}dyds\nm\\
&\leq & C(1+t)^{-\frac{4+\gamma}{2}}\left(\left(1+\frac{x^2}{1+t}\right)^{-(\frac{3}{2}-\varepsilon)}+\left(1+\frac{(|x|-ct)^2}{1+t}\right)^{-(\frac{3}{2}-\varepsilon)}\right),
\lbl{6.3}
\eess
\bess
\mathcal{N}_4^3&=&\int_0^{t_0}\int_{\mathbb{R}^3}(t-s)^{-\frac{4+\gamma}{2}}e^{-\frac{|x-y|^2}{C(t-s)}}(1+s)^{-4}\left(1+\frac{(|y|-cs)^2}{1+s}\right)^{-(3-2\varepsilon)}dyds\\[1mm]
&\leq&C(1+t)^{-\frac{4+\gamma}{2}}\left(\left(1+\frac{x^2}{1+t}\right)^{-(\frac{3}{2}-\varepsilon)}+\left(1+\frac{(|x|-ct)^2}{1+t}\right)^{-(\frac{3}{2}-\varepsilon)}\right).
\lbl{6.4}
\eess
\end{lemma}

\begin{lemma}\lbl{A5}
For any $\gamma\geq0$, we have
\bess
&\mathcal{N}_5^3=&\int_{t_0}^t\int_{\mathbb{R}^3}(t-s)^{-2}e^{-\frac{|x-y|^2}{C(t-s)}}(1+s)^{-\frac{3+\gamma}{2}}\left(1+\frac{|y|^2}{1+s}\right)^{-(\frac{3}{2}-\varepsilon)}dyds\nm\\
&\leq & C(1+t)^{-\frac{2+\gamma}{2}}\left(1+\frac{x^2}{1+t}\right)^{-(\frac{3}{2}-\varepsilon)},
\lbl{6.5}
\eess
\bess
\mathcal{N}_6^3&=&\int_{t_0}^t\int_{\mathbb{R}^3}(t-s)^{-2}(t-s)^{-\frac{1}{2}}e^{-\frac{(|x-y|-c(t-s))^2}{C(t-s)}}(1+s)^{-\frac{4+\gamma}{2}}\left(1+\frac{(|y|-cs)^2}{1+s}\right)^{-(\frac{3}{2}-\varepsilon)}dyds\nm\\
&\leq & C(1+t)^{-\frac{2+\gamma}{2}}\left(\left(1+\frac{x^2}{1+t}\right)^{-(\frac{3}{2}-\varepsilon)}+\left(1+\frac{(|x|-ct)^2}{1+t}\right)^{-(\frac{3}{2}-\varepsilon)}\right),
\lbl{6.6}
\eess
\bess
\mathcal{N}_7^3&=&\int_{t_0}^t\int_{\mathbb{R}^3}(t-s)^{-2}(t-s)^{-\frac{1}{2}}e^{-\frac{(|x-y|-c(t-s))^2}{C(t-s)}}(1+s)^{-\frac{3+\gamma}{2}}
\left(1+\frac{|y|^2}{1+s}\right)^{-(\frac{3}{2}-\varepsilon)}dyds\nm\\
&\leq & C\left((1+t)^{-\frac{1+\gamma}{2}}\left(1+\frac{x^2}{1+t}\right)^{-(\frac{3}{2}-\varepsilon)}+(1+t)^{-\frac{2+\gamma}{2}}\left(1+\frac{(|x|-ct)^2}{1+t}\right)^{-(\frac{3}{2}-\varepsilon)}\right),
\lbl{6.7}
\eess
\bess
\mathcal{N}_8^3&=&\int_{t_0}^t\int_{\mathbb{R}^3}(t-s)^{-2}e^{-\frac{|x-y|^2}{C(t-s)}}(1+s)^{-\frac{4+\gamma}{2}}\left(1+\frac{(|y|-cs)^2}{1+s}\right)^{-(\frac{3}{2}-\varepsilon)}dyds\\[1mm]
&\leq&C(1+t)^{-\frac{2+\gamma}{2}}\left(\left(1+\frac{x^2}{1+t}\right)^{-(\frac{3}{2}-\varepsilon)}+\left(1+\frac{(|x|-ct)^2}{1+t}\right)^{-(\frac{3}{2}-\varepsilon)}\right).
\lbl{6.8}
\eess
\end{lemma}

\begin{lemma}\lbl{A6} For any $\gamma\geq0$, we have
\bess
\mathcal{N}_9^3&=&\int_0^{t_0}\int_{\mathbb{R}^3}(t-s)^{-\frac{4+\gamma}{2}}\chi_{|x-y|\le c(t-s)}\left(1+\frac{|x-y|^2}{t-s}\right)^{-2} (1+s)^{-3}\left(1+\frac{|y|^2}{1+s}\right)^{-(3-2\varepsilon)}dyds\nm\\
&\leq & C(1+t)^{-\frac{4+\gamma}{2}}\left(1+\frac{x^2}{1+t}\right)^{-(\frac{3}{2}-\varepsilon)},
\lbl{6.9}
\eess
\bess
\mathcal{N}_{10}^3&=&\int_{t_0}^t\int_{\mathbb{R}^3}(t-s)^{-2}\chi_{|x-y|\le c(t-s)}\left(1+\frac{|x-y|^2}{t-s}\right)^{-2} (1+s)^{-\frac{3+\gamma}{2}}\left(1+\frac{|y|^2}{1+s}\right)^{-(\frac{3}{2}-\varepsilon)}dyds\nm\\
&\leq & C(1+t)^{-\frac{2+\gamma}{2}}\left(1+\frac{x^2}{1+t}\right)^{-(\frac{3}{2}-\varepsilon)},\lbl{6.29}
\eess
\bess
\mathcal{N}_{11}^3&=&\int_0^{t_0}\int_{\mathbb{R}^3}(t-s)^{-\frac{4+\gamma}{2}}\chi_{|x-y|\le c(t-s)}\left(1+\frac{|x-y|^2}{t-s}\right)^{-2} (1+s)^{-4}\left(1+\frac{(|y|-cs)^2}{1+s}\right)^{-(3-2\varepsilon)}dyds\nm\\
&\leq & C(1+t)^{\frac{4+\gamma}{2}}\left(\left(1+\frac{x^2}{1+t}\right)^{-(\frac{3}{2}-\varepsilon)}+\left(1+\frac{(|x|-ct)^2}{1+t}\right)^{-(\frac{3}{2}-\varepsilon)}\right),
\lbl{6.10}
\eess
\bess
\mathcal{N}_{12}^3&=&\int_{t_0}^t\int_{\mathbb{R}^3}(t-s)^{-2}\chi_{|x-y|\le c(t-s)}\left(1+\frac{|x-y|^2}{t-s}\right)^{-2} (1+s)^{-\frac{4+\gamma}{2}}\left(1+\frac{(|y|-cs)^2}{1+s}\right)^{-(\frac{3}{2}-\varepsilon)}dyds\nm\\
&\leq & C(1+t)^{-\frac{2+\gamma}{2}}\left(\left(1+\frac{x^2}{1+t}\right)^{-(\frac{3}{2}-\varepsilon)}+\left(1+\frac{(|x|-ct)^2}{1+t}\right)^{-(\frac{3}{2}-\varepsilon)}\right).\lbl{6.31}
\eess
\end{lemma}

\section*{Acknowledgments}

The research of Z.G. Wu  was sponsored by Natural Science Foundation of Shanghai (No. 16ZR1402100) and the Fundamental Research Funds for the Central Universities (No. 2232015D3-33).
The research of W.K. Wang was supported by Natural Science Foundation of China
(No. 11231006).

\end{document}